\documentclass[11pt,a4paper]{article}
\usepackage{geometry}                % See geometry.pdf to learn the layout options. There are lots.
\usepackage{amsmath}
\usepackage{graphicx}
\usepackage{url}
\usepackage{a4wide}
\usepackage{color}
\usepackage{amssymb}
\usepackage{epstopdf}
\usepackage{color}
\newcommand{\subsethook}{\mbox{$\,\subset \!\!\!\mbox{\lower
0.58ex\hbox{$\rightarrow$}}\,$}} %Also \hookrightarrow

\newtheorem{theorem}            {Theorem}
\newtheorem{lemma}[theorem]{Lemma}
\newtheorem{definition}[theorem]{Definition}
\newtheorem{corollary} [theorem]{Corollary}
\newtheorem{assumption}[theorem]{Assumption}
\newtheorem{example}[theorem]{Example}
\newtheorem{remark}[theorem] {Remark}

\newcommand {\proof} {\par{\it Proof}. \ignorespaces}
\newcommand {\eproof}
      {\space
        {\ \vbox{\hrule\hbox{\vrule height1.3ex\hskip0.8ex\vrule}\hrule}}
        \par}
\title{Abstract dissipative Hamiltonian differential-algebraic equations are everywhere}
\author{Volker Mehrmann\footnotemark[1] \,  and  Hans Zwart \footnotemark[2]
}

\begin{document}
\maketitle

\begin{abstract} In this paper we study the representation of partial differential equations (PDEs) as abstract differential-algebraic equations (DAEs) with dissipative Hamiltonian structure (adHDAEs). We show that these systems not only arise when there are constraints coming from the underlying physics, but many standard PDE models can be seen as an adHDAE  on an extended state space. This reflects the fact that models often include closure relations and structural properties. We present a unifying operator theoretic approach to analyze the properties of such operator equations and illustrate this by several applications.
%including infinite dimensional port-Hamiltonian differential-algebraic systems (pHDAEs). }
\end{abstract}

%\date{\today}

\renewcommand{\thefootnote}{\fnsymbol{footnote}}

\footnotetext[1]{
Institut f\"ur Mathematik MA 4-5, TU Berlin, Str.\ des 17.\ Juni 136,
D-10623 Berlin, FRG.
\texttt{mehrmann@math.tu-berlin.de}. Supported by {\it by Deutsche Forschungsgemeinschaft (DFG, German Research Foundation) Excellence Cluster 2046 Math${}^+$: Project No.~390685689.}}
\footnotetext [2] {University of Twente, Department of Applied Mathematics, P.O.\ Box 217, 7500 AE Enschede, The Netherlands, and  Department of Mechanical Engineering, Eindhoven University of Technology, P.O.\ Box 513, 5600 MB Eindhoven, The Netherlands.
\texttt{h.j.zwart@utwente.nl}}

\section{Introduction}
%The class of port-Hamiltonian (pH) systems is a very %important model class that allows to model many classes of %physical systems with interact with their environment. %Examples include ordinary and partial differential equations, %see e.g.\ \cite{SchJ14}. By now the analytic theory and also %the numerical methods are well developed with a multitude of %results ranging from control, approximation, and well-%posedness of partial differential equations, \cite%{JacZ12,SchJ14}.
In this paper we study the mathematical modeling, the analytical theory and the representation of abstract  linear differential-algebraic equations (DAEs) of the form
\begin{equation}
\label{eq:adae}
  \mathcal E \dot{x}(t) = {\mathcal A}\mathcal Q x(t)
\end{equation}
on the infinite-dimensional Hilbert space   $\mathbb X$ with inner product $\langle\cdot, \cdot \rangle $.
We assume  that ${\mathcal A}: D({\mathcal A}) \subseteq {\mathbb X} { \rightarrow} {\mathbb X}$ is a \emph{dissipative linear operator}, i.e.,
$\langle \mathcal Ax, x \rangle+\langle x, \mathcal Ax \rangle\leq 0$ for all $x$ in the domain of ${\mathcal A}$.
%{\mathcal A}+ {\mathcal A}^*$ is negative semidefinite denoted as ${\mathcal A}+ {\mathcal A}^* \leq 0$.
The operators $\mathcal E:{\mathbb X}\to \mathbb X$ and $\mathcal Q:{\mathbb X} \to \mathbb X$
are assumed to be bounded linear operators that satisfy further geometric conditions and define an \emph{energy functional or Hamiltonian} via
\begin{equation}\label{energy}
 {\mathcal H}(x)
 %= \langle x(t), \tilde {\mathcal Q} x(t) \rangle
 := \langle \mathcal E x,  \mathcal Q x \rangle,
\end{equation}
which is assumed to be non-negative, i.e., $\mathcal H(x) \geq 0$ for all $x \in {\mathbb X}$.
%
%in the standard dual pairing in {\color{blue} $\mathbb X$}.

We call this class of problems \emph{abstract dissipative Hamiltonian DAEs (adHDAEs)}.

Abstract differential-algebraic systems do not only arise by including constraints coming from the underlying physical system, see e.g.\ \cite{ErbJMRT24,GerR23,LamMT13}, but many standard systems of partial differential equations (PDEs) can be viewed as  abstract differential-algebraic equation on an extended state-space. We present several applications where this is the case.

The class of adHDAEs is also  strongly motivated by modeling physical systems in the model class of (abstract) port-Hamiltonian differential-algebraic  systems (pHDAEs), a class which is of great relevance in many applications and  has recently seen a huge number of applications in almost all physical domains, see e.g.\  \cite{BarCGJR23,BeaMXZ18,BenMHL24,GerHR21,GerHRS21,GueBJR21,JacM22,JaeEGJ22,MehS23,Mor23,PhiRS23,Rei21,Sch13,SchJ14,SchM02}.
To illustrate the concept of adHDAEs, consider the following example.
\begin{example}\label{ex:vibstring} \cite{JacZ12} {\rm
The vibrating string in one space dimension can be modelled by the PDE
\begin{equation}
\label{wave}
  \rho \frac{\partial^2 w}{\partial t^2} = \frac{\partial }{\partial \zeta}\left(T \frac{\partial w}{\partial \zeta} \right),
\end{equation}
where $\rho$ is the mass density,  $w$ is the vertical  displacement, and $T$ is the Young modulus of the material, $\zeta$ is in a one-dimensional spatial domain, and $t$ the time.

The port-Hamiltonian modeling approach, see \cite{JacZ12},
introduces the extended state
\[
  z(t) = \begin{bmatrix} \rho \frac{\partial w}{\partial t} \\  \frac{\partial w}{\partial \zeta}\end{bmatrix}
\]
in the state space ${\mathbb X}=L^2(\Omega;{\mathbb R}^2)$, with $\Omega$ the spatial interval,
and leads to a representation of (\ref{wave}) given by
\begin{align*}
  \dot{z}(t) =&\ \underbrace{\begin{bmatrix} 0 & \frac{\partial}{\partial \zeta}\\  \frac{\partial }{\partial \zeta} & 0 \end{bmatrix}} \underbrace{\begin{bmatrix} \frac{1}{\rho} & 0 \\  0 & T \end{bmatrix}}z(t)\\
  := & \qquad {\mathcal A}\qquad  \quad {\tilde {\mathcal Q}}\quad z(t),
\end{align*}
with a Hamiltonian $\mathcal H(z(t))= \langle z(t), \tilde {\mathcal Q} z(t) \rangle $.
If the mass density $\rho$ is close to zero, then it is important to analyze what happens when one considers the density $\rho= 0$. For $\rho$ close to zero, it is more appropriate to consider a different state
%this is gives a problem in
%\[
%  {\mathcal H} = \begin{bmatrix} %\frac{1}{\rho} & 0 \\  0 & T %\end{bmatrix}.
%\]
%This can be solved this by changing %the state, i.e.,
\[
%  x(t) = \begin{bmatrix} \rho \frac{\partial w}{\partial t} \\  \frac{\partial w}{\partial \zeta}\end{bmatrix} \quad \mbox{replaced by }
x(t) = \begin{bmatrix}  \frac{\partial w}{\partial t} \\  \frac{\partial w}{\partial \zeta}\end{bmatrix},
\]
which leads to a representation
\begin{align*}
  \underbrace{\begin{bmatrix} \rho & 0 \\  0 & 1 \end{bmatrix}}\dot{x}(t) =&\ \underbrace{\begin{bmatrix} 0 & \frac{\partial}{\partial \zeta}\\  \frac{\partial }{\partial \zeta} & 0 \end{bmatrix}} \underbrace{\begin{bmatrix} 1 & 0 \\  0 & T \end{bmatrix}}x(t), \\% \Leftrightarrow\\
 \mathcal E \qquad \dot{x}(t) = & \qquad {\mathcal A}\qquad\quad  \mathcal Q\qquad x(t)
\end{align*}
where we have introduced the matrices $\mathcal E$ and $\mathcal Q$ and the differential operator ${\mathcal A}$.

Note that by this change of variables the value of the Hamiltonian $\mathcal H$ stays the same, i.e.,
\[
  {\mathcal H}(t) = \langle z(t), \tilde {\mathcal Q} z(t) \rangle = \langle \mathcal E x(t), \mathcal Q x(t) \rangle.
\]
%So we have the differential equation
%\begin{align*}
%  \underbrace{\begin{bmatrix} \rho & 0 \\  0 & 1 \end{bmatrix}}\dot{z}(t) =&\ \underbrace{\begin{bmatrix} 0 & \frac{\partial}{\partial \zeta}\\  \frac{\partial }{\partial \zeta} & 0 \end{bmatrix}} \underbrace{\begin{bmatrix} 1 & 0 \\  0 & T \end{bmatrix}}z(t) \Leftrightarrow\\
% E \qquad \dot{z}(t) = & \qquad {\mathcal J}\qquad\qquad  Q\quad z(t).
%\end{align*}
%
However,  in this formulation, we can set both $\rho=0$, and $T=0$ and then either $\mathcal E$ or $\mathcal Q$ or both become singular.

For invertible $\mathcal E$, we can express this system as a standard wave equation, of which it is known that it will generate a contraction semigroup, provided the appropriate boundary conditions are posed, \cite{JacZ12}.
}
\end{example}

\begin{remark}\label{rem:whyDAEs}{\rm
Example~\ref{ex:vibstring}  demonstrates that the use of differential-algebraic equations is essential when considering limiting situations, see also \cite{BenMHL24,SchM20,MehS23,SchM23} for detailed discussions. In many applications one can resolve the constraint equations and return to explicit formulations in the time derivative. But this is not always a good mathematical formulation for several reasons. First of all it may happen that the resulting system is much more sensitive under perturbations. But more important, by resolving the constraints, they are not visible in the equations any longer, even though they usually are of physical relevance. Furthermore,  they are then also not enforced during a numerical simulation of the system, see \cite{BreCP96,HaiW96,KunM06} and this can lead to a drift of the numerical solution from the constraint manifold. }%end rm
\end{remark}

Example~\ref{ex:vibstring} is a motivation to study the properties of adHDAEs of the form \eqref{eq:adae} in which both operators $\mathcal E$ and $\mathcal Q$ may be {\em singular}\/ matrices or non-bijective operators, and where ${\mathcal A}$ generates a contraction semigroup on the Hilbert space $\mathbb X$.
When modelling physical systems in a modular fashion then often not only $\mathcal E$ and $\mathcal Q$ may be singular but the equation \eqref{eq:adae} may be overdetermined or not uniquely solvable. For  general abstract DAEs this is hard to analyze but we present a simple characterization of singularity for \eqref{eq:adae} in Subsection~\ref{sec:regular}.

One may have the impression that the case that $\mathcal E$ and/or $\mathcal Q$ are singular  is a very special case that is not encountered often when modeling physical processes.  However, we will demonstrate that this is almost the standard case. To illustrate this, consider the following example.

\begin{example}\label{ex:heat-equation}
{\rm Consider the derivation of the diffusion/heat equation in a one-dimensional domain. The defining relation between the temperature $T$ and the heat flux $J$ is given by the PDE
\begin{equation}
\label{eq:heat1}
  \frac{ \partial T}{\partial t} = -\alpha \frac{ \partial J}{\partial \zeta},
\end{equation}
where  $\alpha >0$ is the diffusivity constant.  Using Fourier's  law to model the heat flux as proportional (with thermal conductivity $k$) to the spatial derivative of the temperature, i.e.,
\begin{equation}
\label{eq:heat2}
  J = -k \frac{ \partial T}{\partial \zeta}
\end{equation}
%
%Combining (\ref{eq:heat1}) and (\ref{eq:heat2})
gives the standard diffusion/heat equation
\[
   \frac{ \partial T}{\partial t} = k \alpha \frac{ \partial^2 T}{\partial \zeta^2}.
\]
However, we can also express the system as adHDAE system
\begin{align}
\label{eq:heat3}
  \underbrace{\begin{bmatrix}\alpha^{-1} & 0 \\  0 & 0 \end{bmatrix}}\frac{\partial}{\partial t} \begin{bmatrix} T\\ J \end{bmatrix} & = \underbrace{\begin{bmatrix} 0 & -\frac{\partial}{\partial \zeta}\\  -\frac{\partial }{\partial \zeta} & -1 \end{bmatrix}}\underbrace{\begin{bmatrix} 1 & 0 \\  0 & k^{-1} \end{bmatrix}}
 \begin{bmatrix} T\\ J \end{bmatrix} \\
 \nonumber
\mathcal E \qquad \quad\dot{x}(t) &=  \qquad \quad{\mathcal A}\qquad\quad  \mathcal Q\qquad x(t)
\end{align}
with ${\mathbb X}$ as in the previous example. Thus it is in the form \eqref{eq:adae}
with $\mathcal E$ being singular. Note that in this case the singularity of $\mathcal E$ is not caused by a physical parameter becoming zero, but it is a direct consequence of the closure relation (\ref{eq:heat2}). Since these closure relations appear almost everywhere in mathematical modelling, we see that a singular $\mathcal E$ is very common.
%Again, there is a natural Hamiltonian.
}
\end{example}

The discussed examples demonstrate that in modeling with adHDAEs different representations  are possible, and some are preferable to others, e.g. in the case of limiting situations.

The operator $\mathcal A$ in (\ref{eq:heat3}) is very similar to that in \eqref{wave}, and so one may think  that properties are related, but it is well-known that the wave and heat equation behave completely differently, {the first has oscillating solution behavior while the second is diffusive, so the solution decays.}
However, as we will demonstrate, the solution theory of the two PDEs is strongly related, see Example~\ref{E3:HZ} below.

The structure in (\ref{eq:adae}) is also motivated by the class of finite dimensional dissipative Hamiltonian descriptor systems introduced in  \cite{BeaMXZ18}, see also \cite{MehU23,Mor23} that have the form \eqref{eq:adae} with $\mathcal A=\mathcal J-\mathcal R$,
where $\mathcal J$ is (formally) skew-adjoint, and $\mathcal E^*\mathcal Q$ as well as $\mathcal R$ are self-adjoint and { nonnegative (positive semidefinite)}.

The paper is organized as follows. In  Section~\ref{sec:2} we introduce our basic set-up together with several assumptions.
%Furthermore, we study these assumptions in order to have more simple conditions to replace them.
In Section \ref{sec:3} we study the solution theory of adHDAEs of the form (\ref{eq:adae}). These results are illustrated in Section~\ref{sec:4} by several examples, showing their applicability. In these examples we also recover many results, which often were obtained by other methods. In Section \ref{sec:5} we treat the case in which the singularity of $\mathcal E$  restricts the state space, and again our result is illustrated by examples.

\section{Representation of adHDAEs}
\label{sec:2}
In this section we study adHDAEs of the form \eqref{eq:adae}
on an infinite-dimensional Hilbert space {$\mathbb X$}. In order to analyze the solution properties we make some general assumptions on the structure of $\mathcal A,\mathcal E$, and $\mathcal Q$.

Consider an abstract dissipative Hamiltonian differential-algebraic equation (adHDAE)
\begin{equation}
\label{eq:1HZ-new}
\mathcal   E_{ext} \dot{x}(t) = {\mathcal A}_{ext} \mathcal Q_{ext} x(t)
\end{equation}
of the form \eqref{eq:adae}
with the following structural properties.
\begin{assumption}
\label{A0:HZ}
\begin{itemize}
\item [i)]
  The state-space is a Hilbert space ${\mathbb X}_{ext}= { \mathbb X}_1 \oplus {\mathbb X}_2 \oplus { \mathbb X}_3$.
\item [ii)]
 The operator  ${\mathcal A}_{ext}=\begin{bmatrix} {\mathcal A}_{1,ext} \\  {\mathcal A}_{2,ext}\\  {\mathcal A}_{3,ext} \end{bmatrix}$ is a \emph{dissipative operator} on ${\mathbb X}_{ext}$, i.e., $\mathrm{Re}\langle {\mathcal A}_{ext}x, x\rangle \leq 0$ for all $x$ in the domain $D({\mathcal A}_{ext})$ of ${\mathcal A}_{ext}$.
\item [iii)]
  The operators $\mathcal E_{ext}$ and $\mathcal Q_{ext}$ are block-operators of the  form
  \begin{equation}
  \label{eq:2HZ}
    \mathcal E_{ext} = \begin{bmatrix} \mathcal E_1 & 0 & 0 \\ 0 & 0& 0 \\ 0 & 0 & \mathcal E_3\\  \end{bmatrix}, \qquad \mathcal Q_{ext} = \begin{bmatrix} \mathcal Q_1 & 0 & 0 \\ 0 & \mathcal Q_2 & 0\\ 0 & 0& 0 \end{bmatrix},
  \end{equation}
  where $\mathcal E_1,\mathcal E_3,\mathcal Q_1$ and $\mathcal Q_2$ are bounded and boundedly invertible. Furthermore, we assume that $\mathcal E_1^*\mathcal Q_1$ is \emph{coercive}, i.e., it is self-adjoint and (strictly) positive.
  \item [iv)]
    There exists an $s \in {\mathbb C}^+ := \{ s \in {\mathbb C} \mid \mathrm{Re}(s) >0\}$ such that the operator $s\mathcal E_{ext}- {\mathcal A}_{ext} \mathcal Q_{ext}$ with domain $\{ x \in {\mathbb X} \mid \mathcal Q_{ext} x \in D({\mathcal A}_{ext})\}$ is boundedly invertible.
\end{itemize}
\end{assumption}

\begin{remark}\label{rem:ass}{\rm
Assumption \ref{A0:HZ} seems to be very restrictive at first sight. However, as we will demonstrate, it holds for many examples
and it allows us to prove our main results.
However, this assumption can be relaxed in many particular cases by using different proof techniques.

{ See also
%Note also that condition iv)
%implies that the abstract DAE has chain-%index $0$ on the given subset, see %
\cite{ErbJMRT24} for the analysis of the chain-index under this assumption.}
}
\end{remark}

In the setting of finite dimensional DAEs, see \cite{MehMW18},  condition iv) in Assumption~\ref{A0:HZ} implies that the pair $(\mathcal E_{ext},{\mathcal A}_{ext}\mathcal Q_{ext})$ forms a \emph{regular pair}, see e.g.\ \cite{KunM06}.
We will use this terminology also in the infinite dimensional case when the operators satisfies Assumption~\ref{A0:HZ}. iv). A characterization when a pair is regular or singular is given in Subsection~\ref{sec:regular}.
\begin{remark}\label{rem:cs}{\rm
In the case that $\mathcal E_{ext},\mathcal Q_{ext}$ are matrices,  the condition that $\mathcal E_{ext}^*\mathcal Q_{ext}$ is self-adjoint means that
the columns of
\[
\begin{bmatrix}
    \mathcal E_{ext}\\
    \mathcal Q_{ext}
\end{bmatrix}
\]
span an isotropic subspace of ${\mathbb X} \times { \mathbb X}^*={ {\mathbb X} \times { \mathbb X}}$, see e.g. \cite{MehS23,SchM23}, which is a Lagrange subspace if the dimension is maximal, i.e., that of {$\mathbb X$}. This is the case if and only if the pair $(\mathcal E_{ext},\mathcal Q_{ext})$  is regular. For Lagrange subspaces the representation~\eqref{eq:2HZ} can always be achieved by a change of basis using a cosine-sine decomposition, see \cite{MehMW18,PaiW94}.
}
\end{remark}

\begin{remark}\label{rem:energy}
{\rm
From the modeling point of view systems of the form \eqref{eq:adae} lead to a natural definition of an \emph{energy functional (Hamiltonian)} %for \eqref{eq:adae} would be given by
\begin{equation}\label{defenergy}
  {\mathcal H(x)}:=\langle \mathcal E_{ext} x, \mathcal Q_{ext} x \rangle.
\end{equation}
However, the definition of the Hamiltonian is by no means unique, in particular the choice of variables in the kernels of {$\mathcal E_{ext}$ and $\mathcal Q_{ext}$} is arbitrary and thus there are many different representations of the state variables with the same Hamiltonian, see Example~\ref{ex:vibstring}. Under the conditions in Assumption~\ref{A0:HZ}, we  have that
\[
 {\mathcal H}(x_1)=\langle \mathcal E_1 x_1, \mathcal Q_1 x_1 \rangle= {\mathcal H}(x),
\]
i.e., the Hamiltonian may also be defined on a restricted subspace.

For a detailed discussion of this topic of different representations in the finite dimensional case, see
\cite{MehS23,SchM23}.
}
\end{remark}

Looking at a system (\ref{eq:1HZ-new}) that satisfies Assumption~\ref{A0:HZ}, we see that the third state, $x_3$, does not influence the first nor the second state. However, its behaviour is dictated by the other two. So we could regard $\dot{x}_3$ in (\ref{eq:1HZ-new}) as a kind of \emph{output to the system}. Since we are mainly interested in the dynamics of the first state, a natural question is if we can find a reduced representation of the system with similar properties by removing the third state. This topic has been discussed extensively in the case of finite dimensional
port-Hamiltonian DAEs, see \cite{BeaMXZ18,MehU23}. Since the conditions in Assumption~\ref{A0:HZ} include $\mathcal E_3$ and ${\mathcal A}_{3,ext}$, it is not clear a priori whether similar properties still hold without these assumptions.
Our first result shows that this is indeed the case.
\begin{theorem}\label{T:1}
Consider an adHDAE of the form~\eqref{eq:adae} that satisfies Assumption \ref{A0:HZ}. Introduce the operator
\begin{equation}
  \label{eq:def-A}
    {\mathcal A}\begin{bmatrix} x_1 \\ x_2\end{bmatrix} := \begin{bmatrix} {\mathcal A}_1 \\  {\mathcal A}_2\end{bmatrix} \begin{bmatrix} x_1 \\ x_2 \end{bmatrix} := \begin{bmatrix} {\mathcal A}_{1,ext} \\  {\mathcal A}_{2,ext} \end{bmatrix} \begin{bmatrix} x_1 \\ x_2 \\ 0\end{bmatrix} ,
\end{equation}
with domain
\begin{equation}
  \label{eq:dom-A}
    D({\mathcal A}) = \{ \begin{bmatrix} x_1 \\ x_2\end{bmatrix} \in {\mathbb X}_1 \oplus {\mathbb X}_2 \mid \begin{bmatrix} x_1\\ x_2 \\ 0 \end{bmatrix} \in D({\mathcal A}_{ext})\}.
  \end{equation}
  Then ${\mathcal A}$ is dissipative and  with
    \begin{equation}
  \label{eq:2aHZ}
    \mathcal E = \begin{bmatrix} \mathcal E_1 & 0  \\  0& 0 \end{bmatrix}, \qquad \mathcal Q = \begin{bmatrix} \mathcal Q_1 & 0 \\ 0&\mathcal Q_2 \end{bmatrix},
  \end{equation}
the pair $(\mathcal E,{\mathcal A}\mathcal Q)$ is regular.
\end{theorem}
\noindent\proof
For $\left[\begin{smallmatrix} x_1 \\ x_2\end{smallmatrix}\right] \in D({\mathcal A})$ we have
\[
  \left\langle \begin{bmatrix} x_1 \\ x_2\end{bmatrix}, {\mathcal A} \begin{bmatrix} x_1 \\ x_2\end{bmatrix} \right\rangle = \left\langle \begin{bmatrix} x_1 \\ x_2\\ 0\end{bmatrix}, {\mathcal A}_{ext} \begin{bmatrix} x_1 \\ x_2\\ 0\end{bmatrix} \right\rangle.
\]
Since by assumption the last expression has non-positive real part, we see that ${\mathcal A}$ is dissipative.

Let $s \in {\mathbb C}^+$ be as in  Assumption \ref{A0:HZ} iv). We will show that the operator $s\mathcal E-{\mathcal A} \mathcal  Q$, with domain $\{ x \in {\mathbb X}_1 \oplus {\mathbb X}_2 \mid \mathcal Q x \in D({\mathcal A})\}$, is boundedly invertible.

Let  $x=\left[\begin{smallmatrix} x_1 \\ x_2\end{smallmatrix} \right] \in  {\mathbb X}_1 \oplus {\mathbb X}_2$ be such that $\mathcal Qx \in D({\mathcal A})$ and $(s\mathcal E - {\mathcal A}\mathcal Q)x=0$.
Define $x_3 = \frac{1}{s} \mathcal E_3^{-1}{\mathcal A}_{3,{ext}} \left[\begin{smallmatrix}\mathcal Q_1 x_1\\ \mathcal Q_2 x_2 \\ 0\end{smallmatrix}\right]$.
With this choice, then
\[
  (s\mathcal E_{ext} - {\mathcal A}_{ext}{ \mathcal Q_{ext}})\begin{bmatrix} x_1 \\ x_2 \\ x_3\end{bmatrix} =0.
\]
Since the pair $(\mathcal E_{ext},{\mathcal A}_{ext}{ \mathcal Q_{ext})}$ is regular, this implies, in particular, that $x_1=0$ and $x_2=0$. Thus $(s \mathcal E - {\mathcal A}\mathcal Q)$ is injective.

For $\left[\begin{smallmatrix} y_1 \\ y_2\end{smallmatrix}\right] \in {\mathbb X}_1 \oplus {\mathbb X}_2$, define
\begin{equation}
\label{eq:2bHZ}
  \begin{bmatrix} \tilde{x}_1 \\ \tilde{x}_2 \\ \tilde{x}_3\end{bmatrix}
= (s\mathcal E_{ext} - {\mathcal A}_{ext}\mathcal Q_{ext})^{-1} \begin{bmatrix} y_1  \\ y_2\\ 0\end{bmatrix}.
\end{equation}
Then
\begin{equation}
\label{eq:2cHZ}
  \begin{bmatrix} x_1 \\ x_2 \\ 0 \end{bmatrix} := \begin{bmatrix}\mathcal Q_1 \tilde{x}_1 \\ \mathcal Q_2 \tilde{x}_2 \\ 0 \end{bmatrix} = \mathcal Q_{ext} \begin{bmatrix} \tilde{x}_1 \\ \tilde{x}_2 \\ \tilde{x}_3\end{bmatrix}
\end{equation}
is an element of $D({\mathcal A}_{ext})$, and
\[
  \begin{bmatrix} (s\mathcal E-{\mathcal A} \mathcal Q)\begin{bmatrix} \tilde{x}_1 \\ \tilde{x}_2 \end{bmatrix} \\ s \mathcal E_3 \tilde{x}_3 \end{bmatrix} = (s\mathcal E_{ext} - {\mathcal A}_{ext}\mathcal Q_{ext}) \begin{bmatrix} \tilde{x}_1 \\ \tilde{x}_2 \\ \tilde{x}_3\end{bmatrix} +\begin{bmatrix} 0 \\ 0 \\ y_3 \end{bmatrix} =  \begin{bmatrix} y_1 \\ y_2 \\ y_3 \end{bmatrix},
\]
where $y_3 = {\mathcal A}_{3,ext}\left[\begin{smallmatrix} x_1 \\ x_2 \\ 0 \end{smallmatrix} \right]$.
In particular,
\[
  (s\mathcal E-{\mathcal A}\mathcal Q)\begin{bmatrix} \tilde{x}_1 \\ \tilde{x}_2 \end{bmatrix} = \begin{bmatrix} y_1 \\ y_2 \end{bmatrix}
\]
and so $(s\mathcal E-{\mathcal A} \mathcal Q)$ is surjective. Combined with its injectivity and equations (\ref{eq:2bHZ})--(\ref{eq:2cHZ}), we see that $(s\mathcal E-{\mathcal A} \mathcal Q)$ is boundedly invertible.
\hfill\eproof
\medskip

From Theorem~\ref{T:1} we see that, if Assumption~\ref{A0:HZ} holds for the adHDAE (\ref{eq:1HZ-new}), then for the \emph{reduced adHDAE}
\begin{equation}
\label{eq:1HZ}
  \mathcal E \dot{x}(t) = {\mathcal A}\mathcal Q x(t)
\end{equation}
with ${\mathcal A}$, $\mathcal E$, and $\mathcal Q$ defined in (\ref{eq:def-A})--(\ref{eq:2aHZ}), the following conditions are satisfied.
\begin{assumption}
\label{A1:HZ}~
\begin{itemize}
\item [i)]
  The state space is the Hilbert space ${\mathbb X}= {\mathbb X}_1 \oplus {\mathbb X}_2$.
\item [ii)]
  ${\mathcal A}=\begin{bmatrix} {\mathcal A}_1 \\ {\mathcal A}_2 \end{bmatrix}$ is dissipative on {$\mathbb X$}.
\item [iii)]
  The operators $\mathcal E$ and $\mathcal Q$ are of the form
\begin{equation}
  \label{eq:2HZ-bis}
   \mathcal E = \begin{bmatrix} \mathcal E_1 & 0 \\ 0& 0 \end{bmatrix}, \qquad\mathcal Q = \begin{bmatrix}\mathcal Q_1 & 0 \\ 0&\mathcal Q_2 \end{bmatrix},
\end{equation}
where $\mathcal E_1,\mathcal Q_1$ and $\mathcal Q_2$ are bounded and boundedly invertible operators. Furthermore, $\mathcal E_1^*\mathcal Q_1$ is coercive, i.e., it is self-adjoint and { $\langle \mathcal E_1 x, \mathcal Q_1 x \rangle > \kappa \|x\|^2 >0$ for all nonzero $x$}.
  \item [iv)]
    There exists an $s \in {\mathbb C}^+ := \{ s \in {\mathbb C} \mid \mathrm{Re}(s) >0\}$ such that $s\mathcal E - {\mathcal A}\mathcal Q$ is boundedly invertible.
\end{itemize}
\end{assumption}

Theorem~\ref{T:1} shows that in an adHDAE system \eqref{eq:adae} that satisfies Assumption~\ref{A0:HZ} there exists a reduced subsystem for which Assumption~\ref{A1:HZ} holds.
In our next result we analyze the relation between the two sets of Assumptions \ref{A0:HZ} and \ref{A1:HZ}. We show in particular that we can always extend an adHDAE of the form~\eqref{eq:1HZ} satisfying Assumption  \ref{A1:HZ} to a system of the form \eqref{eq:adae} satisfying Assumption~\ref{A0:HZ} without changing the Hamiltonian.
\begin{theorem} \label{th:ext}
Consider an adHDAE of the form (\ref{eq:1HZ}) satisfying Assumption \ref{A1:HZ}. Let ${\mathcal A}_{ext}$ with $D({\mathcal A}_{ext}) \subset {\mathbb X}_1 \oplus {\mathbb X}_2 \oplus {\mathbb X}_3$ be a dissipative extension of ${\mathcal A}$ such that (\ref{eq:def-A}) and (\ref{eq:dom-A}) hold. Let $\mathcal E_3$ be a bounded and boundedly invertible operator on ${\mathbb X}_3$, and define $\mathcal E_{ext}$ and $\mathcal Q_{ext}$ as in (\ref{eq:2HZ}). Then the triple $(\mathcal E_{ext}, {\mathcal A}_{ext}, \mathcal Q_{ext})$ satisfies Assumption~\ref{A0:HZ} with the same Hamiltonian \eqref{defenergy}.
\end{theorem}
\proof
%By the assumptions made in the theorem, it only
It is clear that the Hamiltonian does not change, so it remains to show that $s\mathcal E_{ext}- {\mathcal A}_{ext} \mathcal Q_{ext}$ is boundedly invertible.
%Using the relation and the form of the %operators it is not hard to see that
The equation
\[
  (s\mathcal E_{ext} - {\mathcal A}_{ext}\mathcal Q_{ext})\begin{bmatrix} x_1 \\ x_2 \\ x_3\end{bmatrix} = \begin{bmatrix} y_1 \\ y_2 \\ y_3\end{bmatrix}
\]
is equivalent to the two equations
\[
  (s\mathcal E - {\mathcal A}\mathcal Q)\begin{bmatrix} x_1 \\ x_2 \end{bmatrix} = \begin{bmatrix} y_1 \\ y_2 \end{bmatrix} \mbox{ and } s \mathcal E_3 x_3 - {\mathcal A}_{3,ext}\begin{bmatrix} \mathcal Q_1 x_1 \\ \mathcal Q_2x_2 \\ 0 \end{bmatrix}= y_3.
\]
Since the pair $(\mathcal E,{\mathcal A}\mathcal Q)$ is regular we can determine $x_1$ and $x_2$ uniquely, and since $\mathcal E_3$ is boundedly invertible, $x_3$ is also  uniquely determined when $x_1,x_2$ are fixed. Since these inverse mappings are bounded, we conclude that $s\mathcal E_{ext} - {\mathcal A}_{ext}\mathcal Q_{ext}$ is boundedly invertible.
\hfill\eproof
\medskip

Based on Theorems~\ref{T:1} and~\ref{th:ext} we see that we can  reduce or extend regular adHDAEs when the Hamiltonian is not changed. For this reason from now on we only consider abstract DAEs without a component $x_3$,  i.e., we study the adHDAE (\ref{eq:1HZ})
under the Assumption \ref{A1:HZ}, see \cite{BeaMXZ18,MehU23,MehS23} for the finite dimensional case.
Note however, that for discretization methods and  practical applications it is essential to keep the equation for $x_3$ for initial value consistency checks and to avoid that the solution for the variables $x_1,x_2$ drifts off from the solution manifold, see \cite{KunM06}.

\subsection{Regularity and singularity of adHDAEs}\label{sec:regular}

In this section we consider the regularity and singularity of the pair of operators
\begin{equation}\label{eq:pair}
    (\mathcal E,\mathcal A\mathcal Q )
\end{equation}
associated with the adHDAE \eqref{eq:1HZ}. We study the regularity of (\ref{eq:pair}) under the first three conditions of Assumption \ref{A1:HZ}.
Using the fact that
\begin{equation}
\label{eq:4i}
  s\mathcal E- {\mathcal A} \mathcal Q = \left( s \begin{bmatrix} \mathcal E_1\mathcal Q_1^{-1}  & 0 \\  0& 0 \end{bmatrix} - {\mathcal A} \right)  \begin{bmatrix} \mathcal Q_1 & 0 \\ 0& \mathcal Q_2 \end{bmatrix}
\end{equation}
and that $\mathcal Q_1$ and $\mathcal Q_2$ are bounded and boundedly invertible,  the following lemma is immediate.
\begin{lemma}
  \label{L:2.1}
  The operator $s\mathcal E- {\mathcal A} \mathcal Q$ is boundedly invertible if and only if $s{\hat{\mathcal E}}- {\mathcal A}$ is boundedly invertible, where $\hat{\mathcal E} = \left[\begin{smallmatrix} \mathcal E_1\mathcal Q_1^{-1}  & 0 \\  0& 0 \end{smallmatrix} \right]$.

  Furthermore, $\mathcal E_1^*\mathcal Q_1$ is coercive if and only if $\mathcal E_1\mathcal Q_1^{-1}$ is coercive if and only if $\mathcal Q_1 \mathcal E_1^{-1}$ is coercive.
\end{lemma}
From this lemma we see that if we want to check the regularity of $(\mathcal E,{\mathcal A} \mathcal Q)$, we may without loss of generality assume that $\mathcal Q_1$ and $\mathcal Q_2$ are the identity operators, and that $\mathcal E_1$ is coercive. We begin by showing that regularity implies that ${\mathcal A}$ is \emph{maximally dissipative},
i.e., it is dissipative and for all $s > 0$ the operator $sI -\mathcal A$ is surjective.
\begin{lemma}
\label{L2.1a} Consider an abstract adHDAE of the form~\eqref{eq:adae} satisfying
Assumption \ref{A1:HZ}. Then the operator ${\mathcal A}$ is maximally dissipative.
\end{lemma}
\proof
If $s \in {\mathbb C}^+$ and since $\mathcal E_1\mathcal Q_1^{-1}$ is coercive,  we have  (see Lemma \ref{L:2.1})  that ${\mathcal A} - s \hat{\mathcal E}$ is dissipative. Since by assumption ${\mathcal A} - s \hat{\mathcal E}$ is boundedly invertible,  by Lemmas~\ref{L:A6} and~\ref{L:A4} in the appendix we have that it is maximally dissipative. Since $s \hat{\mathcal E}$ is bounded this means that ${\mathcal A}$ is maximally dissipative.
%In particular, ${\mathcal A}$ is closed and densely defined.
\eproof

\medskip

\begin{theorem}
\label{t_commonQI} Consider a triple of operators $(\mathcal E, \mathcal A ,\mathcal Q)$, where we assume that these operators satisfy the first three conditions of Assumption \ref{A1:HZ}.
Then the following are equivalent.
\begin{itemize}
\item[i)] The
pair $(\mathcal E, {\mathcal A}{\mathcal Q})$ is regular.
\item[ii)] For all $s\in {\mathbb C}^+$  the operator $ s\mathcal E -\mathcal A \mathcal Q$ is boundedly invertible.
\item[iii)] There exists an $s \in {\mathbb C}^+$ such that the operator $ s\mathcal E -\mathcal A \mathcal Q$ is boundedly invertible.
\item [iv)] The operator ${\mathcal A}$ is maximally dissipative,  %$\ker({\mathcal E}) \cap \ker({\mathcal A}^*{\mathcal Q})=\{0\}$,
and there exists an $m_1>0$ such that
\begin{equation}
\label{eq:4-2}
  \left\|\left[\begin{array}{c} {\mathcal E} \\ {\mathcal A}{\mathcal Q}\end{array}\right] x\right\|   \geq m_1 \|x\| \mbox{ for all } {\mathcal Q} x \in D({\mathcal A}).
\end{equation}
\item [v)] The operator ${\mathcal A}$ is maximally dissipative, % $\ker({\mathcal E}{\mathcal Q}^{-1}) \cap \ker({\mathcal A}^*)=\{0\}$,
and there exists an $m_2>0$ such that
\begin{equation}
\label{eq:4-2-bis}
  \left\|\left[\begin{array}{c} {\mathcal E}{\mathcal Q}^{-1} \\ {\mathcal A}\end{array}\right] x\right\|  \geq m_2 \|x\| \mbox{ for all } x \in D({\mathcal A}).
\end{equation}
\end{itemize}
\end{theorem}

\begin{proof}
It is clear that ii) implies iii), and iii) implies that $(\mathcal E, {\mathcal A}{\mathcal Q})$ is regular, and thus iii) implies i).
So we start by proving that i) implies iv).
By the given assumptions %in the theorem
and since i) holds, Assumption \ref{A1:HZ} holds. Thus Lemma \ref{L2.1a} gives that ${\mathcal A}$ is maximally dissipative. In particular it is densely defined and closed.

Let $s \in {\mathbb C}$ be such that $s{\mathcal E} - {\mathcal A} {\mathcal Q}$ is boundedly invertible, then for $x \in D({\mathcal A} {\mathcal Q})$
\begin{align*}
  \|x\| =&\ \|(s{\mathcal E} - {\mathcal A} {\mathcal Q})^{-1} (s{\mathcal E} - {\mathcal A} {\mathcal Q})x \| \leq M\|(s{\mathcal E} - {\mathcal A} {\mathcal Q})x\| \\
  =&\  M \left\| \begin{bmatrix} sI & -I \end{bmatrix} \begin{bmatrix} {\mathcal E} \\ {\mathcal A}{\mathcal Q}\end{bmatrix} x \right\| \leq M M_1 \left\| \begin{bmatrix} {\mathcal E} \\  {\mathcal A}{\mathcal Q}\end{bmatrix} x \right\|.
\end{align*}
Since both $M = \|(s{\mathcal E} - {\mathcal A} {\mathcal Q})^{-1}\|$ and $M_1 = \left\| \begin{bmatrix} sI & -I \end{bmatrix} \right\|$ are nonzero, (\ref{eq:4-2}) follows. It remains to show that $\ker({\mathcal E}) \cap \ker({\mathcal A}^*{\mathcal Q})=\{0\}$. Let $x\in \ker({\mathcal E}) \cap \ker({\mathcal A}^*{\mathcal Q})$, then $(\overline{s} {\mathcal Q}^* {\mathcal E} - {\mathcal Q}^*{\mathcal A}^*{\mathcal Q})x =0$. Since ${\mathcal Q}^* {\mathcal E}$ is self-adjoint, this is the same as $(\overline{s}  {\mathcal E}^*{\mathcal Q} - {\mathcal Q}^*{\mathcal A}^*{\mathcal Q})x =0$, and so
\[
  \langle {\mathcal Q} x, (s\mathcal E-\mathcal A\mathcal Q)y\rangle=0 \mbox{ for all } {\mathcal Q}y \in D(\mathcal A).
\]
Since $s\mathcal E-\mathcal A\mathcal Q$ is boundedly invertible its range equals ${\mathbb X}$ and thus ${\mathcal Q} x=0$, and since ${\mathcal Q}$ is boundedly invertible, $x=0$.
\smallskip

Since ${\mathcal Q}$ is boundedly invertible it is easy to see that items iv) and v) are equivalent. So it remains to show that iv) implies ii). To prove this, %at iv) implies ii),
suppose that (\ref{eq:4-2}) holds, but that $s\mathcal E-\mathcal A\mathcal Q$ is not boundedly invertible for some $s \in {\mathbb C}$ with positive real part. Then we have the following possibilities:
\begin{enumerate}
\item[(a)] The operator $s\mathcal E-\mathcal A \mathcal Q$ is not injective.
\item[(b)] The range of $s\mathcal E-\mathcal A\mathcal Q$ is not dense in {$\mathbb X$}.
\item[(c)] The range of $s\mathcal E-\mathcal A \mathcal Q$ is dense but not equal to {$\mathbb X$}.
\end{enumerate}

We will show that neither of these options is valid.

Case (a): If  there exists  $0\neq x \in D(\mathcal A\mathcal Q)$ such that $(s\mathcal E-\mathcal A\mathcal Q)x=0$ then consider  $\langle (s\mathcal E-\mathcal A\mathcal Q)x, {\mathcal Q}x \rangle$ which is zero,  and thus
\[
  0 = s \langle \mathcal Ex, {\mathcal Q} x \rangle  -\langle \mathcal A{\mathcal Q} x,{\mathcal Q} x \rangle = s \langle \mathcal E {\mathcal Q}^{-1} \mathcal Q x, {\mathcal Q} x \rangle  -\langle \mathcal A{\mathcal Q} x,{\mathcal Q} x \rangle .
\]
Since Re$(s) >0$, ${\mathcal E}{\mathcal Q}^{-1}$ is non-negative (see Lemma \ref{L:2.1}) and since $\mathcal A$ is dissipative, taking the real part gives that
\[
  0 = \langle \mathcal E {\mathcal Q}^{-1} \mathcal Q x, {\mathcal Q} x \rangle.
\]
Since $\mathcal E {\mathcal Q}^{-1}$ is a non-negative bounded operator, this gives that ${\mathcal Q}x \in \ker (\mathcal E {\mathcal Q}^{-1})$, or equivalently $x\in \ker (\mathcal E)$.

Applying this in  equation $(s\mathcal E-\mathcal A\mathcal Q)x=0$ gives $\mathcal A \mathcal Q x=0$. So we have shown that $x\neq 0$ lies in $\ker(\mathcal E) \cap \ker(\mathcal A \mathcal Q )$, which is a contradiction to (\ref{eq:4-2}).
\smallskip

Case (b): If there exists  $0\neq x \in {\mathbb X}$ such that
\begin{equation}
\label{eq:caseb}
    \langle x, (s\mathcal E-\mathcal A\mathcal Q)y\rangle=0 \mbox{ for all } {\mathcal Q}y \in D(\mathcal A),
\end{equation}
then $x$ lies in the domain of the dual operator, i.e., in $D((s\mathcal E-\mathcal A\mathcal Q)^*)$, which equals $D({\mathcal A}^*)$. Furthermore, since $D({\mathcal A})$ is dense in ${\mathbb X}$, (\ref{eq:caseb}) implies that
$0=(s\mathcal E-\mathcal A\mathcal Q)^*x = (\overline{s}{\mathcal E}^* -{\mathcal Q}^*{\mathcal A}^*)x$.
Writing $x= {\mathcal Q} z$ and using the fact that ${\mathcal A}$ is maximally dissipative, and thus ${\mathcal A}^*$ is dissipative, we can proceed as in case (a) to obtain that ${\mathcal E}^*{\mathcal Q}z=0$ and ${\mathcal Q}^* {\mathcal A}^*{\mathcal Q}z=0$,
or equivalently ${\mathcal E}^*x=0$ and ${\mathcal Q}^* {\mathcal A}^*x=0$. Since ${\mathcal Q}^* $ is boundedly invertible, this gives ${\mathcal A}^*x=0$. The latter gives that
\begin{equation}
\label{eq:37-i}
   \langle x, \mathcal A y\rangle=0 \mbox{ for all } y \in D(\mathcal A).
\end{equation}
Since ${\mathcal A}$ is maximally dissipative, we know that there exists  $y_x\in D({\mathcal A})$ such that
\begin{equation}
\label{eq:37-ii}
  (I- {\mathcal A}) y_x = x.
\end{equation}
Substituting this in (\ref{eq:37-i}) with $y=y_x$ gives
\[
  0= \langle x, {\mathcal A}y_x \rangle
  =
  \langle (I- {\mathcal A}) y_x, {\mathcal A} y_x \rangle = \langle y_x, {\mathcal A}y_x \rangle - \langle {\mathcal A} y_x, {\mathcal A}y_x \rangle.
\]
The last inner product is obviously real and non-positive. The real part of the first term is also non-positive, and thus both terms must be zero. This gives in particular that ${\mathcal A} y_x =0$, and by (\ref{eq:37-ii}) that $y_x=x$.
Note that we still have that ${\mathcal E}^*x=0$.

Since ${\mathcal Q}$ is boundedly invertible, we can define $\tilde{x} ={\mathcal Q}^{-1} x$, and so $\tilde{x} \in \ker({\mathcal A}{\mathcal Q}) \cap \ker({\mathcal E}^*{\mathcal Q})$.
Since ${\mathcal E}^*{\mathcal Q}$ is self-adjoint, this implies that $\tilde{x} \in \ker ({\mathcal Q}^*{\mathcal E})$, but since ${\mathcal Q}$ is boundedly invertible $\tilde{x} \in \ker ({\mathcal E})$. So substituting $\tilde{x}$ in (\ref{eq:4-2}) gives that $\tilde{x} =0$ and hence $x=0$, which is in contradiction to our assumption $x \neq 0$.
\smallskip

Case (c): Let $s \in {\mathbb C}^+$  be given and let ${\mathcal Q} x_n$ be a sequence in $D(\mathcal A)$ such that $(s\mathcal E-\mathcal A\mathcal Q)x_n \rightarrow z$ as $n \rightarrow \infty$ with $z\in\mathcal  X$, but not in the range of $s\mathcal E-\mathcal A\mathcal Q$. Then by defining $x_{n,m} = x_n - x_m$, we have
\begin{equation}
\label{eq:4-3}
  (s\mathcal E-\mathcal A\mathcal Q)x_{n,m} \rightarrow 0 \mbox{ as } n,m \rightarrow \infty.
\end{equation}
If $\|x_{n,m}\| \rightarrow 0$ as $n,m \rightarrow \infty$, then $x_n$ would be a Cauchy sequence, and thus converge to some  $x$. In that case $z = (s\mathcal E-\mathcal A\mathcal Q)x$, and thus in the range of $(s\mathcal E-\mathcal A\mathcal Q)x$. Hence in this case we have a contradiction. So we assume that  $\|x_{n,m}\|$ stays bounded away from zero for some sequence of indices $\{n,m\}$. In the remainder of the proof we consider this sequence.

Taking the inner product  of \eqref{eq:4-3} with $\frac{{\mathcal Q} x_{n,m}}{\|x_{n,m}\|}$, gives
\[
  0 = \lim_{n,m \rightarrow \infty} \left[ s \langle \mathcal Ex_{n,m}, \frac{{\mathcal Q}x_{n,m}}{\|x_{n,m}\|}\rangle  -\langle \mathcal A {\mathcal Q} x_{n,m} , \frac{{\mathcal Q}x_{n,m}}{\|x_{n,m}\|}\rangle  \right].
\]
Since Re$(s) >0$ and ${\mathcal Q}^*{\mathcal E}$ is self-adjoint, taking the real part gives
\[
  0 = \lim_{n,m \rightarrow \infty}  \left[ \mathrm{Re}(s) \langle {\mathcal Q}^* \mathcal Ex_{n,m}, \frac{x_{n,m}}{\|x_{n,m}\|}\rangle   - \mathrm{Re}\left(\langle \mathcal A{\mathcal Q} x_{n,m} ,\frac{{\mathcal Q}x_{n,m}}{\|x_{n,m}\|} \rangle\right)\right].
\]
Both terms are nonnegative and since ${\mathcal Q}^* \mathcal E\geq 0$ we find that
\begin{equation}
  \label{eq:4-4}
   \lim_{n,m \rightarrow \infty}{\mathcal Q}^* \mathcal  E \frac{x_{n,m}}{\sqrt{\|x_{n,m}\|}} =0 \Rightarrow \lim_{n,m \rightarrow \infty}\mathcal  E\frac{x_{n,m}}{\sqrt{\|x_{n,m}\|}} = 0,
 \end{equation}
 where we have used that ${\mathcal Q}$ is boundedly invertible.
 Applying this in (\ref{eq:4-3}) and using that $\|x_{n,m}\|$ stays bounded away from zero, we find that
\begin{equation}
 \label{eq:4-5}
   \lim_{n,m \rightarrow \infty} {\mathcal A} {\mathcal Q} \frac{x_{n,m}}{\sqrt{\|x_{n,m}\|}} =0.
\end{equation}
Define $z_{n,m} = \frac{x_{n,m}}{\sqrt{\|x_{n,m}\|}} $, then by (\ref{eq:4-3}) and (\ref{eq:4-4}), equation (\ref{eq:4-2}) implies that $z_{n,m} \rightarrow 0$. However, by (\ref{eq:4-2}) this gives that
\[
  \lim_{n,m \rightarrow \infty} \|z_{n,m}\| =0
\]
which is equivalent to $\sqrt{\|x_{n,m}\|} \rightarrow 0$, which is a contradiction.

So we see that neither of the cases $(a)$, $(b)$, or $(c)$ is possible, and hence item ii) holds.
\hfill\eproof
\end{proof}
\medskip

From Theorem~\ref{t_commonQI} we can derive some easy consequences, but we begin by showing that the conditions as stated in item iv) and v) can be simplified when ${\mathcal A}$ has more structure.
\begin{lemma}
\label{L:A=J-R}
Consider a triple of operators $(\mathcal E, \mathcal A ,\mathcal Q)$ that satisfy the first three conditions of Assumption \ref{A1:HZ}.
 Assume further that ${\mathcal A}$ can be written as ${\mathcal A}=\mathcal J-\mathcal R$, with
$\mathcal J$ skew-adjoint, i.e., $\mathcal J^*=-\mathcal J$ and $\mathcal R$ is bounded, self-adjoint and non-negative, then item v) in Theorem \ref{t_commonQI} is equivalent to%can be replaced by
\begin{itemize}
\item[v')]
  There exists an $m_2>0$ such that
  \begin{equation}
  \label{eq:4-2-bis2}
  \left\|\left[\begin{array}{c} {\mathcal E}{\mathcal Q}^{-1} \\ {\mathcal J}\\ {\mathcal R}
  \end{array}\right] x\right\|  \geq m_2 \|x\| \mbox{ for all } x \in D({\mathcal J}).
\end{equation}
\end{itemize}
\end{lemma}
\begin{proof} Assume that v) holds, then we only have to show that (\ref{eq:4-2-bis}) implies (\ref{eq:4-2-bis2}). If (\ref{eq:4-2-bis2}) would not hold, then there exits a sequence $\{x_n\}, n \in {\mathbb N}$ such that, $x_n \in D({\mathcal J})$, $\|x_n\|=1$ and ${\mathcal E}{\mathcal Q}^{-1}x_n, {\mathcal J}x_n$ and ${\mathcal R}
   x_n$ all converge to zero. This implies that ${\mathcal E}{\mathcal Q}^{-1}x_n$ and ${\mathcal A}x_n = ({\mathcal J} - {\mathcal R})x_n$ converge to zero, which is a contradiction.
   \smallskip

Next we assume that v') holds. Then, since
%${\mathcal A}$ is given as
${\mathcal A} = {\mathcal J} - {\mathcal R}$, with ${\mathcal R}$ is bounded and non-negative, and ${\mathcal J}$  skew-adjoint, we have that $D({\mathcal A})= D({\mathcal A}^*)$ and both ${\mathcal A}$ and ${\mathcal A}^* = -{\mathcal J} - {\mathcal R}$ are dissipative, which implies that ${\mathcal A}$ is maximally dissipative.

Let $x_n\in D({\mathcal J})$ be of norm one, and assume that ${\mathcal A}x_n \rightarrow 0$ as $n \rightarrow \infty$, then
\[
  \langle x_n, ({\mathcal J}- {\mathcal R})x_n \rangle \rightarrow 0
\]
Taking the real part of this expression gives that $\langle x_n,{\mathcal R}x_n \rangle \rightarrow 0$. Since ${\mathcal R}$ is non-negative and bounded, this implies that ${\mathcal R}x_n \rightarrow 0$, see \cite[Lemma A.3.88.c]{CuZw20}.  Combining this with ${\mathcal A}x_n \rightarrow 0$ gives that ${\mathcal J}x_n \rightarrow 0$. Hence from this we conclude that (\ref{eq:4-2-bis2}) implies (\ref{eq:4-2-bis}).
%Very similarly we have that (\ref{eq:4-2-bis2}) implies $\ker({\mathcal E}{\mathcal Q}^{-1}) \cap \ker({\mathcal A}^*) =\{0\}$, and thus v') implies v).
\hfill\eproof
\end{proof}
\medskip

Note that similarly, the condition iv) in Theorem~\ref{t_commonQI} can be replaced.
Note further that for matrices or bounded operators a dissipative ${\mathcal A}$ can always be written as ${\mathcal A}=\mathcal J-\mathcal R$, with
$\mathcal J$ skew-adjoint and $\mathcal R$ non-negative.

The result of Lemma \ref{L:A=J-R} also holds when ${\mathcal A} =\mathcal J-\mathcal R$, with
$\mathcal J$ skew-adjoint and bounded and $\mathcal R$ self-adjoint and non-negative. In that case we have $D({\mathcal A})=D({\mathcal R})$.

Given the special form of ${\mathcal E}$ and ${\mathcal Q}$ the following is an easy consequence of Theorem \ref{t_commonQI}.
\begin{corollary}
\label{C:8}
Consider an adHDAE of the form~\eqref{eq:1HZ} satisfying
  the  conditions i)-iii) in Assumption \ref{A1:HZ} and define
  \begin{equation}
\label{eq:5}
  \mathcal E_I := \begin{bmatrix} I & 0  \\ 0& 0 \end{bmatrix}.
\end{equation}
  Then the following are equivalent.
\begin{enumerate}
\item [i)]
  There exists an $s \in {\mathbb C}^+$ such that the operator $s\mathcal E-{\mathcal A}\mathcal Q$  is boundedly invertible.
\item [ii)]
  There exists an $s \in {\mathbb C}^+$ such that the operator $s\mathcal E_I-{\mathcal A}$  is boundedly invertible.
\end{enumerate}
\end{corollary}
\begin{proof}
  From Theorem \ref{t_commonQI}
  %we see that
  we have to show that we may replace ${\mathcal E}{\mathcal Q}^{-1}$ by ${\mathcal E}_I$. This follows since
  \[
    {\mathcal E}_I= \begin{bmatrix} I & 0  \\ 0& 0 \end{bmatrix} = \begin{bmatrix} {\mathcal Q}_1{\mathcal E}_1^{-1} & 0  \\ 0& I \end{bmatrix} \begin{bmatrix} {\mathcal E}_1{\mathcal Q}_1^{-1}  & 0  \\ 0& 0 \end{bmatrix},
    %= {\mathcal E}{\mathcal Q}^{-1},
  \]
  where we have used the invertiblity of ${\mathcal E}_1$ and ${\mathcal Q}$. So ${\mathcal E}{\mathcal Q}^{-1}$ and ${\mathcal E}_I$ are boundedly invertible related to each other, and this implies that in Theorem \ref{t_commonQI} part iv) and v) we may do the replacements.
\hfill\eproof
\end{proof}
\medskip

We have shown that the regularity of the pair $(\mathcal E,{\mathcal A} \mathcal Q)$ is equivalent to that of $(\mathcal E_I, {\mathcal A})$. However, this may still be a difficult condition to check. In the following lemma we derive conditions under which  this follows from the maximal dissipativity of ${\mathcal A}$.
\begin{lemma}
\label{L:9}Consider an adHDAE of the form~\eqref{eq:1HZ} satisfying   the  conditions i)-iii) in Assumption~\ref{A1:HZ}.
 If there exists an $\varepsilon >0$ such that ${\mathcal A} + \varepsilon \begin{bmatrix} 0 & 0  \\ 0& I \end{bmatrix}$ is maximally dissipative, then $(\mathcal E,{\mathcal A} \mathcal Q)$ is regular.
\end{lemma}
\proof
  Since ${\mathcal A} + \varepsilon \begin{bmatrix} 0 & 0  \\ 0& I \end{bmatrix}$ is maximally dissipative, we know that ${\mathcal A} + \varepsilon \begin{bmatrix} 0 & 0  \\ 0& I \end{bmatrix} - \delta \begin{bmatrix} I & 0  \\ 0& I \end{bmatrix}$ is boundedly invertible for every $\delta >0$. Choosing $\delta = \varepsilon$, we see that this implies that $\varepsilon \mathcal E_I - {\mathcal A}$ is boundedly invertible. By Corollary \ref{C:8} it follows that this is equivalent to $(\mathcal E,{\mathcal A} \mathcal Q)$ being regular.
\hfill \eproof
\medskip

We end this section with a few observations and additional results.

In the finite-dimensional case it has been shown in \cite{MehMW21} that if $\mathcal Q$ is injective and the pair is singular then the three matrices
$\mathcal E,\mathcal J\mathcal Q, \mathcal R\mathcal Q$ have a common nullspace. { Here $\mathcal J=\frac{1}{2} (\mathcal A-\mathcal A^*)$ and $\mathcal R= -\frac{1}{2} (\mathcal A+\mathcal A^*)$.}
However, this is not true if $\mathcal Q$ is not injective.
\begin{example}\label{ex:counter}{\rm
Consider the matrices
\[
  \mathcal E= \begin{bmatrix} e_{11} & e_{12} \\ 0 & 0 \end{bmatrix}, \quad {\mathcal J}= \begin{bmatrix} 0 & -1 \\ 1 & 0 \end{bmatrix},\quad \mathcal Q= \begin{bmatrix} 0 & 0 \\ q_{21} & q_{22} \end{bmatrix}.
\]
Then
\[
  \mathcal J\mathcal Q = \begin{bmatrix} -q_{21} & -q_{22}\\  0 & 0  \end{bmatrix}, \mbox{ and } s\mathcal E - \mathcal J\mathcal Q = \begin{bmatrix} s e_{11}  + q_{21} & se_{12}   + q_{22} \\ 0 & 0 \end{bmatrix},
\]
and so the pair $(\mathcal E,\mathcal J\mathcal Q)$ is singular. Furthermore,
\[
 \mathcal E^* \mathcal Q = \begin{bmatrix} e_{11} & 0  \\ e_{12} & 0 \end{bmatrix}\begin{bmatrix} 0 & 0 \\ q_{21} & q_{22} \end{bmatrix} = \begin{bmatrix} 0 & 0  \\ 0 & 0 \end{bmatrix}
\]
is symmetric, and positive semidefinite.
However, $\mathcal E$ and $\mathcal J\mathcal Q$ do not have a common kernel.
}%end rm
\end{example}

Since our aim was to study regularity, i.e., boundedly invertibility of $s{\mathcal E} - {\mathcal A}{\mathcal Q}$, we had to check conditions for injectivity and surjectivity. However, separate conditions can also be obtained.
%
%\begin{theorem} \label{t:vcommon} Consider a pair of operators $(\mathcal E, (\mathcal J -\mathcal R)\mathcal Q)$. If $\mathcal Q$ and $\mathcal Q^*$ are injective, then every $s\in {\mathbb C}$ (or every positive $s$) is in the point spectrum of the pencil $ \lambda \mathcal E -(\mathcal J-\mathcal R) \mathcal Q$ if and only if $\ker \mathcal R\mathcal Q \cap \ker \mathcal E \cap \ker \mathcal J \mathcal Q \neq \{0\}$.
%\end{theorem}
%
%
%Using Lemma~\ref{L2.1a}, we  show next, that if Assumption~\ref{A1:HZ} holds, then regularity is independent of the particular representation of $\mathcal E$, $\mathcal Q$, and the complex number $s$.
%
\begin{lemma}
  \label{L:2.2-new} Consider an adHDAE of the form~\eqref{eq:adae} satisfying
  the  conditions i)-iii) in Assumption \ref{A1:HZ}, and let $\tilde{\mathcal E}$ and $\tilde{\mathcal Q}$ satisfy the same assumptions as $\mathcal E$ and $\mathcal Q$ in Assumption \ref{A1:HZ}, respectively. Furthermore, let $s,\tilde{s} \in {\mathbb C}^+$. Then the following assertions hold.
  \begin{enumerate}
  \item[a)]
    $(\tilde{s}\tilde{\mathcal E} - {\mathcal A}\tilde{\mathcal Q})$ is injective if and only if $(s \mathcal E - {\mathcal A} \mathcal Q)$ is. Furthermore, this holds if and only if ${\mathcal A}\left[\begin{smallmatrix}  0 \\ x_2 \end{smallmatrix} \right] =0$ implies $x_2=0$.
    \item[b)]
  Let ${\mathcal A}$ be maximally dissipative. The range of $\tilde{s}\tilde{\mathcal E} - {\mathcal A}\tilde{\mathcal Q}$ is dense if and only if the range of $s\mathcal E - {\mathcal A}\mathcal Q$ is dense.
    Furthermore, this holds if and only if ${\mathcal A}^*\left[\begin{smallmatrix}  0 \\ z_2 \end{smallmatrix} \right] =0$ implies $z_2=0$.
  \end{enumerate}
\end{lemma}
\proof
The proofs are similar to the corresponding parts of the proof of Theorem \ref{t_commonQI}.
\hfill \eproof

\subsection{Special block operators}\label{sec:stokes}
As a prototypical example of adHDAEs, in this section we study special block operators pairs, as they arise e.g. in Stokes and Oseen equations that have been formulated in \cite{EmmM13}, or \cite{ReiS23}, as abstract DAE. Similar abstract block DAE operators arise also in the study of the Euler equations in gas transport \cite{EggK18,EggKLMM18}.
%and in  \cite{ReiS23} as  port-Hamiltonian systems.
%We follow the setting of \cite{EmmM13} to analyze the %properties of the associated  pair of operators.

In the following ${\mathcal L}({\mathbb W},{\mathbb Y})$ denotes the space of  bounded, linear operators between Hilbert spaces ${\mathbb W}$ and ${\mathbb Y}$. Furthermore, ${\mathcal L}({\mathbb W}) = {\mathcal L}({\mathbb W},{\mathbb W})$.

Let ${\mathbb V}$ be a real Hilbert space such that ${\mathbb V} \subsethook {\mathbb X}_1 = { \mathbb X}_1^* \subsethook { \mathbb V}^*$, i.e., they form a  \emph{Gelfand triple}, see e.g.,~\cite{Wlo87}. Let $ \mathcal A_0 \in {\mathcal L}({\mathbb V},{ \mathbb V}^*)$, $\mathcal B_0 \in {\mathcal L}({\mathbb U},{\mathbb V}^*)$, where ${\mathbb U}$ is a second (real) Hilbert space. So  $\mathcal  B_0^* \in {\mathcal L}({\mathbb V},{\mathbb U})$, where we have identified ${\mathbb U}^*$ with ${\mathbb U}$. Finally, with these operators and $\mathcal D_0 \in {\mathcal L}({\mathbb U})$,
we define the block operator
 \begin{equation}
  \label{eq:50}
    {\mathcal A} = \begin{bmatrix}\mathcal  A_0 & \mathcal  B_0 \\-\mathcal  B_0^* &- \mathcal D_0 \end{bmatrix}
 \end{equation}
 with domain
 \begin{equation}
 \label{eq:50a}
    D({\mathcal A}) = \{ \left[\begin{smallmatrix} v \\ u \end{smallmatrix}\right] \in {\mathbb V} \oplus {\mathbb U} \mid  \mathcal  A_0v +  \mathcal B_0u \in {\mathbb X}_1\}.
 \end{equation}

For this operator $\mathcal A$, we study the pair $(\mathcal E_I,\mathcal A)$ with $\mathcal E_I$ as in \eqref{eq:5} and begin our analysis with two simple lemmas.
\begin{lemma}
\label{L:12}
 Consider the operator ${\mathcal A}$ as in (\ref{eq:50}) with its domain as in (\ref{eq:50a}). Assume that $-\mathcal  D_0$ and $\mathcal  A_0$ are dissipative,
 i.e.,

\begin{equation}
\label{eq:50b}
     \langle \mathcal  A_0 v, v \rangle_{{\mathbb V}^*,{\mathbb V}} \leq 0 \mbox{ for all } v \in {\mathbb V}, \ {\langle \mathcal  -D_0 u, u \rangle_{{\mathbb U}^*,{\mathbb U}} \leq 0 \mbox{ for all } u \in {\mathbb U},}
\end{equation}
then ${\mathcal A}$ is dissipative on ${\mathbb X}_1 \oplus {\mathbb U}$.
 \end{lemma}
\proof
To show that ${\mathcal A}$ is dissipative on ${\mathbb X}_1 \oplus {\mathbb U}$, we choose $\left[\begin{smallmatrix} v \\ u \end{smallmatrix}\right] \in D({\mathcal A})$.
Then we have
 \begin{align*}
   \langle {\mathcal A}\begin{bmatrix} v \\ u \end{bmatrix} , \begin{bmatrix} v \\ u \end{bmatrix} \rangle_{{\mathbb X}_1 \oplus {\mathbb U}} &+ \langle \begin{bmatrix} v \\ u \end{bmatrix} , {\mathcal A} \begin{bmatrix} v \\ u \end{bmatrix} \rangle_{{\mathbb X}_1 \oplus {\mathbb U}} \\
   =&\ \langle  \mathcal A_0v + \mathcal B_0u, v \rangle_{{\mathbb X}_1} + \langle   v, \mathcal A_0v + \mathcal B_0u \rangle_{{\mathbb X}_1} +\\
   &\  \langle -\mathcal B_0^* v-\mathcal D_0 u, u\rangle_{{\mathbb U}} + \langle u, -\mathcal B_0^* v-\mathcal D_0 u\rangle_{{\mathbb U}} \\
   =&\ \langle  \mathcal A_0v + \mathcal B_0u, v \rangle_{{\mathbb V}^*, {\mathbb V}} + \langle   v, \mathcal A_0v + \mathcal B_0u \rangle_{{\mathbb V},{\mathbb V}^*} - \\
   &\ \langle \mathcal B_0^* v, u\rangle_{{\mathbb U}} - \langle u, \mathcal B_0^* v\rangle_{{\mathbb U}}- \langle \mathcal D_0 u,u\rangle_{{\mathbb U}} - \langle u, \mathcal D_0 u\rangle_{{\mathbb U}}\\
   =&\ \langle  \mathcal A_0v, v \rangle_{{\mathbb V}^*, {\mathbb V}} +  \langle \mathcal B_0u, v \rangle_{{\mathbb V}^*, {\mathbb V}} + \langle   v, \mathcal A_0v \rangle_{{\mathbb V},{\mathbb V}^*}  +\langle   v, \mathcal B_0u \rangle_{{\mathbb V},{\mathbb V}^*} \\
   & - \langle v, \mathcal B_0 u\rangle_{{\mathbb V},{\mathbb V}^*} - \langle \mathcal B_0u, v\rangle_{{\mathbb V}^*,{\mathbb V}} - \langle \mathcal D_0 u, u\rangle_{{\mathbb U}} - \langle u, \mathcal D_0 u\rangle_{{\mathbb U}}\\
   =& \ \langle  \mathcal A_0v, v \rangle_{{\mathbb V}^*, {\mathbb V}} + \langle   v, \mathcal A_0v \rangle_{{\mathbb V},{\mathbb V}^*} - \langle \mathcal D_0 u, u\rangle_{{\mathbb U}} - \langle u, \mathcal D_0 u\rangle_{{\mathbb U}} \leq 0,
 \end{align*}
 where we have used (\ref{eq:50b}). Thus we have proved the assertion.
\hfill \eproof

{ Therefore}, if we choose $\mathcal E=\mathcal E_I$, $\mathcal Q=I$, and ${\mathcal A}$ as in (\ref{eq:50})--(\ref{eq:50a}) to be dissipative, then the conditions i)--iii) of Assumption~\ref{A1:HZ} are  satisfied. In this setting, we study the injectivity of $(\mathcal E_I,{\mathcal A})$.
 \begin{lemma}
 \label{L:13}
 Consider the operator ${\mathcal A}$ with its domain as in (\ref{eq:50}) and (\ref{eq:50a}). Suppose that ${\mathcal A}$ is dissipative and one of the following two conditions holds:
\begin{itemize}
 \item [a)]
   $\mathcal B_0$ is injective, or
 \item [b)]
   the range of $\mathcal B_0$ intersected with ${\mathbb X}_1$ contains only the zero element,
 \end{itemize}
then $\mathcal E_I-{\mathcal A}$ is injective.
 \end{lemma}
\proof
We use Lemma \ref{L:2.2-new} a) to prove the assertion and study the equation ${\mathcal A} \left[\begin{smallmatrix} 0 \\ u \end{smallmatrix}\right]=0$. Note that this implies in particular that $ \left[\begin{smallmatrix} 0 \\ u \end{smallmatrix}\right] \in D({\mathcal A})$. By (\ref{eq:50a}) this gives the condition that $\mathcal A_0 0+ \mathcal B_0u =\mathcal B_0u \in {\mathbb X}_1$. So if b)  holds, this can only happen when $u=0$. If $\mathcal B_0$ can map into ${ \mathbb X}_1$, then the equation ${\mathcal A} \left[\begin{smallmatrix} 0 \\ u \end{smallmatrix}\right]=0$ implies $\mathcal B_0u=0$. Then a) gives $u=0$, and the proof is complete.
\hfill\eproof
\medskip

Note that condition b) in Lemma~\ref{L:13} is sometimes rephrased as \emph{$\mathcal B_0$ is completely unbounded}.

To show that $\mathcal E_I-{\mathcal A}$ is boundedly invertible, we need stronger conditions on $\mathcal B_0$ and $\mathcal A_0$.
We say that $\mathcal A_0$ \emph{satisfies a G\aa rding inequality} with respect to ${\mathbb X}_1$ and ${\mathbb V}$, if
there exists an $\alpha_1>0$ such for all $v\in {\mathbb V}$ the inequality
\begin{equation}
\label{eq:22}
 \|v\|^2_{{\mathbb X}_1} + |\langle \mathcal A_0 v, v \rangle_{{\mathbb V}^*,{\mathbb V}}| \geq \alpha_1 \|v\|^2_{{\mathbb V}}
\end{equation}
holds. Note that since $\mathcal A_0 \in {\mathcal L}({\mathbb V},{\mathbb V}^*)$ and ${\mathbb V} \subsethook {\mathbb X}_1$, we always have that
\[
 \|v\|^2_{{\mathbb X}_1} + |\langle \mathcal A_0 v, v \rangle_{{\mathbb V}^*,{\mathbb V}}| \leq \alpha_2  \|v\|^2_{{\mathbb V}}
\]
for some $\alpha_2>0$.
\begin{lemma}
\label{L:14}
Let $\mathcal A_0 \in {\mathcal L}({\mathbb V},{\mathbb V}^*)$ be dissipative and satisfy the G\aa rding inequality~\eqref{eq:22}. Then
$i_{{\mathbb V}} - \mathcal A_0$ is a boundedly invertible operator from ${\mathbb V}$ to ${\mathbb V}^*$. Here $i_{{\mathbb V}}$ is the inclusion map from  ${\mathbb V}$ into ${\mathbb V}^*$, i.e., $i_{{\mathbb V}}(v)=v$, for $v \in {\mathbb V}$.
\end{lemma}
\proof
See e.g.\ [Section 6.5] in \cite{Hac17}.
\hfill \eproof
\medskip

We will now present two theorems which give sufficient conditions for $(\mathcal E_I,{\mathcal A})$ to be regular. %
%We begin with the one which is closely %related to Theorem \ref{T:11}.
\begin{theorem}
\label{T:15}
Consider the operator ${\mathcal A}$ given by (\ref{eq:50}) and (\ref{eq:50a}). Let $\mathcal A_0$ and $-\mathcal D_0$  be dissipative, and  assume further that $\mathcal A_0$  satisfies the G\aa rding inequality (\ref{eq:22}). Finally, let $\left[ \begin{smallmatrix} \mathcal B_0\\ {\mathcal D}_0 \end{smallmatrix} \right]$be injective and have closed range, i.e., there exists $\beta>0$ such that for all $u \in {\mathbb U}$
\begin{equation}
\label{eq:23}
  \left\|\begin{bmatrix} \mathcal B_0\\ {\mathcal D}_0 \end{bmatrix} u\right\|_{ {\mathbb V}^* \oplus {\mathbb U}} \geq \beta \|u\|_{\mathbb U}.
\end{equation}

Under these conditions,  $\mathcal E_I - {\mathcal A}$ is boundedly invertible.

Moreover, $\mathcal E_I - {\mathcal A}$ is boundedly invertible if and only if the \emph{Schur complement} $\mathcal B_0^*(i_{{\mathbb V}} - \mathcal A_0)^{-1}\mathcal B_0+\mathcal D_0$ is boundedly invertible.
\end{theorem}
\proof
The proof consists of several parts. We begin by showing that the  Schur complement
%/transfer
function $G(1):=\mathcal B_0^*(i_{{\mathbb V}} - \mathcal A_0)^{-1}\mathcal B_0+\mathcal D_0$ is accretive, i.e., for all $u \in {{\mathbb U}}$ it  holds that
\begin{equation}
\label{eq:G1>0}
  \mathrm{Re}\langle G(1)u,u\rangle \geq 0.
\end{equation}
We have
\begin{align*}
  \langle G(1)u,u\rangle =&\ \langle \mathcal B_0^*(i_{{\mathbb V}} - \mathcal A_0)^{-1}\mathcal B_0u, u \rangle +  \langle {\mathcal D}_0u,u\rangle\\
  =&\ \langle (i_{{\mathbb V}} - \mathcal A_0)^{-1}\mathcal B_0u,{\mathcal B}_0u \rangle_{{{\mathbb V}},{{\mathbb V}}^*} +  \langle {\mathcal D}_0u,u\rangle\\
   =&\ \langle v, (i_{{\mathbb V}} - \mathcal A_0)v\rangle_{{{\mathbb V}},{{\mathbb V}}^*} +  \langle {\mathcal D}_0u,u\rangle,
\end{align*}
with $v= (i_{{\mathbb V}} - \mathcal A_0)^{-1}\mathcal B_0u$. Since $-{\mathcal D}_0$ and ${\mathcal A}_0$ are dissipative, inequality (\ref{eq:G1>0}) then follows.
\smallskip

Next we show that ${\mathcal A}$ is maximally dissipative. For this we look at the equation
\[
  (I - {\mathcal A}) \begin{bmatrix} v \\ u \end{bmatrix} = \begin{bmatrix} x_1 \\ y \end{bmatrix}
\]
for an arbitrary $x_1 \in {{\mathbb X}}_1$ and $y \in {{\mathbb U}}$, where we search a solution  $\left[\begin{smallmatrix} v \\ u \end{smallmatrix} \right] \in D({\mathcal A})$. The above equation can be written as two equations
\[
  (I - {\mathcal A}_0)v - {\mathcal B}_0 u = x_1\quad \mbox {and } {\mathcal B}_0^*v + {\mathcal D}_0u + u = y.
\]
Since ${{\mathbb X}}_1 \subset {{\mathbb V}}^*$, we have by Lemma \ref{L:14} that the first equation has the solution $v \in {\mathbb V}$ given by
\begin{equation}\
\label{eq:v}
  v= ({i_{{\mathbb V}}- \mathcal A}_0)^{-1}\mathcal B_0u +(i_{{\mathbb V}} - {\mathcal A}_0)^{-1}x_1.
\end{equation}
Substituting this in the second equation leads to the following equation for $u$
\[
   {\mathcal B}_0^*({i_{{\mathbb V}}- \mathcal A}_0)^{-1}\mathcal B_0u + {\mathcal D}_0u + u + {\mathcal B}_0^*(i_{{\mathbb V}} - {\mathcal A}_0)^{-1}x_1 = y,
\]
which we can write as
\begin{equation}
\label{eq:u}
  (G(1) +I)u = y - {\mathcal B}_0^*(i_{{\mathbb V}} - {\mathcal A}_0)^{-1}x_1.
\end{equation}
By (\ref{eq:G1>0}) we have that $-G(1)$ is a dissipative operator which is bounded, and thus maximally dissipative. Hence for every $x_1 \in {{\mathbb X}}_1$ and $y \in {\mathbb U}$ the equation (\ref{eq:u}) has a unique solution, depending continuously on $x_1$ and $y$. Now $v$ is given by (\ref{eq:v}) which depends continuously on $u$ and $x_1$, and thus on $y$ and $x_1$. It remains to show that $\left[\begin{smallmatrix} v \\ u \end{smallmatrix} \right] \in D({\mathcal A})$. This follows directly since ${\mathcal A}_0v + {\mathcal B}_0 u = -x_1 +v$. So we conclude that ${\mathcal A}$ is maximally dissipative.
\smallskip

Next we show that ${\mathcal A}$ satisfies (\ref{eq:4-2}). We know that we only have to show this for ${\mathcal E}_I$, see Corollary \ref{C:8}.  If (\ref{eq:4-2}) would not hold, then there exists a sequence $\left[\begin{smallmatrix} v_n \\ u_n \end{smallmatrix} \right] \in {{\mathbb X}}_1 \oplus {{\mathbb U}}$ of norm 1, such that $\left[\begin{smallmatrix} v_n \\ u_n \end{smallmatrix} \right] \in D({\mathcal A})$ and ${\mathcal E}_I \left[\begin{smallmatrix} v_n \\ u_n \end{smallmatrix} \right]$, ${\mathcal A}\left[\begin{smallmatrix} v_n \\ u_n \end{smallmatrix} \right]$ both converge to zero. This can equivalently be formulated as $v_n \rightarrow 0$ in ${{\mathbb X}}_1$ and
\begin{equation}
\label{eq:35-i}
  \mathcal A_0 v_n + \mathcal B_0 u_n \rightarrow 0 \mbox{ in }  {\mathbb X}_1, \quad \mathcal B_0^*v_n+ \mathcal D_0u_n\rightarrow 0  \mbox{ in }  {\mathbb U}.
\end{equation}
We have the following equalities
\begin{align}
  \langle v_n, \mathcal A_0 v_n  +& \mathcal B_0 u_n \rangle_{{\mathbb X}_1} - \langle u_n,\mathcal B_0^*v_n+ \mathcal D_0u_n\rangle \nonumber \\
  &=
  \langle v_n, \mathcal A_0 v_n\rangle_{{\mathbb V}, {\mathbb V}^*} +  \langle v_n,  \mathcal B_0 u_n \rangle_{{\mathbb V}, {\mathbb V}^*} - \langle u_n,\mathcal B_0^*v_n\rangle_{{\mathbb U}} -\langle u_n, \mathcal D_0u_n\rangle\nonumber \\
  &= \langle v_n, \mathcal A_0 v_n \rangle_{{\mathbb V}, {\mathbb V}^*} -  \langle u_n, \mathcal D_0u_n \rangle_{{\mathbb U}} .\label{aboveidentity}
\end{align}

 By (\ref{eq:35-i}) both summands in the left-most term converge to zero, and thus also the sum in the right-most term. Since the spaces are real and the operators $\mathcal A_0$ and $-\mathcal D_0$ are dissipative, $- \langle u_n, \mathcal D_0u_n \rangle_{{\mathbb U}}$ and $\langle v_n, \mathcal A_0 v_n \rangle_{{\mathbb V}, {\mathbb V}^*}$ take values in $(-\infty,0]$. This shows that $\langle v_n \mathcal A_0 v_n \rangle_{{\mathbb V}, {\mathbb V}^*} \rightarrow 0$ as $n \rightarrow \infty$.
Combining this with $v_n \rightarrow 0$ in ${{\mathbb X}}_1$ and the G\aa rding inequality (\ref{eq:22}) gives $v_n \rightarrow 0$ in ${{\mathbb V}}$. Since $\mathcal A_0$ is bounded from ${\mathbb V}$ to ${\mathbb V}^*$, and since ${\mathbb X}_1 \subsethook   {\mathbb V}^*$, we find that $\mathcal A_0v_n \rightarrow 0$ in ${\mathbb V}^*$ { as $n\to \infty$}.
The first relation in equation (\ref{eq:35-i}) gives that $\mathcal B_0u_n \rightarrow 0$ in ${\mathbb V}^*$.

 Since $v_n \rightarrow 0$ in ${\mathbb V}$ and since ${\mathcal B}_0$ is bounded from ${\mathbb V}$ to ${\mathbb U}$, we have that ${\mathcal B}_0v_n \rightarrow 0$ in ${\mathbb U}$. The second relation in equation (\ref{eq:35-i}) gives that $\mathcal D_0u_n \rightarrow 0$ in ${\mathbb U}$. Since $\mathcal B_0u_n$ and $\mathcal D_0u_n$ converge to zero, inequality (\ref{eq:23}) gives that $u_n \rightarrow 0$.
  Combined with $v_n \rightarrow 0$ in ${{\mathbb X}}_1$ this is in contradiction to the assumption that
$\left\|\left[\begin{smallmatrix} v_n \\ u_n \end{smallmatrix} \right] \right|_{{{\mathbb X}}_1 \oplus {{\mathbb U}}} =1$. Hence (\ref{eq:4-2}) holds.
\smallskip

\smallskip

So we have shown that $({\mathcal E}_I, {\mathcal A})$ satisfies the condition of Theorem \ref{t_commonQI}.iv), and thus also that $({\mathcal E}, {\mathcal A}{\mathcal Q})$ is regular. It remains to prove the last assertion of the theorem.
\smallskip

To prove that the invertibility of $G(1)$ implies the invertibility of $\mathcal E_I-{\mathcal A}$, we proceed similar to the first item in this proof. Namely, the equation $(\mathcal E_I-{\mathcal A})\left[\begin{smallmatrix} v \\u \end{smallmatrix} \right] = \left[\begin{smallmatrix} x_1 \\y \end{smallmatrix} \right] $ with
$\left[\begin{smallmatrix} v \\u \end{smallmatrix} \right] \in D({\mathcal A})$ gives that $v$ is given by (\ref{eq:v}) and $u$ satisfies (see also (\ref{eq:u}))
\[
  -G(1) u = y + {\mathcal B}_0^*(i_{{\mathbb V}} - {\mathcal A})^{-1} x_1
\]
From this and equation (\ref{eq:v}) it follows that $\mathcal E_I-{\mathcal A}$ is boundedly invertible when $G(1)$ is. It remains to show the opposite direction.

Assuming that $\mathcal E_I-{\mathcal A}$ is boundedly invertible gives that there is a unique and continuous mapping from $y \in {\mathbb U}$ to $\left[\begin{smallmatrix} v \\u \end{smallmatrix} \right] \in {{\mathbb X}}_1 \oplus {{\mathbb U}} $ such that $\left[\begin{smallmatrix} v \\u \end{smallmatrix} \right] \in D({\mathcal A})$, and
\[
  (\mathcal E_I-{\mathcal A}) \begin{bmatrix} v \\u \end{bmatrix}  = \begin{bmatrix} 0 \\y  \end{bmatrix}.
\]
Using Lemma \ref{L:14} we can solve this equation, and find $v=(i_{{\mathbb V}} - {\mathcal A})^{-1} {\mathcal B}_0u$ and $y = G(1) u$. This gives that there exists a continuous mapping from $y$ to $u$, and thus $G(1)$ is boundedly invertible.
\hfill\eproof
\medskip

From the above proof, we see that ${\mathcal D}_0$ being self-adjoint, which is a typical property in many applications, see e.g.\ \cite{EmmM13}, was only needed in one step. Alternatively, we could have assumed that ${\mathcal A}_0$ was self-adjoint. Of course both operators need to be dissipative.
Property (\ref{eq:G1>0}) is a special case of a general property which these systems likely have, namely that $G(s)$ is positive real, i.e., Re$\langle G(s) u, u\rangle \geq 0$ whenever $s \in {\mathbb C}^+$.

\begin{remark}\label{remsysnode}{\rm
In a recent paper, \cite{ReiS23}, it is shown that the formulation that we have discussed can also be formulated in terms of system nodes, see Definition 4.7.2 in \cite{Staf05}. The assumption on ${\mathcal A}$ as stated in Lemma~\ref{L:9} is the condition used in \cite{ZwGM16}.

%, see also the Maxwell class in %Subsection \ref{sec:4.3}. For system nodes, , we have the following result. For the proof we refer to Lemma 4.7.18 of that book.
%\begin{theorem}
%\label{T:11}
%Let ${\mathcal A}= \begin{bmatrix} \mathcal A_0\&\mathcal B_0\\ -C_0\&\mathcal D_0 \end{bmatrix}$ be an operator node which is dissipative. Then $sE_I-{\mathcal A}$ is boundedly invertible if and only if the transfer function is invertible at $s$.
%\end{theorem}
%Regarding this result, remarks can be made. By the assumption that
%${\mathcal A}=  \left[\begin{smallmatrix} \mathcal A_0\&\mathcal B_0\\ -C_0\&\mathcal D_0 \end{smallmatrix}\right]$ is dissipative, it follows easily that the main operator $\mathcal A_0x_1 :=\begin{bmatrix} \mathcal A_0\&\mathcal B_0\end{bmatrix}\left[\begin{smallmatrix} x_1\\ 0\end{smallmatrix}\right]$ is dissipative as well. Since by assumption of a system node $sI-\mathcal A_0$ has a bounded inverse,
%we see by the Lumer-Phillip theorem that this main %operator is the generator of a %contraction semigroup, and thus our %${\mathcal A}$ is a system node. %Furthermore, the system that we get %by removing the minus in the second %row of ${\mathcal A}$ is impedance %passive, see \cite{Staf02}.
%
}
\end{remark}

Using the property of the block structured pencil in \eqref{eq:50}
we have the following well-known inf-sup conditions when ${\mathcal D}_0 =0$.
\begin{theorem}\label{thm:divfree}
Consider the operator in \eqref{eq:50} and define ${\mathbb V}_0 \subset {\mathbb V}$ as ${{\mathbb V}}_0 =\ker \mathcal B_0^*$. Then
\begin{equation}
\label{eq:52}
    \inf_{0\neq v \in {\mathbb V}_0} \sup_{0\neq w \in {\mathbb V}_0}  \frac{\langle \mathcal A_0 v, w \rangle_{{\mathbb V}^*,{\mathbb V}}}{\|v\|\|w\|}  \geq \alpha >0 ,\quad   \inf_{0\neq v \in {\mathbb V}_0} \sup_{0\neq w \in {\mathbb V}_0}  \frac{\langle \mathcal A_0 w, v \rangle_{{\mathbb V}^*,{\mathbb V}}}{\|v\|\|w\|}  \geq \alpha >0,
 \end{equation}
\begin{equation}
\label{eq:53}
    \inf_{0\neq u \in U} \sup_{0\neq v \in {\mathbb V}} \frac{\langle \mathcal B_0u, v\rangle_{\mathcal {\mathbb V}^*,V}}{\|u\|_{{\mathbb U}}\|v\|_{{\mathbb V}}} \geq \gamma >0,
  \end{equation}
  and the pair $(\mathcal E_I,{\mathcal A})$ is regular.
 \end{theorem}
 \proof
 See e.g. \cite{Tem77}.
 \hfill\eproof
% \medskip
%
%We have seen that the regularity of $(\mathcal E,{\mathcal A} \mathcal Q)$ is  equivalent to that of  $(\mathcal E_I,{\mathcal A} )$. To show this, we have used the special form and the properties of $\mathcal Q$ and $\mathcal E$ as given in Assumption \ref{A1:HZ}. In the following lemma we study what the regularity of $(\mathcal E,{\mathcal A} \mathcal Q)$ implies on that of $(\mathcal E,\mathcal Q)$ when we do not have this special form of $\mathcal E$ and $\mathcal Q$, see \cite{MehMW18} for a corresponding result in the finite dimensional case.
%
%\begin{lemma}
%\label{L:10}
%Consider the adHDAE system~\eqref{eq:adae} with
%$(\mathcal E,{\mathcal A} \mathcal Q)$ regular
%for an $s_a \in {\mathbb C}^+$
%and  $\mathcal  E^*\mathcal Q$ coercive. Then
%\begin{enumerate}
%\item[a)]
%  $\ker(s\mathcal E + \mathcal Q) =\{0\}$ for %all $s \in {\mathbb C}^+$;
%\item[b)]
%  The range of $s\mathcal E+\mathcal Q$ is dense %in {\color{blue} $\mathbb X$}.
%\end{enumerate}
%\end{lemma}
%\proof
%a)\ Let $s \in {\mathbb C}^+$ and take $0\neq x %\in \ker(s\mathcal E + \mathcal Q)$. Then $s %\mathcal Ex= -\mathcal Qx$, and taking the inner %product with $\mathcal Ex$ gives
%\[
%  \overline{s} \|\mathcal Ex\|^2 = - \langle %\mathcal Ex, \mathcal Q x\rangle = -\langle %x,\mathcal E^*\mathcal Q x\rangle.
%\]
%Since $\mathcal E^*\mathcal Q\geq 0$ and since %Re$(s) >0$ we find that $\mathcal Ex=0$, and %thus $\mathcal Qx=0$. This implies that %$(s\mathcal E - {\mathcal A} \mathcal Q)x=0$, %yielding a contradiction to the regularity of %$(\mathcal E,{\mathcal A} \mathcal Q)$.
%\eproof

%%%%%%%%%%%%%
\section{Existence of solutions on the whole space}
\label{sec:3}

In this section we study the solution of adHDAEs of the form (\ref{eq:1HZ}).  Since $x_2$ is a constraint to $x_1$, we concentrate on the  solution theory for $x_1$ first. Our first result is based on  the extra assumption that the last row in (\ref{eq:1HZ}) does not impose a condition on $x_1$, i.e., for every $x_1$ there exists an $x_2$ such that this condition is satisfied. This implies that the algebraic equations impose no restriction on the state space ${\mathbb X}_1$. The case in which that may happen is studied in Section \ref{sec:5}.

We define the following reduced state space
\begin{equation}
\label{eq:2a}
  {\mathbb X}_{1,\mathcal E \mathcal Q} = {\mathbb X}_1 \mbox{ with inner product } \langle x_1, \tilde{x}_1 \rangle_{\mathcal E\mathcal Q} = \langle x_1,\mathcal E_1^* \mathcal Q_1 \tilde{x}_1 \rangle,
\end{equation}
where the second inner product is the standard inner product of ${\mathbb X}_1$. Since $\mathcal E_1^*\mathcal Q_1$ is coercive, the new norm is equivalent to the original one.
\begin{theorem}
\label{T:2HZ}
Consider a adHDAE system of the form~\eqref{eq:1HZ} with
$(\mathcal E,{\mathcal A} \mathcal Q)$ regular satisfying
Assumption~\ref{A1:HZ} and  assume that whenever $\left[\begin{smallmatrix} 0 \\ x_2 \end{smallmatrix}\right] \in D({\mathcal A})$ is such that ${\mathcal A}_2  \left[\begin{smallmatrix} 0 \\ x_2 \end{smallmatrix}\right] =0$, then $x_2=0$.

Under these assumptions, the operator $\mathcal A_{red} : D(\mathcal A_{red}) \subset {\mathbb X}_{1,\mathcal E\mathcal Q} \rightarrow {\mathbb X}_{1,\mathcal E\mathcal Q}$ generates a contraction semigroup on the reduced space ${\mathbb X}_{1,\mathcal E \mathcal Q}$, where the domain $D(\mathcal A_{red})$ is defined as
\begin{equation}
\label{eq:3:HZ}
  D(\mathcal A_{red}) = \{ x_1 \in {\mathbb X}_{1,\mathcal E\mathcal Q} \mid \exists\ x_2 \in\mathcal  X_2 \mbox{ such that } \begin{bmatrix} \mathcal Q_1 x_1 \\ \mathcal Q_2 x_2 \end{bmatrix} \in D({\mathcal A}) \mbox{ and } {\mathcal A}_2  \begin{bmatrix} \mathcal Q_1x_1 \\  \mathcal Q_2 x_2 \end{bmatrix} =0 \} ,
\end{equation}
and for $x_{1} \in D(\mathcal A_{red})$ the action of $\mathcal A_{red}$ is defined as
\begin{equation}
\label{eq:4:HZ}
  \mathcal A_{red}x_1 = \mathcal E_1^{-1} {\mathcal A}_1  \begin{bmatrix} \mathcal Q_1x_1 \\ \mathcal Q_2x_2 \end{bmatrix}.
\end{equation}
\end{theorem}
\proof
First we have to prove that $\mathcal A_{red}$ is well-defined. So if for a given $x_1\in D(\mathcal A_{red})$ we have that $x_2$ and $\tilde{x}_2$ are such that the condition of the domain are satisfied for  $\left[\begin{smallmatrix} x_1 \\ x_2 \end{smallmatrix}\right]$ and $\left[\begin{smallmatrix} x_1 \\ \tilde{x}_2 \end{smallmatrix} \right]$, then by the linearity of ${\mathcal A}_2$, we have that
\[
  {\mathcal A}_2  \begin{bmatrix} 0 \\ \mathcal Q_2(x_2 - \tilde{x}_2)\end{bmatrix} =0.
\]
By assumption, this implies that $\mathcal Q_2(x_2-\tilde{x}_2) =0$, and since $\mathcal Q_2$ is invertible, then $x_2-\tilde{x}_2 =0$. So there exists at most one $x_2$ to every $x_1 \in D(\mathcal A_{red})$, and hence $\mathcal A_{red}$ is well-defined.

 Using that $\mathcal E_1^*\mathcal Q_1 = (\mathcal E_1^*\mathcal Q_1 )^*$, we have
\begin{align*}
  \langle \mathcal A_{red}x_1, x_1\rangle_{\mathcal E\mathcal Q} + \langle x_1, \mathcal A_{red}x_1, \rangle_{\mathcal E\mathcal Q} =&
  \langle \mathcal A_{red}x_1, \mathcal E_1^*\mathcal Q_1 x_1\rangle + \langle  \mathcal E_1^* \mathcal Q_1 x_1,  \mathcal A_{red}x_1\rangle\\
   =&\ \langle \mathcal E_1^{-1} { {\mathcal A}}_1  \begin{bmatrix} \mathcal Q_1x_1 \\ \mathcal Q_2x_2 \end{bmatrix}, \mathcal E_1^*\mathcal Q_1 x_1\rangle + \langle \mathcal E_1^* \mathcal Q_1 x_1, \mathcal E_1^{-1} {\mathcal A}_1  \begin{bmatrix} \mathcal Q_1x_1 \\ \mathcal Q_2x_2 \end{bmatrix} \rangle\\
  =&\ \langle {\mathcal A}  \begin{bmatrix} \mathcal Q_1x_1  \\ \mathcal Q_2 x_2 \end{bmatrix}, \begin{bmatrix} \mathcal Q_1x_1 \\ \mathcal Q_2 x_2 \end{bmatrix}\rangle + \langle \begin{bmatrix} \mathcal Q_1 x_1 \\ \mathcal Q_2 x_2 \end{bmatrix}, {\mathcal A}  \begin{bmatrix} \mathcal Q_1 x_1  \\ \mathcal Q_2 x_2 \end{bmatrix} \rangle\\
  \leq &\ 0,
  \end{align*}
  where we have used ${\mathcal A}_2 \left[ \begin{smallmatrix} \mathcal Q_1 x_1 \\ \mathcal Q_2 x_2 \end{smallmatrix} \right] =0$ and the dissipativity of ${\mathcal A}$.
  Hence $\mathcal A_{red}$ is dissipative.

 Now we show that $sI-\mathcal A_{red}$ is onto for $s \in {\mathbb C}^+$.
  Given $\left[\begin{smallmatrix} y_1 \\ 0 \end{smallmatrix} \right] \in {\mathbb X}$. Then by assumption, see also Corollary \ref{C:8}, we know that there exists a $\left[\begin{smallmatrix} x_1 \\ x_2  \end{smallmatrix} \right]\in D({\mathcal A}\mathcal Q)$ such that
  \begin{equation}
  \label{eq:5HZ}
    \begin{bmatrix} \mathcal E_1 y_1 \\ 0 \end{bmatrix} = (s\mathcal E - {\mathcal A}\mathcal Q) \begin{bmatrix} x_1 \\ x_2 \end{bmatrix} = s \begin{bmatrix} \mathcal E_1 x_1 \\ 0 \end{bmatrix} - {\mathcal A}\begin{bmatrix} \mathcal Q_1 x_1 \\ \mathcal Q_2 x_2 \end{bmatrix} .
  \end{equation}
  The last row of this expression gives that
  \[
    {\mathcal A}_2  \begin{bmatrix} \mathcal Q_1 x_1 \\ \mathcal Q_2 x_2 \end{bmatrix} =0
  \]
  and so $x_1 \in D(\mathcal A_{red})$. The top row of \eqref{eq:5HZ} gives
  \[
    s\mathcal E_1 x_1- {\mathcal A}_1 \begin{bmatrix} \mathcal Q_1 x_1 \\ \mathcal Q_2 x_2 \end{bmatrix} = \mathcal E_1 y_1
  \]
  or equivalently (using \eqref{eq:3:HZ} and that $\mathcal E_1$ is boundedly invertible)
    $(sI -\mathcal A_{red})x_1 =  y_1$. This gives that $sI-\mathcal A_{red}$ is surjective for  $s \in {\mathbb C}^+$. By the Lumer-Phillips Theorem, see e.g.~\cite{Staf05}, we conclude that $\mathcal A_{red}$ generates a contraction semigroup on ${\mathbb X}_1$.
  \hfill
\eproof
\medskip

In the proof of Theorem~\ref{T:2HZ}, we did not use the  regularity of the pair $(\mathcal E,{\mathcal A}\mathcal Q)$ to show that $\mathcal A_{red}$ is well-defined and dissipative. It was used only to prove the surjectivity of $sI-\mathcal A_{red}$. Since the latter is the property we want for $\mathcal A_{red}$, we can ask if our regularity assumption is not too strong. The following lemma shows that under a mild condition the two properties are equivalent.
\begin{lemma}
\label{L:11}
 Let the first three conditions of Assumption \ref{A1:HZ} hold. Furthermore, we
 assume that whenever $\left[\begin{smallmatrix} 0 \\ x_2 \end{smallmatrix}\right] \in D({\mathcal A})$ is such that ${\mathcal A}_2  \left[\begin{smallmatrix} 0 \\ x_2 \end{smallmatrix}\right] =0$, then $x_2=0$. Consider the operator $\mathcal A_{red}$ on the domain (\ref{eq:3:HZ}) with the action (\ref{eq:4:HZ}). Then the following are equivalent:
 \begin{itemize}
 \item [a)]
   The pair $(\mathcal E,{\mathcal A} \mathcal Q)$ is regular.
 \item [b)]
   $\mathcal A$ is closed, ${\mathcal A}_2 : D(\mathcal A) \mapsto {\mathbb X}_2$ is surjective, and there exists an $s \in {\mathbb C}^+$ such $sI-\mathcal A_{red}$ is boundedly invertible;
% \item
%   ${\mathcal A}_2 : D({\mathcal A}) \mapsto X_2$ is right-ivertible and there exists an $s \in {\mathbb C}^+$ such $sI-A$ is boundedly invertible;
\item [c)]
   $\mathcal A$ is closed, ${\mathcal A}_2 : D(\mathcal A) \mapsto {\mathbb X}_2$ is surjective, and there exists an $s \in {\mathbb C}^+$ such $sI-\mathcal A_{red}$ is surjective;
\item [d)]
    $\mathcal A$ is closed, ${\mathcal A}_2: D(\mathcal A) \mapsto {\mathbb X}_2$ is surjective and $\mathcal A_{red}$ is maximally dissipative.
\end{itemize}
\end{lemma}
\proof
By Corollary \ref{C:8} we only have to prove the equivalences  for the pair $(\mathcal E_I, \mathcal A)$. From the assumptions it follows that $\mathcal A_{red}$ is well-defined and dissipative, see the proof of Theorem \ref{T:2HZ}.
\smallskip

\noindent $a) \Rightarrow b)$: By Lemma \ref{L2.1a} ${\mathcal A}$ is maximally dissipative and so it is closed. The last part follows from Theorem \ref{T:2HZ}, since if $\mathcal A_{red}$ generates a contraction semigroup, then $sI- \mathcal A_{red}$ is boundedly invertible for all $s \in {\mathbb C}^+$, see also Section \ref{sec:app}.
It remains to show that ${\mathcal A}_2$ is surjective. Since $(\mathcal E_I,\mathcal A)$ is regular, $s\mathcal E_I - \mathcal A$ is surjective. This %Looking at the second line
immediately implies that ${\mathcal A}_2$ is surjective.
\smallskip

\noindent $b) \Leftrightarrow d)$: This equivalence follows from the fact that a dissipative operator $\mathcal A_{red}$ is maximally dissipative if and only if $sI-\mathcal A_{red}$ is surjective for some/all $s \in {\mathbb C}^+$, see Lemma \ref{L:A4} in the appendix.
\smallskip

\noindent $b) \Rightarrow c)$: This holds trivially, since when $sI-\mathcal A_{red}$ is boundedly invertible, its range equals ${\mathbb X}_1$ and the operator is closed.
\smallskip

\noindent $c) \Rightarrow a)$:
We begin by showing that $s \mathcal E_I - \mathcal A$ is injective.
Let $\left[\begin{smallmatrix} x_1 \\ x_2 \end{smallmatrix}\right] \in D(\mathcal A)$ be such that  $(s \mathcal E_I - \mathcal A)\left[ \begin{smallmatrix} x_1 \\ x_2 \end{smallmatrix} \right]=0$. So
\[
  0= \left\langle \begin{bmatrix} x_1 \\x_2 \end{bmatrix} , (s \mathcal E_I - \mathcal A) \begin{bmatrix} x_1 \\ x_2 \end{bmatrix}\right\rangle = \overline{s} \|x_1\|^2 - \left\langle \begin{bmatrix} x_1 \\x_2 \end{bmatrix} , \mathcal A \begin{bmatrix} x_1 \\ x_2 \end{bmatrix}\right\rangle
\]
where the last term has nonnegative real part, since $\mathcal A$ is dissipative. If $x_1\neq 0$, the first part would have strictly positive real part, which contradicts the equality and thus $x_1 =0$. This gives that
\[
  0=(s \mathcal E_I - \mathcal A) \begin{bmatrix} x_1 \\ x_2 \end{bmatrix} = (s\mathcal  E_I - \mathcal A) \begin{bmatrix} 0 \\ x_2 \end{bmatrix}= - \mathcal A\begin{bmatrix} 0 \\ x_2 \end{bmatrix},
\]
and in particular ${\mathcal A}_2\left[ \begin{smallmatrix} 0 \\ x_2 \end{smallmatrix} \right]=0$, which by assumption gives $x_2=0$. Thus we have shown that $s \mathcal E_I - \mathcal A$ is injective.

% then the second row of this equation gives that $x_1 \in D(A)$ and the first row gives that $(sI-A)x_1 =0$. Since $A$ is maximally dissipative and $s \in {\mathbb C}^+$, we know that $(sI-A)$ is injective. Thus $x_1=0$, and our assumption implies that $x_2=0$, which proves the injectivity of $s E_I - {\mathcal A}$.

Next we prove the surjectivity.
Let $ \left[\begin{smallmatrix} y_1 \\ y_2 \end{smallmatrix}\right] \in {\mathbb X}$ be given. By the surjectivity of ${\mathcal A}_2$ there exists an $\left[\begin{smallmatrix} \tilde{x}_1 \\ \tilde{x}_2 \end{smallmatrix}\right] \in D(\mathcal A)$ such that ${\mathcal A}_2 \left[\begin{smallmatrix} \tilde{x}_1 \\ \tilde{x}_2 \end{smallmatrix}\right] = -y_2$. Defining
\[
  \tilde{y}_1 = s \tilde{x}_1 - {\mathcal A}_1\begin{bmatrix} \tilde{x}_1 \\ \tilde{x}_2 \end{bmatrix},
\]
we obtain
\begin{equation}
\label{eq:16}
   \begin{bmatrix} \tilde{y}_1 \\ y_2 \end{bmatrix} = (s\mathcal E_I - \mathcal A) \begin{bmatrix} \tilde{x}_1 \\ \tilde{x}_2 \end{bmatrix}.
\end{equation}
Since $(sI-\mathcal A_{red})$ is surjective, there exists $x_1$ such that $y_1 - \tilde{y}_1 = (sI-\mathcal A_{red}) x_1$. By the definition of $\mathcal A_{red}$ this means that there exists  $x_2$ such
\begin{equation}
\label{eq:17}
   \begin{bmatrix} y_1 - \tilde{y}_1 \\ 0 \end{bmatrix} = (s\mathcal E_I - \mathcal A) \begin{bmatrix} x_1 \\ x_2 \end{bmatrix}.
\end{equation}
Adding (\ref{eq:16}) and (\ref{eq:17}) gives that $\left[\begin{smallmatrix} x_1 +  \tilde{x}_1 \\ x_2 + \tilde{x}_2\end{smallmatrix}\right]$ is mapped to $\left[ \begin{smallmatrix} y_1 \\ y_2 \end{smallmatrix} \right]$ by $s\mathcal E_I - \mathcal A$, and we conclude that  $s\mathcal E_I - \mathcal A$ is surjective.

Since $\mathcal A - s\mathcal E_I $ is injective and surjective and since $\mathcal A$ is closed, so is $\mathcal A - s\mathcal E_I $ and thus injectivity and surjectivity implies bounded invertibility, see e.g.\ [Corollary A.3.50]\cite{CuZw95}.
\hfill\eproof
\medskip

In the part of the  proof of Lemma~\ref{L:11} that c) implies d) we show that we only needed the closedness of $\mathcal A$ to conclude from the injectivity plus surjectivity the bounded invertibility of $s\mathcal E_I-\mathcal A_{red}$. Since $\mathcal A - s\mathcal E_I$ is dissipative, it is closable, see Section \ref{sec:app}. The closure is obviously still surjective, and thus it remains to show that it is injective. The $s\mathcal E_I$ term gives that any element in the kernel should have $x_1=0$. If the following implication holds: ${\mathcal A}_2 \left[\begin{smallmatrix} 0 \\ x_{2,n} \end{smallmatrix}\right] \rightarrow 0$, $x_{2,n} \rightarrow x_2 \Rightarrow x_2 =0$, then the closure is injective.

Theorem \ref{T:2HZ} implies that for all $x_{1,0}\in {\mathbb X}_1$ there exists a unique (weak or mild) solution of
\begin{equation}
\label{eq:5aHZ}
  \dot{x}_1(t) = \mathcal A_{red}x_1(t), \quad x_1(0)=x_{1,0}.
\end{equation}
However, this is only a part of the solution of the adHDAE~\eqref{eq:adae}.
For a classical solution, we have that $x_1(t) \in D(\mathcal A_{red})$, and so since a given $x_1$ yields a unique $x_2$, we also find a (unique) $x_2(t)$ such that the bottom equation of \eqref{eq:adae} is satisfied. In general the equation for $x_2(t)$ will not exist for all mild solutions, as is shown on basis of Example \ref{E3:HZ}, see the text following that example.
%By assumption on the domain of $A$, we see that for a classical solution of   (\ref{eq:5aHZ}), ${\mathcal A_2}\left[\begin{smallmatrix} x_1(t) \\ 0 \\ x_3(t) \end{smallmatrix} \right]$ is well defined, and so
%\[
%  \dot{x}_2(t) = {\mathcal A_2}\begin{bmatrix} x_1(t) \\ 0 \\ x_3(t) \end{bmatrix}, \quad x_2(0)=x_{20}
%\]
%is well-defined, and has a unique solution.

\section{Applications}
\label{sec:4}

In this section we study several well-known classes of systems, and show that they can be seen as examples of Theorem \ref{T:2HZ}. We start with the class of abstract port-Hamiltonian systems.

\subsection{Abstract port-Hamiltonian  systems on a 1D spatial domain.}

In this section we discuss  port-Hamiltonian systems and we begin with a very general setup.
Let $L^2,H^1,H^2$ denote the usual Hilbert spaces of square integrable functions, and associated Sobolev spaces.
On  $L^2((0,1); {\mathbb R}^n)$ we consider the operator
\begin{equation}
\label{eq:19}
  {\mathcal A}x = P_1 \frac{d}{d\zeta} x + G_0(\zeta)  x
\end{equation}
with domain
\begin{equation}
\label{eq:20}
  D({\mathcal A}) =\{ x \in H^1((0,1);{\mathbb R}^n) \mid W_B \begin{bmatrix} x(1)\\ x(0) \end{bmatrix} =0\}.
\end{equation}
Here $P_1$ is a real constant, symmetric, invertible matrix, and $G_0: [0,1] \mapsto {\mathbb C}^{n\times n}$ is Lipschitz continuous satisfying $G_0(\zeta) + G_0(\zeta)^* \leq 0$ for all $\zeta \in [0,1]$. Furthermore,  $W_B$ is a (constant) $n \times 2n$ matrix of full rank. From \cite{JacZ12, GoZM05} or \cite{JaMZ15} it is known that ${\mathcal A}$ is maximally dissipative if and only if
\begin{equation}
\label{eq:21}
  v^T P_1v - w^T P_1 w \leq 0 \quad \mbox{for all } v,w \in {\mathbb R}^n\ \mbox{satisfying }\ W_B \begin{bmatrix} v \\ w \end{bmatrix} =0.
\end{equation}

For this class of systems we show that if the conditions of Lemma~\ref{L:11} hold, then the associated operator pair is regular.
\begin{theorem}
\label{T:12} Consider the adHDAE system ~\eqref{eq:adae}, where the operator ${\mathcal A}$ has its domain defined in equations (\ref{eq:19}) and (\ref{eq:20}). Furthermore, assume that (\ref{eq:21}) holds. Let $n_1+n_2 = n$ and write ${\mathbb X}= L^2((0,1); {\mathbb R}^n)= L^2((0,1); {\mathbb R}^{n_1}) \oplus L^2((0,1); {\mathbb R}^{n_2}) =: {\mathbb X}_1 \oplus {\mathbb X}_2$.

If the subset ${\mathbb V}_0 := \{x_2 \in { \mathbb X}_2 \mid \left[\begin{smallmatrix} 0 \\ x_2 \end{smallmatrix}\right] \in D({\mathcal A})\ \mbox {and }  {\mathcal A}_2  \left[\begin{smallmatrix} 0 \\ x_2 \end{smallmatrix}\right] =0\} $ contains only the zero element, then $(\mathcal E_I, {\mathcal A})$ is regular, i.e. Assumption \ref{A1:HZ} is satisfied for this class of systems.

If the $n_2 \times n_2$ right lower block of $P_1$ is zero and the corresponding block of $G_0(\zeta)$ is invertible for almost all $\zeta \in [0,1]$, then ${\mathbb V}_0 = \{0\}$.
\end{theorem}
\proof
We have to show that $s\mathcal E_I - {\mathcal A}$ is boundedly invertible. To do so we introduce some notation. We split the matrices according to the dimensions $n_1$ and $n_2$, i.e.
\begin{equation}
\label{eq:24}
  P_1 = \begin{bmatrix} P_{1,11} & P_{1,12} \\ P_{1,21} & P_{1,22} \end{bmatrix} \quad G_0 = \begin{bmatrix} G_{0,11} & G_{0,12} \\ G_{0,21} & G_{0,22} \end{bmatrix}.
\end{equation}
The equation $y = (\mathcal E_I - {\mathcal A})x$ can then equivalently be written as
\[
  P_1 \frac{dx}{d\zeta}(\zeta) = \begin{bmatrix} sI - G_{0,11}(\zeta)  & -G_{0,12}(\zeta) \\ -G_{0,21}(\zeta) & -G_{0,22}(\zeta) \end{bmatrix} \begin{bmatrix} x_1(\zeta) \\ x_2(\zeta) \end{bmatrix} - \begin{bmatrix} y_1(\zeta) \\ y_2(\zeta) \end{bmatrix} =:- G_s(\zeta) x(\zeta) - y(\zeta).
\]
Since $P_1$ is invertible, this is an implicit linear ordinary differential equation in $\zeta$ with variable coefficients. Since $G_s$ is Lipschitz continuous, for every initial condition $x(0)$ this equation has a unique solution, which we write as
\[
  x(\zeta)= \begin{bmatrix} x_1(\zeta) \\ x_2(\zeta) \end{bmatrix} = \Psi(\zeta,0)  \begin{bmatrix} x_1(0) \\ x_2(0) \end{bmatrix} + \int_0^{\zeta} \Psi(\zeta,\tau)  \begin{bmatrix} -y_1(\tau) \\ -y_2 (\tau) \end{bmatrix} d\tau,
\]
where $\Psi$ is the fundamental solution matrix of the homogeneous system. If we would have that this solution is in the domain of ${\mathcal A}$, then this part of the proof is complete. For this we need that
\[
  W_B \begin{bmatrix} x(1) \\ x(0)\end{bmatrix}=0,
\]
or equivalently
\[
  (W_{B,1} \Psi(1,0) + W_{B,2} ) x(0) = W_{B,1} \int_0^{1} \Psi(1,\tau)  \begin{bmatrix} y_1(\tau) \\ y_2 (\tau) \end{bmatrix} d\tau.
\]
If the (constant) matrix in front of $x(0)$ is invertible, then we can find a unique $x(0)$, and so the solution $x(\cdot)$ is uniquely determined. If this matrix is not invertible, then choose $0\neq x(0)$ in its kernel, i.e.,
\[
  (W_{B,1} \Psi(1,0) + W_{B,2} ) x(0) = 0.
\]
This implies that
\[
  x(\zeta): =  \Psi(\zeta,0) x(0)
\]
is a solution of  $P_1 \frac{dx}{d\zeta}(\zeta)  =- G_s(\zeta) x(\zeta)$ which satisfies the boundary condition, i.e.\ $W_B \left[\begin{smallmatrix} x(1) \\ x(0)\end{smallmatrix}\right]=0$. By the definition of $G_s$, this means that $x$ satisfies $(s\mathcal E_I - {\mathcal A})x =0$, implying that $s\mathcal E_I - {\mathcal A}$ is not injective. By Lemma \ref{L:2.2-new} this means that the first component of $x$ is zero, and the second component satisfies ${\mathcal A}  \left[\begin{smallmatrix} 0 \\ x_2 \end{smallmatrix}\right]=0$. By our assumption this gives that $x_2=0$.  Concluding, we see that $(W_{B,1} \Psi(1,0) + W_{B,2})$ must be injective and thus surjective, implying that the pair $(\mathcal E_I, {\mathcal A})$ is regular.

To prove the last statement, considering (\ref{eq:24}), we get
\begin{equation}
\label{eq:25}
  {\mathcal A_2}\begin{bmatrix} 0 \\ x_2 \end{bmatrix} = \begin{bmatrix} P_{1,12} \frac{dx_2}{d\zeta} + G_{0,12} x_2  \\ G_{0,22}x_2 \end{bmatrix},
\end{equation}
where we have used the condition on $P_1$.
Using the invertibility of $G_{0,22}$ this can only be zero, when $x_2=0$.
\hfill$\Box$
\medskip

We see from (\ref{eq:25}) that even when $G_{0,22}$ is singular, this equation could have only the zero function as its solution. Since $P_1$ is invertible and $P_{1,22}=0$, $P_{1,12}$ is of full rank and so ${\mathcal A}  \left[\begin{smallmatrix} 0 \\ x_2 \end{smallmatrix}\right]=0$ is (partly) a differential equation. This means that it will also depend on the boundary conditions, imposed by $W_B$, whether $x_2=0$ is its only solution.

%Furthermore, ${\mathcal H}$ is a coercive operator on $L^2((0,1); {\mathbb R}^n)$.
There are several applications of Theorem~\ref{T:12}.
\begin{example}
\label{E3:HZ}{\rm
Choose
\[
  P_1= \begin{bmatrix} 0 &1 \\ 1 & 0 \end{bmatrix}
  \mbox{ and } G_0= \begin{bmatrix} -g_0%(\zeta)
  & 0\\0  & -r
  %(\zeta)
  \end{bmatrix},
\]
where $g_0$ is a bounded function and $r$ is a bounded and invertible function, and moreover both satisfy that their real part is non-negative. Furthermore, choose
\[
  \mathcal E = \begin{bmatrix} e_1%(\zeta)
  &0 \\ 0 & 0 \end{bmatrix}  \mbox{ and } \mathcal Q = \begin{bmatrix} 1 & 0\\0  & q_2%(\zeta)
  \end{bmatrix},
\]
where $e_1, q_2$ are positive, bounded, and invertible functions.
We take a full rank $W_B$ such that (\ref{eq:21}) holds, and thus ${\mathcal A}$ is dissipative.
For $n_1=n_2 =1$, it is not hard to see that the assumptions of Theorem \ref{T:12} are satisfied. Hence $(\mathcal E_I, {\mathcal A})$ is regular, and so is $(\mathcal E, {\mathcal A \mathcal Q})$, see Corollary \ref{C:8}.

Applying Theorem \ref{T:2HZ}, by \eqref{eq:4:HZ} we find that
\[
  \mathcal A_{red}x_1 = \frac{1}{e_1} \left[ \frac{d (q_2x_2)}{d\zeta} - g_0 x_1 \right]  \mbox{ with } \frac{dx_1}{d\zeta} - r q_2 x_2 =0
\]
or equivalently
\begin{equation}
\label{eq:26}
  \left(\mathcal A_{red}x_1\right)(\zeta)  = \frac{1}{e_1(\zeta)} \left[ \frac{d}{d\zeta} \left( \frac{1}{r(\zeta)}  \frac{dx_1}{d\zeta} (\zeta) \right) - g_0(\zeta) x_1(\zeta) \right]
\end{equation}
with domain
\[
  D(\mathcal A_{red}) =\{ x_1 \in H^1(0,1) \mid  \frac{1}{r}  \frac{dx_1}{d\zeta} \in H^1(0,1) \mbox{ and }
  W_B  \begin{bmatrix} x_1(1) \\ \frac{1}{r(1)} \frac{d x_1}{d \zeta}(1)\\  x_1(0) \\ \frac{1}{r(0)} \frac{dx_1}{d \zeta}(0) \end{bmatrix} =0\}.
\]
If we assume that $r$ is real-valued, then this operator is (minus) a Sturm-Liouville operator, see \cite[p.\ 82]{CuZw95}, with the exception of the sign condition on the last term. This sign condition is a consequence of the fact that we want dissipative operators, whereas that is not imposed in general for Sturm-Liouville operators.

Sturm-Liouville operators always come with a specific set of boundary conditions. We can  obtain these boundary conditions by choosing the right $W_B$, e.g.\ with the real matrix
\[
  W_B = \begin{bmatrix} \alpha_1 & \beta_1 & 0 &0 \\ 0 & 0 & \alpha_2 & \beta_2 \end{bmatrix},
\]
with $\alpha_1^2+ \beta_1^2 > 0$ and $\alpha_2^2+ \beta_2^2 > 0$. This matrix satisfies (\ref{eq:21}) if and only if $\alpha_1 \beta_1 \geq 0$ and  $\alpha_2 \beta_2 \geq 0$. Again these conditions are a consequence of the fact that we want dissipative operators, whereas that is not imposed for Sturm-Liouville operators.

The diffusion/heat equation is the most well-known Sturm-Liouville operator. If we choose $r=e_1 =1$, $\alpha_1=\alpha_2=1$ and $\beta_1=\beta_2=0$, then the PDE $\dot{x}(t) = {\mathcal A} x(t)$  corresponds to an undamped vibrating string which is fixed at the boundary, whereas the PDE $\dot{x}_1(t) = \mathcal A_{red} x_1(t)$ corresponds to the diffusion/heat equation with temperature zero at the boundary.

So we have constructed the heat equation out of the wave equation.
If we choose $r=-i$, then the PDE $\dot{x}_1(t) = \mathcal A_{red} x_1(t)$ corresponds to the 1-D Schr\"{o}dinger equation.
}
\end{example}

Now we return to the comments made below equation (\ref{eq:5aHZ}). We once more look at the differential equation we found for $x_1$ in Example \ref{E3:HZ}. For simplicity, we assume that $e_1=1$, $r=-i$, and so $x_1$ satisfies the standard Schr\"{o}dinger equation. It is well-known that for an arbitrary initial condition in $L^2(0,1)$ this will have a unique weak/mild solution. However, for an initial condition in $L^2(0,1)$ the solution will not get smoother, and so $x_2(t) = i \frac{\partial}{\partial \zeta}x_1(t)$ will in general not lie in the state space.

Next we apply Theorem \ref{T:12} to  show that the equations for an Euler-Bernoulli beam can be constructed out of two wave equations.
\begin{example}
\label{E:14} {\rm
Consider ${\mathcal A}$ of equation (\ref{eq:19}) with
\[
  P_1= \begin{bmatrix} 0&0& 0 &1 \\ 0& 0 & 1& 0\\ 0&1& 0  & 0 \\ 1& 0& 0& 0\end{bmatrix}, \quad G_0= \begin{bmatrix} 0&0& 0 &0 \\ 0& 0 & 0& 0\\ 0&0& 0  & 1 \\ 0& 0& -1 & 0 \end{bmatrix}
\]
and assume that $W_B$ is a full rank $4\times 8$ matrix satisfying (\ref{eq:21}).
%and
%\[
%  W_B = \begin{bmatrix} 0&0& 1 &0 & 0 & 0 & 0 &0 \\ 0&0& 0 &1 & 0 & 0 & 0 &0 \\ 0&0& 0 &0&  1 & 0 & 0 & 0 \\ 0&0& 0 &0&  0 & 1 & 0 & 0  \end{bmatrix}.
%\]
We take $n_1=n_2=2$, $\mathcal E=\mathrm{diag}(\mathcal E_1, 0)$, $\mathcal Q= \mathrm{diag}(\mathcal Q_1, \mathcal Q_2)$, with $\mathcal E_1, \mathcal Q_1,\mathcal Q_2$ strictly positive $(2\times 2)$-matrix valued bounded functions. It is easy to see that the conditions of Theorem \ref{T:12} are satisfied, and so are those of Assumption~\ref{A1:HZ}.

The operator $\mathcal A_{red}$ from Theorem \ref{T:2HZ} then becomes
\begin{equation}
\label{eq:27}
  \mathcal A_{red}x_1 =  \mathcal E_1^{-1} \begin{bmatrix}  0 &1 \\ 1 &0 \end{bmatrix} \frac{d}{d\zeta}\left(  \begin{bmatrix}  0 &1 \\ -1 &0 \end{bmatrix}  \begin{bmatrix}  0 &1 \\ 1 &0 \end{bmatrix} \frac{d ({\mathcal Q}_1 x_1)}{d\zeta} \right) = \mathcal E_1^{-1}  \begin{bmatrix}  0 &-1 \\ 1 &0 \end{bmatrix}  \frac{d^2({\mathcal Q}_1 x_1)}{d\zeta^2}.
\end{equation}
For $\mathcal E_1=\left[ \begin{smallmatrix} \rho & 0 \\ 0 & 1 \end{smallmatrix} \right]$, $\mathcal Q_1=\left[ \begin{smallmatrix} q_1 & 0 \\ 0 & q_2 \end{smallmatrix} \right]$, and  $x_1:=\left[\begin{smallmatrix} x_{1,1} \\ x_{1,2}\end{smallmatrix}\right]$ the associated PDE $\dot{x}_1(t) = \mathcal A_{red} x_1(t)$ takes the form
\[
  \frac{\partial x_{1,1}}{\partial t} = - \frac{1}{\rho} \frac{\partial^2 (q_2 x_{1,2})}{\partial \zeta^2} \quad \mbox{ and }  \quad \frac{\partial x_{1,2}}{\partial t} = \frac{\partial^2 (q_1 x_{1,1})}{\partial \zeta^2},
\]
or in the variable $x_{1,1}$
\[
  \rho(\zeta) \frac{\partial^2 x_{1,1}}{\partial t^2}(\zeta,t) = -\frac{\partial^2}{\partial \zeta^2}\left[ q_2(\zeta) \frac{\partial^2 (q_1(\zeta) x_{1,1})}{\partial \zeta^2}(\zeta,t)\right].
\]
For $\rho$ the mass density, $q_1=1$, and $q_2 = E\,K$, with $E$ the elastic modulus and
$K$ the second moment of area of the beam's cross section, this is the well-known Euler-Bernoulli beam model.

We note that $P_1$ can be seen to correspond to two wave equations, namely one in the variables $x_{1,1}$ and $x_{1,4}$ and the other in the variables $x_{1,2}$ and $x_{1,3}$. So we can construct the beam equation out of two wave equations.}
\end{example}

In Example~\ref{E:14} we have discussed the construction of a second order port-Hamiltonian system from a first order one,
%have seen that out of a port-Hamiltonian system with a $P_1$ we can build a port-Hamiltonian system with a $P_2$,
see also \cite{GoZM05}. In the following lemma we will do this generally, and also pay attention to the boundary conditions.

\begin{lemma}\label{L:P1P2}
 Consider an operator $\mathcal A_{red}$ on $L^2((0,1);{\mathbb R}^{n_1})$ of the  form
  \[
      \mathcal A_{red} x_1 =P_2 \frac{d^2x_1}{d\zeta^2} +  P_{1,1}\frac{dx_1}{d\zeta} + P_0 x_1
  \]
  with domain
  \[
    D(\mathcal A_{red}) =\{ x_1 \in H^2((0,1); {\mathbb R}^{n_1}) \mid \tilde{W}_B \left[ \begin{smallmatrix} x_1(1) \\  \frac{dx_1}{d\zeta} (1) \\ x_1(0) \\ \frac{dx_1}{d\zeta} (0) \end{smallmatrix} \right] =0\},
  \]
where we assume that, for the $n_1\times n_1$ coefficient matrices, we have that $P_2$, $P_0$ are skew-symmetric, $P_{1,1}$ is symmetric and $P_2$ is invertible. Furthermore, $\tilde{W}_B$ is a full rank $2 n_1\times 4 n_1$-matrix.

If $\mathcal A_{red}$ is a generator of a contraction semigroup on $L^2((0,1);{\mathbb R}^{n_1})$, then it can be constructed via Theorem~\ref{T:2HZ}
from an ${\mathcal A}$
as in (\ref{eq:19}) and (\ref{eq:20}).
\end{lemma}
\proof
We recall from \cite{GoZM05} that under the conditions on the coefficient matrices, $\mathcal A_{red}$ generates a contraction semigroup if and only if
\begin{equation}
\label{eq:28}
  \begin{bmatrix} v_{1,1}^T &  v_{1,2}^T  \end{bmatrix} \begin{bmatrix} P_{1,1} & P_2 \\ - P_2 & 0 \end{bmatrix}  \begin{bmatrix} v_{1,1} \\  v_{1,2} \end{bmatrix} -
   \begin{bmatrix} w_{1,1}^T &  w_{1,2}^T  \end{bmatrix} \begin{bmatrix} P_{1,1} & P_2 \\ - P_2 & 0 \end{bmatrix}  \begin{bmatrix} w_{1,1} \\  w_{1,2}  \end{bmatrix}
  \leq 0
\end{equation}
for $v_{1,1},v_{1,2}, w_{1,1}, w_{1,2} \in {\mathbb R}^{n_1}$ such that $\tilde{W}_B \left[ \begin{smallmatrix} v_{1,1} \\ v_{1,2} \\ w_{1,1}\\ w_{1,2} \end{smallmatrix} \right] =0$.

With the matrices $P_2, P_{1,1}$, and $P_0$ we choose the following $n\times n= 2n_1 \times 2n_1$-matrices in (\ref{eq:19})
\begin{equation}
\label{eq:28a}
  P_1 = \begin{bmatrix} P_{1,1} & I \\ I & 0 \end{bmatrix} \mbox{ and } G_0 = \begin{bmatrix} P_{0} & 0 \\ 0 & -P_2^{-1} \end{bmatrix} .
\end{equation}
From our assumption and choices we see that $P_1^T=P_1$ and $G_0^T=-G_0$. Furthermore, $P_1$ is invertible.

Next we choose the matrix $W_B$ in (\ref{eq:20}) as
\begin{equation}
\label{eq:29}
 W_B = \tilde{W}_B \cdot \mathrm{diag}(I, P_{2}^{-1}, I , P_{2}^{-1}).
\end{equation}
It is clear that this has full rank $n=2n_1$. It remains to check  that for these choices the operator ${\mathcal A}$ with domain $D({\mathcal A})$, defined via (\ref{eq:19}) and (\ref{eq:20}), satisfy the condition (\ref{eq:21}).

First we note that $\left[\begin{smallmatrix} v_1 \\ v_2 \\ w_1 \\w_2 \end{smallmatrix}\right]\in \ker W_B$ if and only if $\left[\begin{smallmatrix} v_1 \\ P_2^{-1} v_2 \\ w_1 \\ P_2^{-1}w_2 \end{smallmatrix}\right]\in \ker \tilde{W}_B$.
Secondly, the following equality holds
\begin{align}
\nonumber
 \begin{bmatrix}v_1^T &  v_2^T  \end{bmatrix}  P_1 \begin{bmatrix} v_1 \\  v_2  \end{bmatrix} =& \begin{bmatrix}v_1^T &  v_2^T  \end{bmatrix} \begin{bmatrix} P_{1,1} & I \\ I & 0 \end{bmatrix}   \begin{bmatrix} v_1 \\  v_2  \end{bmatrix} \\
 \nonumber
 =&\ \begin{bmatrix}v_1^T &  v_2^T  \end{bmatrix} \begin{bmatrix} P_{1,1} & P_{2} \\ I & 0 \end{bmatrix}   \begin{bmatrix} v_1 \\  P_{2}^{-1} v_2  \end{bmatrix} \\
 \label{eq:30}
 =&\  \begin{bmatrix}v_1^T &  \left( P_{2}^{-1}v_2\right)^T  \end{bmatrix} \begin{bmatrix} P_{1,1} & P_{2} \\ -P_{2} & 0 \end{bmatrix}   \begin{bmatrix} v_1 \\  P_{2}^{-1} v_2  \end{bmatrix} .
\end{align}
Combining these two facts with (\ref{eq:28}) we have that
\[
  \begin{bmatrix}v_1^T & v_2^T  \end{bmatrix}  P_1 \begin{bmatrix} v_1 \\  v_2  \end{bmatrix}  -
  \begin{bmatrix}w_1^T &  w_2^T  \end{bmatrix}  P_1 \begin{bmatrix} w_1 \\  w_2  \end{bmatrix} \leq 0 \mbox{ for all } \begin{bmatrix} v_1 \\ v_2 \\ w_1 \\w_2 \end{bmatrix} \in \ker W_B.
\]
So we have that the operator ${\mathcal A}$ defined in (\ref{eq:19}) and \ref{eq:20}), with $P_1$ and $G_0$ given in (\ref{eq:28a}), is maximally dissipative.
Choosing $n_2=n_1$, we see that all the conditions needed in Theorem \ref{T:12} are satisfied.

Choosing $\mathcal E=\mathcal E_I$ and $\mathcal Q=I$, the operator from (\ref{eq:4:HZ}) is given as
\begin{equation}
\label{eq:31}
  \begin{bmatrix} P_{1,1} & I \end{bmatrix}\frac{d}{d\zeta} \begin{bmatrix} x_1 \\ x_2 \end{bmatrix} +P_0 x_1 \mbox{ with } \frac{dx_1}{d\zeta} - P_2^{-1} x_2 =0,
\end{equation}
with domain
\begin{equation}
\label{eq:32}
 \{ x_1,x_2 \in H^1((0,1); {\mathbb R}^{n_1}) \mid W_B \left[ \begin{smallmatrix} x_1(1) \\  x_2(1) \\ x_1(0) \\  x_2 (0) \end{smallmatrix} \right] =0\}.
\end{equation}
\mbox{}From (\ref{eq:31}) we obtain $x_2 =P_{2} \frac{dx_1}{d\zeta}$. Substituting this in the first equality of (\ref{eq:31}) and in (\ref{eq:32}) gives the operator  $\mathcal A_{red}$ and its domain as asserted.
\eproof
\medskip

We have shown how different models can be constructed out of the wave equation model by imposing a closure relation. This is the opposite construction as is usually done in Stokes or Oseen equations where the heat equation is obtain by a restriction, see \cite{EmmM13,Tem77}.

In the following example we show that we can as well obtain coupled PDEs which act on different physical domains.
\begin{example}
\label{E:15}{\rm
Consider the operator ${\mathcal A}$ of equation (\ref{eq:19}) with
\[
  P_1= \begin{bmatrix} 0&1& 0 &0 \\ 1& 0 & 0& 0\\ 0&0& 0  & 1 \\ 0& 0& 1& 0\end{bmatrix}, \quad G_0= \begin{bmatrix} 0&0& 0 &0 \\ 0& 0 & 0& 0\\ 0&0& 0  & 0 \\ 0& 0& 0 & -r \end{bmatrix}
\]
with $r$ is a bounded and invertible function satisfying Re$\left(r(\zeta)\right) \geq 0$ for all $\zeta \in [0,1]$. We choose its domain to be given by (\ref{eq:20}) with
\[
  W_B = \begin{bmatrix} 1&0& 0 &0 & 0 & 0 & -1 &0 \\ 0&1& 0 &0 & 0 & 0 & 0 &-1 \\ 0&0& 0 &0&  1 & 0 & 0 & 0 \\ 0&0& 0 &1&  0 & 0 & 0 & 0  \end{bmatrix}.
\]
It is clear that this is of full rank, and it is not hard to see that (\ref{eq:21}) holds.
We take $n_1=3$, $n_2=1$. With these choices it is straightforward to see that the conditions of  Theorem~\ref{T:12} hold, and so $(\mathcal E_I,{\mathcal A})$ is regular. Using Corollary \ref{C:8}, we can build the operator $\mathcal A_{red}$ from Theorem~\ref{T:2HZ}. For this we choose $\mathcal E=\mathcal E_I$,  and $\mathcal Q=\mathrm{diag}( \rho^{-1}, T, 1, 0)$ with $\rho,T$ (strictly) positive functions. This operator then satisfies
\[
  \mathcal A_{red}x_1 = \begin{bmatrix} \frac{d (T x_{1,2})}{d\zeta} \\  \frac{d (\rho^{-1}x_{1,1})}{d\zeta} \\  \frac{dx_{2}}{d\zeta} \end{bmatrix} \mbox{ with }  \frac{dx_{1,3}}{d\zeta} = r x_2.
\]
The corresponding PDE splits into the two PDEs
\[
  \frac{\partial}{\partial t} \begin{bmatrix} x_{1,1} \\ x_{1,2} \end{bmatrix} =  \begin{bmatrix} 0 & 1 \\ 1 & 0 \end{bmatrix} \frac{\partial}{\partial \zeta} \left (\begin{bmatrix} \rho^{-1} & 0 \\ 0 & T \end{bmatrix} \begin{bmatrix} x_{1,1} \\ x_{1,2} \end{bmatrix} \right )\mbox{ and } \frac{\partial x_{1,3}}{\partial t} =   \frac{\partial}{\partial\zeta} \left[ r^{-1} \frac{\partial x_{1,3}}{\partial\zeta}\right] .
\]
In the first PDE we recognise the wave equation, whereas the second is a heat/diffusion equation. They seem to be uncoupled, but we have not looked at the boundary conditions of $\mathcal A$. Using the closure relation $\frac{dx_{1,3}}{d\zeta} = r x_2$, we see that the boundary conditions become
\[
  \rho^{-1} x_{1,1}(1) = x_{1,3}(0),~  T x_{1,2}(1) = r^{-1} \frac{\partial x_{1,3}}{\partial\zeta}(0),~ x_{1,1}(0) = 0 \mbox{ and } \frac{\partial x_{1,3}}{\partial\zeta}(1) =0.
\]
So in this way the heat equation is coupled at the boundary to the wave equation, but certainly other couplings are possible as well.
}
\end{example}

The proof that $\mathcal A$ with $D(\mathcal A)$ given in (\ref{eq:19}) and (\ref{eq:20}) is maximally dissipative if and only if (\ref{eq:21}) holds was given in \cite{GoZM05} by using  boundary triplets, which will be the topic of the next subsection.
\subsection{Boundary triplets}\label{sec:boundarytriplets}

In this section we illustrate that our approach can also be used in the context of boundary triplets to derive results that have also been obtained via different approaches.
We begin by recalling the concept of a \emph{boundary triplet}. Let $\mathcal A_m$ be a densely defined operator on a Hilbert space $\mathcal H$ with dual $\mathcal A_m^*$, and let $\Gamma_1,\Gamma_2$ be two linear mappings from $D(\mathcal A_m^*)$ to another Hilbert space ${\mathbb U}$. The triplet $({\mathbb U}, \Gamma_1,\Gamma_2)$ is a {\em boundary triplet}\/ if the following conditions are satisfied, see \cite[section 3.1]{GoGo91}:
\begin{enumerate}
\item For all $f,g \in D(\mathcal A_m^*)$ it holds that
\begin{equation}
\label{eq:4.3.1}
  \langle \mathcal A_m^* f, g\rangle_{\mathcal H} - \langle f, \mathcal A_m^* g\rangle_{\mathcal H} = \langle \Gamma_1f,\Gamma_2 g\rangle_{{\mathbb U}} - \langle \Gamma_2f,\Gamma_1 g\rangle_{{\mathbb U}} .
\end{equation}
\item For all $u_1,u_2 \in {\mathbb U}$ there exists $f \in D(\mathcal A_m^*)$ such that $\Gamma_1 f=u_1$ and $\Gamma_2 f=u_2$.
\end{enumerate}
By choosing $f$ in the kernel of the boundary operators $\Gamma_1$ and $\Gamma_2$, we see that the corresponding restriction of $\mathcal A_m^*$ is symmetric, and not skew-symmetric, as is normally the case for generators of contraction semigroups. Therefore we will work with $i\mathcal A_m^*$ and $i\Gamma_1,\Gamma_2$.

For these operators, (\ref{eq:4.3.1}) becomes (equivalently)
\begin{equation}
\label{eq:4.3.2}
   \langle i \mathcal A_m^*f, g\rangle_{\mathcal H} + \langle f, i\mathcal A_m^* g\rangle_{\mathcal H} = \langle i\Gamma_1f,\Gamma_2 g\rangle_{{\mathbb U}} + \langle \Gamma_2f,i \Gamma_1 g\rangle_{{\mathbb U}} .
\end{equation}
In Theorem 3.1.6 of \cite{GoGo91} it is shown that $i\mathcal A_m^*$ restricted to the domain
\begin{equation}
\label{eq:35}
  \{x_0 \in D(\mathcal A_m^*) \mid (\mathcal K-I) \Gamma_1 x_0 + i(\mathcal K + I) \Gamma_2 x_0 =0\}
\end{equation}
with $\mathcal K$ satisfying $\|\mathcal K\|\leq 1$, are all maximally dissipative %\footnote{Note that accumulative in \cite{GoGo91} is dissipative here, see \cite[p.149]{GoGo91}}
restrictions of $i\mathcal A_m^*$. We will show that this result can be obtained alternatively via Theorem \ref{T:2HZ}.

To show this, for a given  boundary triplet, we define ${\mathbb X}_1=\mathcal H$, ${\mathbb X}_2={\mathbb U}$, and %(see \cite{TuWe14}???)
\begin{equation}
\label{eq:4.3.3}
  {\mathcal A} = \left[\begin{array}{c} {\mathcal A}_1 \\  {\mathcal A}_2 \end{array}\right] = \left[\begin{array}{c} {\mathcal A}_1 \\  \frac{1}{2} L ( -i \Gamma_1 + \Gamma_2)  \quad -
  \frac{1}{2} I  \end{array}\right]
\end{equation}
with
\[
  {\mathcal A}_1  \left[\begin{array}{c} x_1\\ u \end{array} \right] = i \mathcal A_m^* x_1,
\]
\begin{equation}
\label{eq:4.3.4}
  D( {\mathcal A} ) = \left\{ \left[\begin{array}{c} x_1\\ u \end{array} \right] \mid x_1 \in D(\mathcal A_m^*) \mbox{ with } (i\Gamma_1 + \Gamma_2)x_1 = u  \right\},
\end{equation}
and %$\alpha \in {\mathbb R}$,
$L \in {\mathcal L}({\mathbb U})$.

%Given the boundary triplet, we define $X_1=H$, $X_2=U\oplus U$, and (see \cite{TuWe14}???)
%\begin{equation}
%\label{eq:4.3.3}
%  {\mathcal A} = \left[\begin{array}{c} {\mathcal A}_1 \\  {\mathcal A}_2 \end{array}\right] = \left[\begin{array}{c} {\mathcal A}_1 \\ -\beta_2 \Gamma_2 \quad K_{11} \quad K_{12} \\  -i\beta_1 \Gamma_1 \quad  K_{21} \quad K_{22} \end{array}\right]
%\end{equation}
%with
%\[
%  {\mathcal A}_1  \left[\begin{array}{c} x_1\\ u\\ y_0\end{array} \right] = i A_m x_1,
%\]
%and
%\begin{equation}
%\label{eq:4.3.4}
%  D( {\mathcal A} ) = \left\{ \left[\begin{array}{c} x_1\\ u\\ y_0\end{array} \right] \mid x_1 \in D(A_m) \mbox{ with } i\Gamma_1 x_1 = \alpha_1 u, \Gamma_2x_1= \alpha_2 y_0 \right\},
%\end{equation}
%and $K_{ij} \in {\mathcal L}(U)$, $\alpha_1,\alpha_2,\beta_1,\beta_2 \in {\mathbb R}$.

Next we study the dissipativity of ${\mathcal A}$.
By the definition of ${\mathcal A}$ and relation (\ref{eq:4.3.2}) we find
\begin{align*}
  \langle {\mathcal A}\left[\begin{array}{c} x_1\\ u \end{array} \right] , \left[\begin{array}{c} x_1\\ u \end{array} \right] \rangle &+ \langle \left[\begin{array}{c} x_1\\ u\end{array} \right] , {\mathcal A}\left[\begin{array}{c} x_1\\ u \end{array} \right] \rangle\\
  & =
  \langle i {\mathcal A}_m^* x_1, x_1\rangle + \langle  x_1, i {\mathcal A}_m^* x_1\rangle + \\
  &\quad  \langle \frac{1}{2} L (-i  \Gamma_1 +  \Gamma_2) x_1, u\rangle_{{\mathbb U}} +  \langle  u, \frac{1}{2} L (-i  \Gamma_1 +  \Gamma_2)x_1\rangle_{{\mathbb U}} \\
  & \quad - \langle \frac{1}{2} u,u\rangle_{{\mathbb U}} - \langle u, \frac{1}{2} u\rangle_{{\mathbb U}} \\
  &=
  \langle i\Gamma_1x_1,\Gamma_2 x_1\rangle_{{\mathbb U}} + \langle \Gamma_2x_1,i \Gamma_1 x_1\rangle_{{\mathbb U}} + \\
  &\quad \langle \frac{1}{2} L (-i  \Gamma_1 +  \Gamma_2) x_1, u\rangle_{{\mathbb U}} +  \langle  u, \frac{1}{2} L (-i  \Gamma_1 +  \Gamma_2)x_1\rangle_{{\mathbb U}}  -  \langle  u,u\rangle_{{\mathbb U}} .
\end{align*}
Now we define $y = (-i  \Gamma_1 +  \Gamma_2) x_1$, and using (\ref{eq:4.3.4}) it is easy to see that
\[
   \langle i\Gamma_1x_1,\Gamma_2 x_1\rangle_{{\mathbb U}} + \langle \Gamma_2x_1,i \Gamma_1 x_1\rangle_{{\mathbb U}} = \frac{1}{2} \langle u,u\rangle_{{\mathbb U}} - \frac{1}{2}  \langle y,y\rangle_{{\mathbb U}}.
\]
Hence
\begin{align}
\nonumber
  \langle \mathcal A \left[\begin{array}{c} x_1\\ u \end{array} \right] , \left[\begin{array}{c} x_1\\ u \end{array} \right] \rangle &+ \langle \left[\begin{array}{c} x_1\\ u\end{array} \right] , {\mathcal A}\left[\begin{array}{c} x_1\\ u \end{array} \right] \rangle\\
\nonumber
  &=
  \frac{1}{2} \langle u,u\rangle_{{\mathbb U}} - \frac{1}{2}  \langle y,y\rangle_{{\mathbb U}} + \\
  \nonumber
  &\quad \langle \frac{1}{2} L y, u\rangle_{{\mathbb U}} +  \langle  u, \frac{1}{2} L y\rangle_{{\mathbb U}}  -  \langle  u,u\rangle_{{\mathbb U}} \\
  \label{eq:4.3.5}
   &=
  \langle \left[\begin{array}{cc} -\frac{1}{2}  I  &  \frac{1}{2} L \\  \frac{1}{2} L^* & -\frac{1}{2} I \end{array} \right] \left[\begin{array}{c} u  \\   y\end{array} \right] , \left[\begin{array}{c}  u  \\   y \end{array} \right] \rangle_{{\mathbb U}\oplus { \mathbb U}} .
 \end{align}
Using the equality
\[
  \left[\begin{array}{cc} -\frac{1}{2} I  &  \frac{1}{2} L \\  \frac{1}{2} L^* & -\frac{1}{2} I \end{array} \right] =
  \left[\begin{array}{cc}I  &  - L \\  0& I \end{array} \right]
   \left[\begin{array}{cc} -\frac{1}{2} I + \frac{1}{2} LL^* & 0 \\  0& -\frac{1}{2} I \end{array} \right]
    \left[\begin{array}{cc}I  & 0\\  -L^* & I \end{array} \right]
\]
and  (\ref{eq:4.3.5}) we see that %for $\alpha=\frac{1}{2}$
the operator $\mathcal A$ is dissipative if and only if $LL^* \leq I$, or equivalently if $\|L\| \leq 1$.
Next we choose $\mathcal E=\mathcal E_I$ and $\mathcal Q=I$, and so for $\|L\| \leq 1$ all conditions in Assumption \ref{A1:HZ} are satisfied except possibly the  regularity. By Lemma \ref{L:11}, the regularity can be checked by the maximally dissipativity of {$\mathcal A_{red}$, the closedness of $\mathcal A$}, and ${\mathcal A}_2$ being surjective. Since the pair $(\Gamma_1, \Gamma_2)$ is surjective, it follows that for every $u\in {\mathbb U}$ there exists an $x_1 \in D(\mathcal A_m^*)$ such that $(-i\Gamma_1 +\Gamma_2)x_1 =0$ and $(i\Gamma_1 +\Gamma_2)x_1 =-2u$. Hence $\left[\begin{smallmatrix} x_1\\-2u\end{smallmatrix} \right] \in D({\mathcal A})$, and ${\mathcal A}_2\left[\begin{smallmatrix} x_1\\-2u\end{smallmatrix} \right] =u$, and thus ${\mathcal A}_2$ is surjective.
That ${\mathcal A}$ is closed follows from the fact that $\mathcal A_m^*$ is closed.

So to obtain the regularity, we have to study the operator $\mathcal A_{red}$, i.e., the closure of ${\mathcal A}_{red}$.
%Therefor we first need the $A$.
Note that the definition of the domain of $\mathcal A_{red}$ already gives that the condition $\left[\begin{smallmatrix} 0\\u\end{smallmatrix} \right] \in D(\mathcal A_{red})$ implies that $u=0$. So all conditions of Theorem \ref{T:2HZ} are satisfied.

We find that $\mathcal A_{red}$ is given via
\[
  \mathcal A_{red}x_1 = i\mathcal A_m^* x_1
\]
with domain
\begin{align*}
  D(\mathcal A_{red}) =&\ \{ x_1 \in D(\mathcal A_m^*) \mid (i\Gamma_1 + \Gamma_2)x_1 = u = L(-i\Gamma_1 + \Gamma_2)x_1 \}\\
   =&\  \{ x_1 \in D{(\mathcal A_m^*)} \mid (L+I) i\Gamma_1x_1 + (-L+I) \Gamma_2x_1  =0 \}.
\end{align*}
Multiplying this expression with $i$ and taking $\mathcal K=-L$ we obtain (\ref{eq:35}), i.e., the condition of \cite{GoGo91}.

To complete the regularity proof it remains to show that $\mathcal A_{red}$ is maximally dissipative which is shown in \cite{GoGo91}.

In this subsection we have seen that boundary triplets fit into the framework of adHDAEs and in the next subsection we show this for impedance passive systems.

\subsection{Impedance passive systems}
\label{sec:4.3}

 Let $\mathcal H$, ${\mathbb V}$, and ${{\mathbb U}}$ be Hilbert spaces and let $\left[ \begin{smallmatrix} L \\ K_0 \end{smallmatrix} \right]$ be a closed operator from ${\mathbb V}$ to $\mathcal H \oplus { \mathbb U}$. We define ${\mathbb V}_0:=D\left([ \begin{smallmatrix} L \\ K_0 \end{smallmatrix} \right]) \subset {\mathbb V}$. Since $\left[ \begin{smallmatrix} L \\ K_0 \end{smallmatrix} \right]$ is closed, ${\mathbb V}_0$ with its graph norm is a Hilbert space and  $\left[ \begin{smallmatrix} L \\ K_0 \end{smallmatrix} \right]$ is a bounded operator from ${\mathbb V}_0$ to $\mathcal H \oplus {\mathbb U}$. Therefore $L^*$ and $K_0^*$ are in ${\mathcal L}(\mathcal H, {\mathbb V}_0^*)$ and ${\mathcal L}({\mathbb U}, {\mathbb V}_0^*)$, respectively. We view ${\mathbb V}$ as the pivot space, i.e., ${\mathbb V}_0 \subsethook {\mathbb V} = {\mathbb V}^* \subsethook {\mathbb V}_0^*$ are subsets with dense continuous injections.

Motivated by the Maxwell equation as well as the (damped) beam equation, the following system was introduced in \cite{StWe12}.
\begin{equation}
\label{eq:39}
  \dot{x}(t) = \left[\begin{array}{cc} 0 & - L \\ L^* & G - K^*_0 K_0 \end{array}\right]x(t) + \begin{bmatrix} 0 \\ \sqrt{2} K_0^* \end{bmatrix}u(t), \quad y(t) = \begin{bmatrix} 0 & -\sqrt{2} K_0 \end{bmatrix} x(t),
\end{equation}
on the state space ${\mathbb X}_1 = \mathcal H \oplus {\mathbb V}$.
Here $L, K_0$ satisfy the properties stated above, and
$G \in {\mathcal L}({\mathbb V}_0, {\mathbb V}_0^*)$. With our notation, we see how we have to interpret (\ref{eq:39}). Namely, the system operator has the following domain
\begin{equation}
\label{eq:40}
  D\left( \left[\begin{smallmatrix} 0 & - L \\ L^* & G - K^*_0 K_0 \end{smallmatrix}\right]\right) =\left\{ \begin{bmatrix}h\\ v\end{bmatrix}\in \mathcal H\oplus {\mathbb V}_0 \mid L^*h + (G-K_0^*K_0)v \in {\mathbb V}\right\},
\end{equation}
where the addition is done in ${\mathbb V}_0^*$.
For the rest of this subsection we concentrate on this system operator.

For the study of the system in \cite{StWe12} the following operator is introduced
\begin{equation}
\label{eq:41}
  \mathcal T = \left[\begin{array}{ccc} 0 & - L & 0 \\ L^* & G & K_0^* \\ 0  & -K_0 & 0 \end{array}\right]
\end{equation}
with domain
\begin{equation}
\label{eq:42}
  D(\mathcal T) = \left\{  \left[\begin{array}{c} h\\ e\\  u \end{array} \right] \in \mathcal H \oplus {\mathbb V}_0 \oplus {\mathbb U} \mid L^*h + Ge + K_0^*u \in {{\mathbb V}} \right\}.
\end{equation}
In \cite{StWe12} it is shown that $\mathcal T$ is maximally dissipative\footnote{Actually in \cite{StWe12} it is shown that $\mathcal T$ is m-dissipative which in our situation is equivalent to being maximally dissipative, see Lemma \ref{L:A4}.} if
\begin{equation}
\label{eq:43}
  \mathrm{Re} \langle Ge, e\rangle_{{\mathbb V}_0^*, {\mathbb V}_0} \leq 0.
\end{equation}
Under this condition, they apply an ``external Cayley transform'' to show that the system (\ref{eq:39}) is well-defined. This gives that the system operator generates a contraction semigroup, and thus is maximally dissipative.
We will show that this result can also be obtained via our techniques. For this we define ${\mathbb X}_2= {\mathbb U}$, and
\[
  {\mathcal A} = \left[\begin{array}{ccc} 0 & - L & 0 \\ L^* & G & K_0^* \\ 0  & -K_0 & -I \end{array}\right]
\]
with the domain given by that of $\mathcal T$, see (\ref{eq:42}).
Since ${\mathcal A}$ differs from $\mathcal T$ by just the $-I$ is the lower right corner it is also maximally dissipative when (\ref{eq:43}) holds.
We choose $\mathcal E=\mathcal E_I$ and $\mathcal Q=I$. Since $\mathcal T$ is maximally dissipative, we have that ${\mathcal A} + \left[\begin{smallmatrix} 0 & 0\\ 0 & I \end{smallmatrix} \right]$ is maximally dissipative, and so by Lemma \ref{L:9} $(\mathcal E,{\mathcal A}\mathcal Q)$ is regular.
Again by the $-I$ is the lower right corner of ${\mathcal A}$, we see that the condition of Theorem \ref{T:2HZ} is satisfied, and thus the operator $\mathcal A_{red}$ defined by
\[
  \mathcal A_{red}x_1 = \mathcal A_{red}\begin{bmatrix} h \\ e \end{bmatrix} = \left[\begin{array}{ccc} 0 & - L & 0 \\ L^* & G & K_0^*  \end{array}\right]\begin{bmatrix} h \\ e \\ u \end{bmatrix} \mbox{ with } -K_0e-u=0
\]
generates a contraction semigroup. It is now straightforward to see that this is the system operator from (\ref{eq:39}) and (\ref{eq:40}). So applying Theorem \ref{T:2HZ} we obtain the result of \cite{StWe12}.

In \cite{StWe12} a similar result is also obtained for the Maxwell equations.

In general, we can regard the condition ${\mathcal A}_2 \left[\begin{smallmatrix} x_1\\ x_2 \end{smallmatrix} \right]$ as a closure relation, but also as an output feedback, as we will discuss in the next subsection.

\subsection{Output feedback and systems}
\label{sec:4.4}

In this section we study output feedback, i.e., we look at $\mathcal A_{cl}=\mathcal A_0 - \mathcal B\mathcal K\mathcal C$. We can regard $z= \mathcal A_{cl}x_1$ as the solution of
\begin{equation}\label{eq:feedback}
     z=\mathcal A_0x_1 +\mathcal Bu \mbox{ with } \mathcal C{x_1} + \mathcal K^{-1} u =0,
\end{equation}
but then we would have to assume that $\mathcal K$ is invertible. In the following example we  will show that this assumption can be removed.
%To motivate the set-up, we first consider %a simple infinite-dimensional system with %bounded input and output operators.
%
\begin{example}
\label{E:18}{\rm
%{ the lower bound on B0 follows only if it is injective.
%Is that the case or is an assumption missing here? We could use the new formumation, but we do another proof. HZ}

It is easy to see that if $\mathcal A_0$ generates a contraction semigroup, so will $\mathcal A_0-\mathcal R$ for any bounded, $\mathcal R$ with $-\mathcal R$ dissipative. In this example we show this using Theorem \ref{T:2HZ}. For this we define
\[
   {\mathcal A} = \begin{bmatrix} \mathcal A_0 & \mathcal B_0 &0 \\ -\mathcal B_0^* & 0 & I \\ 0 & -I & - K \end{bmatrix},
\]
where $(\mathcal A_0, D(\mathcal A_0))$ generates a contraction semigroup on  the Hilbert space ${\mathbb Z}$, $\mathcal B_0 \in {\mathcal L}({\mathbb U},{\mathbb Z})$, and  $\mathcal K \in {\mathcal L}({\mathbb U},{\mathbb U})$, with $\mathcal K+\mathcal K^* \geq 0$. Here ${\mathbb U}$ is another Hilbert space. The domain of $\mathcal A$ is given by $D(\mathcal A_0) \oplus {\mathbb U} \oplus {\mathbb U}$.  Note that to apply our results we could even allow that $\mathcal B_0$ is unbounded, but here we apply it for bounded $\mathcal B_0$.

Next we choose ${\mathbb X}_1= {\mathbb Z}$, ${\mathbb X}_2= {\mathbb U}\oplus {\mathbb U}$, $\mathcal E = \mathcal E_I$ and $\mathcal Q=I$. Since $\mathcal A_0$ is a generator of a contraction semigroup, it is clear that the first three conditions of Assumptions \ref{A1:HZ} are fulfilled. It remains the show that $(\mathcal E_I, {\mathcal A})$ is regular.

Take $s$ in the right half plane, then by the maximal dissipativity of $\mathcal A_0$,
 \[
   (s\mathcal E_I-{\mathcal A}) \begin{bmatrix} x\\ u \\ y \end{bmatrix} = \begin{bmatrix} z\\ v \\ w \end{bmatrix} \Leftrightarrow \begin{bmatrix} x\\ \mathcal B_0^*x - y \\ u+\mathcal Ky \end{bmatrix} = \begin{bmatrix} (sI-\mathcal A_0)^{-1} z + (sI-\mathcal A_0)^{-1}\mathcal B_0 u \\ v \\ w \end{bmatrix}.
\]
Substituting the expression for $x$ into the second row, the following two equations remain to be solved:
\begin{equation}
\label{eq:59}
  \begin{bmatrix} \mathcal B_0^*(sI-\mathcal A_0)^{-1}\mathcal B_0 & -I  \\ I & \mathcal K \end{bmatrix}\begin{bmatrix}  u \\ y\end{bmatrix} = \begin{bmatrix}  v - \mathcal B_0^*(sI-\mathcal A_0)^{-1}z \\ w \end{bmatrix}.
\end{equation}
Since $\mathcal B_0$ is bounded and $\mathcal A_0$ generates a contraction semigroup, the transfer function $\mathcal B_0^*(sI-\mathcal A_0)^{-1}\mathcal B_0$ converges to zero as $s \rightarrow \infty$. Combined with the fact that $\left[ \begin{smallmatrix} 0 & -I  \\ I & \mathcal K \end{smallmatrix} \right]$ is boundedly invertible, we see that the left hand side of (\ref{eq:59}) is boundedly invertible for $s$ sufficiently large.

So the conditions of Theorem \ref{T:2HZ} are satisfied, and we can construct the corresponding ${\mathcal A}_{red}$. It is given via
\[
  {\mathcal A}_{red}x_1 = \mathcal A_0x_1 +\mathcal B_0u \mbox{ with } -\mathcal B_0^*x_1 + y =0 \mbox{ and } u+ \mathcal K y =0
\]
The latter two properties give $u= -\mathcal Ky=-\mathcal K\mathcal B_0^*x_1$, and so  $\mathcal A$ becomes
\[
 \mathcal  A = \mathcal A_0 - \mathcal B_0{ \mathcal K}\mathcal B_0^*,
\]
which we can view as an output feedback on the system $\dot{x}_1(t) = \mathcal A_0 x_1(t) + \mathcal B_0 u(t), y(t) = \mathcal B_0^* x_1(t)$.

{ After applying the feedback we can again incorporate an input and an output,
%we lost our input signal, but to keep a system with inputs and outputs we can consider
by considering} the following ${\mathcal A}$ on the space  ${\mathbb X}= \mathbb Z \oplus {\mathbb U}_1 \oplus {\mathbb U} \oplus {\mathbb U}$, where ${\mathbb U}_1$ is a Hilbert space, and $\mathcal B_1 \in {\mathcal L}({\mathbb U}_1,{\mathbb Z})$.
\[
   {\mathcal A} = \begin{bmatrix} \mathcal A_0 & \mathcal B_1 &\mathcal B_0 & 0 \\ -\mathcal B_1^* & 0 & 0 & 0 \\ -\mathcal B_0^* & 0 & 0 & I \\  0 & 0 & -I & \mathcal K \end{bmatrix}.
\]
We split the space as ${\mathbb X}_1=\mathbb Z\oplus {\mathbb U}_1$, ${\mathbb X}_2={\mathbb U}\oplus {\mathbb U}$, and so the last two rows of ${\mathcal A}$ form ${\mathcal A}_2$.

Choosing $\mathcal E=\mathcal E_I$ and $\mathcal Q=I$, then $\mathcal A$ obtained after applying the closure relation, is given by
\[
  \mathcal A=\begin{bmatrix} \mathcal A_0-\mathcal B_0{ \mathcal K}\mathcal B_0^* & \mathcal B_1 \\ -\mathcal B_1^* & 0 \end{bmatrix}
\]
which is maximally dissipative. This implies that the system
\[
  \dot{z}(t) = (\mathcal A_0-\mathcal B_0{ \mathcal K}\mathcal B_0^*)z(t) + \mathcal B_1 u_1(t), \quad y_1 (t) = -\mathcal B_1^* z(t)
\]
is impedance passive. Note that with the choice of $\mathcal Q=\mathrm{diag}(\mathcal Q_1, I, I , I)$, we get impedance passivity with the storage function $q(z) = \langle z, \mathcal Q_1 z\rangle$, see \cite[Theorem 7.5.4]{CuZw20}.
}
\end{example}

%The above example can easily be extended to systems nodes. %Consider the dissipative operator/system node $\Sigma= %\begin{bmatrix} A\&B \\ -C\&D \end{bmatrix}$, see %\cite{Staf05}. Based on Theorem \ref{T:11} we see that the %output feedback $y=Ku$ gives a contraction semigroup if %the transfer function of the extended systems node
%\[
%  {\mathcal A} = \begin{bmatrix} A\&B & 0 \\ -C\&D & I \\ %0~-I & -K \end{bmatrix}
%\]
%has an invertible transfer function for some $s \in %{\mathbb C}^+$. This holds if $G(s)K + I$ is boundedly %invertible, where $G$ is the transfer function of the %original system.

\section{Existence of solutions on a subspace}
\label{sec:5}

In the previous section we have considered the operator $\mathcal A$ under the condition that ${\mathcal A}_2$ was injective on $\{0\}\oplus {\mathbb X}_2 \cap D({\mathcal A})$. In the following theorem we use a stronger assumption and study the existence of solutions to \eqref{eq:1HZ}.
%We no longer can expect that $\mathcal A$ will %generate a semigroup on the whole of ${\color{blue} \mathbb X}_1$. %Note that for a specific example both theorems may be %applicable, but since the next theorem has the %strongest assumptions, the previous theorem is %preferred (when applicable).
%
\begin{theorem}
\label{T:6}
Consider an adHDAE \eqref{eq:1HZ} with operators $\mathcal E$, ${\mathcal A}$, $\mathcal Q$ and the Hilbert spaces ${\mathbb X}, {\mathbb X}_1$, and ${\mathbb X}_2$ satisfying the conditions of Assumption \ref{A1:HZ}.
Define $\mathcal W_0 \subset {\mathbb X}_1$ as the first component of the kernel of ${\mathcal A}_2\left[\begin{smallmatrix} \mathcal Q_1  \\ \mathcal Q_2 \end{smallmatrix}\right]$, i.e.,
\begin{equation}
\label{eq:5.1}
   \mathcal W_0= \{ x_1 \in {\mathbb X}_1 \mid \exists \ x_2\in {\mathbb X}_2 \mbox{ s.t. }  \left[\begin{smallmatrix} \mathcal Q_1x_1 \\ \mathcal Q_2 x_2 \end{smallmatrix}\right] \in D({\mathcal A}) \mbox{ and }{\mathcal A}_2  \left[\begin{smallmatrix} \mathcal Q_1 x_1 \\ \mathcal Q_2x_2 \end{smallmatrix}\right] =0\}.
\end{equation}
 Let ${\mathbb X}_0 \subseteq {\mathbb X}_1$ be the closure of $\mathcal W_0$ in ${\mathbb X}_1$.
If %the following set only contains the zero element
\begin{equation}
\label{eq:5.2}
  \{ y_1 \in {\mathbb X}_0 \mid \exists \ x_2\in {\mathbb X}_2 \mbox{ s.t. } \left[\begin{smallmatrix} 0  \\ \mathcal Q_2x_2 \end{smallmatrix}\right] \in D({\mathcal A}) \mbox{ and } {\mathcal A}\left[\begin{smallmatrix} 0  \\ {\mathcal Q}_2x_2 \end{smallmatrix}\right] = \left[\begin{smallmatrix} \mathcal E_1 y_1  \\  0\end{smallmatrix}\right] \}=\{0\},
\end{equation}
then the operator $\mathcal A_{red} : D(\mathcal A_{red}) \subset {\mathbb X}_0 \rightarrow {\mathbb X}_0$ generates a contraction semigroup on ${\mathbb X}_0$, where the domain $D(\mathcal A_{red})$ is defined as
\begin{align}
\label{eq:7:HZ}
  D(\mathcal A_{red}) = \{ x_1 \in {\mathbb X}_0 \mid&\ \exists\ x_2 \in {\mathbb X}_2 \mbox{ such that } \left[\begin{smallmatrix} \mathcal Q_1x_1 \\ \mathcal Q_2x_2 \end{smallmatrix} \right]\in D({\mathcal A}), \\
  \nonumber
  &\ {\mathcal A}_2  \left[\begin{smallmatrix} \mathcal Q_1x_1 \\ \mathcal Q_2x_2 \end{smallmatrix} \right] =0 \mbox{ and } \mathcal E_1^{-1} {\mathcal A}_1  \left[\begin{smallmatrix} \mathcal Q_1 x_1 \\ \mathcal Q_2 x_ 2 \end{smallmatrix} \right] \in {\mathbb X}_0 \}
\end{align}
and for $x_1 \in D(\mathcal A_{red})$ the action of $\mathcal A_{red}$ is defined via
\begin{equation}
\label{eq:8:HZ}
  \mathcal A_{red}x_1 = \mathcal E_1^{-1} {\mathcal A}_1  \begin{bmatrix} \mathcal Q_1 x_1 \\ \mathcal Q_2 x_2 \end{bmatrix}.
\end{equation}
\end{theorem}
\proof
First we have to prove that $\mathcal A_{red}$ is well-defined. Note that $D(\mathcal A_{red})\subset \mathcal W_0$. So if for a given $x_1\in D(\mathcal A_{red})$ we have that $x_2$ and $\tilde{x}_2$ are such that the conditions on the domain are satisfied for  $\left[\begin{smallmatrix} \mathcal Q_1x_1 \\\mathcal Q_2 x_2 \end{smallmatrix}\right]$ and $\left[\begin{smallmatrix}\mathcal Q_1x_1  \\\mathcal Q_2 \tilde{x}_2 \end{smallmatrix} \right]$, then by the linearity of ${\mathcal A}_2$, we have that
\[
  {\mathcal A}_2  \begin{bmatrix} 0 \\ \mathcal Q_2 x_2 -\mathcal Q_2 \tilde{x}_2\end{bmatrix} =0.
\]
Furthermore, we know that $y_1 :=\mathcal E_1^{-1} {\mathcal A}_1 \begin{bmatrix}\mathcal Q_1 x_1 \\\mathcal Q_2 x_2 \end{bmatrix}$ and $\tilde{y}_1:=\mathcal E_1^{-1} {\mathcal A}_1  \begin{bmatrix}\mathcal Q_1 x_1 \\\mathcal Q_2 \tilde{x}_2 \end{bmatrix}$ are in ${\mathbb X}_0$. Since ${\mathbb X}_0$ is a linear space, we find that
\[
  y_1 -\tilde{y}_1 = \mathcal E_1^{-1} {\mathcal A}_1  \begin{bmatrix} 0 \\ \mathcal Q_2 (x_2-\tilde{x}_2) \end{bmatrix} \in {\mathbb X}_0.
\]
Combining the two equations gives that
\[
  \mathcal A \begin{bmatrix} 0 \\\mathcal Q_2 (x_2 -  \tilde{x}_2) \end{bmatrix} =\begin{bmatrix} \mathcal A_1 \\ \mathcal A_2 \end{bmatrix}\begin{bmatrix} 0 \\\mathcal Q_2 (x_2 -  \tilde{x}_2) \end{bmatrix}= \begin{bmatrix} \mathcal E_1(y_1- \tilde{y}_1)  \\ 0 \end{bmatrix}
\]
with $y_1 -\tilde{y}_1 \in{\mathbb X}_0$. Our assumption gives that $y_1=\tilde{y}_1$, and thus $\mathcal A_{red}x_1$ is unique, and so is well-defined.

We have
\begin{align*}
  \langle \mathcal A_{red}x_1, x_1\rangle_{\mathcal E\mathcal Q} + \langle \mathcal A_{red}x_1, x_1\rangle_{\mathcal E\mathcal Q} =&\ \langle \mathcal E_1^{-1} {\mathcal A}_1  \begin{bmatrix} \mathcal Q_1x_1 \\\mathcal  Q_2x_2 \end{bmatrix},\mathcal  E_1^* \mathcal Q_1 x_1\rangle + \langle\mathcal  E_1^*\mathcal Q_1 x_1, \mathcal E_1^{-1} {\mathcal A}_1  \begin{bmatrix} \mathcal Q_1 x_1 \\\mathcal Q_2 x_2 \end{bmatrix} \rangle\\
  =&\ \langle {\mathcal A}  \begin{bmatrix} \mathcal Q_1x_1 \\ \mathcal Q_2x_2 \end{bmatrix}, \begin{bmatrix}\mathcal  Q_1x_1  \\\mathcal  Q_2x_2 \end{bmatrix}\rangle + \langle \begin{bmatrix} \mathcal Q_1x_1 \\ \mathcal Q_2x_2 \end{bmatrix}, {\mathcal A}  \begin{bmatrix} \mathcal Q_1x_1 \\ \mathcal Q_2x_2 \end{bmatrix} \rangle
  \leq \ 0,
  \end{align*}
where we have used that ${\mathcal A}_2  \left[\begin{smallmatrix}\mathcal  Q_1x_1  \\\mathcal  Q_2x_2 \end{smallmatrix}\right] =0$. % and the dissipativity of ${\mathcal A}$.
Hence $\mathcal A_{red}$ is dissipative.

Next we show that $sI-\mathcal A_{red}$ is onto, where $s$ is a complex number with positive real part in  the regularity assumption. For this, choose  $y_1 \in {\mathbb X}_0$. By the regularity assumption we know that there exists  $\left[\begin{smallmatrix} x_1 \\ x_2 \end{smallmatrix}\right] \in D({\mathcal A}{{\mathcal Q}})$ such that
\begin{equation}
  \label{eq:9HZ}
    \begin{bmatrix} \mathcal E_1 y_1 \\ 0  \end{bmatrix} = (s\mathcal E - {\mathcal A}\mathcal Q) \begin{bmatrix} x_1 \\ x_2  \end{bmatrix} = s \begin{bmatrix} \mathcal E_1 x_1 \\ 0 \end{bmatrix} - {\mathcal A}\begin{bmatrix} \mathcal Q_1x_1 \\\mathcal  Q_2x_2 \end{bmatrix} .
\end{equation}
The second row of this expression gives that
\[
    {\mathcal A}_2  \begin{bmatrix} \mathcal Q_1x_1 \\ \mathcal Q_2 x_2 \end{bmatrix} =0
\]
and so $x_1 \in \mathcal W_0$. The first row of \eqref{eq:9HZ} gives
\[
    s\mathcal E_1 x_1 - {\mathcal A}_1 \begin{bmatrix} \mathcal Q_1x_1 \\ \mathcal Q_2 x_2 \end{bmatrix} = \mathcal E_1y_1.
\]
Using that $y_1$ and $x_1$ are in ${\mathbb X}_0$, we get that $x_1 \in D(\mathcal A_{red})$, and $(sI-\mathcal A_{red})x_1 = y_1$. Hence $sI-\mathcal A_{red}$ is surjective for an $s \in {\mathbb C}^+$. By the Lumer-Phillips Theorem we conclude that $\mathcal A_{red}$ generates a contraction semigroup on ${\mathbb X}_0$.
\hfill \eproof
\medskip

We note that if we have a classical solution of
\[
  \dot{x}_1(t) = \mathcal A x_1(t),
\]
then $x_1(t) \in D(\mathcal A)\subset \mathcal W_0\subset { {\mathbb X}}_0$ for all $t \geq 0$, and thus there exists an $x_2(t)$ such that
\[
  \mathcal E_1\dot{x}_1(t) = {\mathcal A}_1  \begin{bmatrix} \mathcal Q_1x_1(t) \\ \mathcal Q_2 x_2(t)  \end{bmatrix} \mbox{ and } {\mathcal A}_2  \begin{bmatrix} \mathcal Q_1x_1(t) \\ \mathcal Q_2 x_2(t)  \end{bmatrix} =0.
\]
We can regard $x_2$ as the Lagrange multiplier enabling $x_1$ to stay in $\mathcal W_0$.
\begin{example}\label{ex:block}{\rm

Consider the system~\eqref{eq:feedback} but with $\mathcal K=0$, i.e., let
\[
   {\mathcal A} = \begin{bmatrix} \mathcal A_0 & \mathcal B_0 \\ -\mathcal B_0^* & 0 \end{bmatrix},
\]
where $\mathcal A_0$ is maximally dissipative on the Hilbert space ${\mathbb X}_1$, and $\mathcal B_0 \in {\mathcal L}({\mathbb U},{\mathbb X}_1)$ where $\mathcal B_0$ is injective and has closed range, i.e., there exists $\beta >0$ such that $\|\mathcal B_0 u \| \geq \beta \|u\|$, for all $u \in {\mathbb U}$. We choose ${\mathbb X}_2={\mathbb U}$.
To check the regularity for our class of $\mathcal E$ and $\mathcal Q$ it suffices to check it for $\mathcal E_I$ and $\mathcal Q=I$, see Corollary \ref{C:8}.

% This is a very simple case of a system node, and so by Theorem \ref{T:11},
We first study the invertibility of the transfer function $G(s)=\mathcal B_0^*(sI-\mathcal A)^{-1}\mathcal B_0$. It is well-known that $\lim_{s \rightarrow \infty} s G(s) = \mathcal B_0^*\mathcal B_0$, and by our assumption on $\mathcal B_0$ this inverse exists. So for $s$ sufficiently large $sG(s)$ and thus also $G(s)$ is boundedly invertible. Hence $(\mathcal E_I,{\mathcal A})$ is regular, and
%
%So we have to solve the following set of equations for arbitrary $z_1 \in X_1$ and $z_2 \in U$.
%\[
%  (sI- \mathcal A_0)x_1 - \mathcal B_0 u = z_1 \quad \mathcal B_0^* x_1 = z_2
%\]
%Since $\mathcal A_0$ is dissipative and Re$(s) >0$, we find
%\begin{equation}
%\label{eq:49}
%  x_1= (sI-A)^{-1}z_1 + (sI-A)^{-1}\mathcal B_0u,\quad \mathcal B_0^* x_1 = z_2
%\end{equation}
%Substituting the first expression into the second gives
%\[
%  \mathcal B_0^*(sI-A)^{-1}z_1 + \mathcal B_0^*(sI-A)^{-1}\mathcal B_0u =z_2
%\]
%Sinceis boundedly invertible, we can express $u$ in $z_1$ and $z_2$, and using the first equation of (\ref{eq:49}) once more, we also find $x_1$.
 so is $(\mathcal E,{\mathcal A}\mathcal Q)$. With this, we can define ${\mathbb X}_0$ and ${\mathcal A_{red}}$.

Using (\ref{eq:5.1}) we get that $\mathcal W_0 =\{ x_1 \in {\mathbb X}_1\mid \mathcal Q_1 x_1 \in D(\mathcal A_0)$ and $\mathcal Q_1 x_1 \in \ker \mathcal B_0^*\}$. So ${\mathbb X}_0= \mathcal Q_1^{-1} \overline{\ker \mathcal B_0^* \cap D(\mathcal A_0)} $. In many cases the domain of $\mathcal A_0$ will be dense in the kernel of $\mathcal B_0^*$, and thus in that case ${\mathbb X}_0 = \mathcal Q_1^{-1} \ker \mathcal B_0^*$.

The element $y_1$ is in the set defined by (\ref{eq:5.2}) if $\mathcal B_0\mathcal Q_2x_2=\mathcal E_1 y_1$ and ${\mathcal Q_1} y_1 \in \ker \mathcal B_0^*$. Thus $\mathcal B_0^*\mathcal Q_1\mathcal E_1^{-1} \mathcal B_0 \mathcal Q_2 x_2 =0$,  which implies that  $\langle \mathcal B_0 \mathcal Q_2 x_2 , \mathcal Q_1\mathcal E_1^{-1} \mathcal B_0 \mathcal Q_2 x_2\rangle =0$. Since $\mathcal Q_1\mathcal E_1^{-1}$ is coercive, this gives  $ \mathcal B_0 \mathcal Q_2 x_2 =0$ and thus $\mathcal E_1y_1=0$. The invertibility of  $\mathcal E_1$ finally gives $y_1=0$.

Thus, all the conditions of Theorem \ref{T:6} are satisfied. We choose $\mathcal E=\mathcal E_I$ and $\mathcal Q=I$, to study the $\mathcal A$ constructed in Theorem \ref{T:6}.
\[
\hat{\mathcal   A}x_1 = \mathcal A_0x_1 + \mathcal B_0u, \mbox{ with } x_1 \in D(\mathcal A_0), \mathcal B_0^* x_1 =0, \mbox{ and } \mathcal B_0^*(\mathcal A_0x_1 + \mathcal B_0u) =0.
\]
The last expression gives $u= -(\mathcal B_0^*\mathcal B_0)^{-1}\mathcal B_0^*\mathcal A_0x_1$, and so on ${\mathbb X}_0$ we have the  operator
\[
\mathcal A_{red}  x_1 = \left(\mathcal A_0 - \mathcal B_0(\mathcal B_0^*\mathcal B_0)^{-1}\mathcal B_0^*\mathcal A_0\right)x_1.
\]
Theorem~\ref{T:6} states that there is a well-defined dynamics on this space. If we interpret the second state component as the output, then this ${\mathbb X}_0$ has the interpretation as the output nulling subspace. It is well-known that the largest output nulling subspace exists when $\mathcal B_0^*\mathcal B_0$ is invertible, see \cite{Curt86} or\cite{Zwar89}.

In general, when $\mathcal C \in {\mathcal L}({\mathbb X}_1, {\mathbb U})$ is such that there exists a coercive $\mathcal Q_1 \in {\mathcal L}({\mathbb X}_1)$ such that $\mathcal C=\mathcal B_0^*\mathcal Q_1$, then we get, with $\mathcal E=\mathcal E_I$ and $\mathcal Q=\mathrm{diag}(\mathcal Q_1,I)$, that ${\mathbb X}_0 =\ker \mathcal C$, and
\[
\hat{ \mathcal   A}x_1 = \left(\mathcal A_0 - \mathcal B_0(\mathcal C\mathcal B_0)^{-1}{ \mathcal C}\mathcal A_0\right){ \mathcal Q_1}x_1.
\]
}% end rm
\end{example}

In the following example we study  the class studied in Theorem \ref{T:15}. However, the applications of this class are different, it contains e.g.\ the Oseen or Stokes equation, see \cite{EmmM13} and \cite{ReiS23}. The setup is similar as for the impedance passive systems studied in Subsection~\ref{sec:4.3}.
\begin{example}
\label{ex:stokes}{\rm
Let ${\mathbb V}$ be a real Hilbert space such that ${\mathbb V} \subsethook {\mathbb X}_1 = {\mathbb X}_1^* \subsethook {{\mathbb V}}^*$, Let $\mathcal A_0 \in {\mathcal L}({\mathbb V},{\mathbb V}^*)$, $\mathcal B_0 \in {\mathcal L}({\mathbb U},{\mathbb V}^*)$, where ${\mathbb U}$ is a second (real) Hilbert space. So  $\mathcal B_0^* \in {\mathcal L}({\mathbb V},{\mathbb U}^*)$. We identify ${\mathbb U}^*$ with ${\mathbb U}$. We assume that $\mathcal A_0$ is dissipative and $\mathcal B_0$ is injective and has closed range.

With these operators we define, see also (\ref{eq:50}) and (\ref{eq:50a}),
  \begin{equation}
%  \label{eq:50}
    {\mathcal A} = \begin{bmatrix} \mathcal A_0 & \mathcal B_0 \\-\mathcal B_0^* &0 \end{bmatrix}
  \end{equation}
  with domain
  \[
    D({\mathcal A}) = \{ \begin{bmatrix} v \\ u \end{bmatrix} \in {\mathbb V} \oplus {\mathbb U} \mid  \mathcal A_0v + \mathcal B_0u \in {\mathbb X}_1\}.
 \]
By Theorem \ref{T:15} we know that $(\mathcal E_I, {\mathcal A})$ is regular. So we can apply Theorem \ref{T:6} on this class.

By the definition of ${\mathcal A}$ we have that
\[
  \mathcal W_0 =\{ x_1 \in {\mathbb V} \mid \exists\ x_2 \in {\mathbb U} \mbox{ s.t. } \mathcal A_0x_1+\mathcal B_0x_2 \in {\mathbb X}_1, \mbox{ and } \mathcal B_0^*x_1 =0\}.
\]
Next we study the solution set of equation (\ref{eq:5.2}). Let $y_1 \in {\mathbb X}_0 = \overline{\mathcal W_0}$, i.e., the closure of $\mathcal W_0$ in ${\mathbb X}_1$, be such that there exists an $u \in\mathcal  U$ is such that
\begin{equation}
\label{eq:51}
  {\mathcal A}\begin{bmatrix} 0  \\ u  \end{bmatrix} =\begin{bmatrix} y_1  \\  0\end{bmatrix}.
\end{equation}

From the definition of the domain of ${\mathcal A}$ we obtain that $\mathcal B_0 u \in {\mathbb X}_1$. If $\mathcal B_0$ is completely unbounded, then this would imply that $u=0$, and thus $y_1=0$. Otherwise, since $y_1 \in {\mathbb X}_0$ there exists a sequence $z_n \in \mathcal W_0 \subset {\mathbb V}$ such that $z_n \rightarrow y_1$ in ${ \mathbb X}_1$. In particular, $\mathcal B_0^*z_n=0$. Combining this with the fact that $y_1=\mathcal B_0u$, see (\ref{eq:51}), we find
\[
   \langle y_1, y_1 \rangle_{{\mathbb X}_1} = \lim_{n \rightarrow \infty}  \langle z_n, \mathcal B_0 u\rangle_{{\mathbb X}_1} = \lim_{n \rightarrow \infty}  \langle z_n, \mathcal B_0 u\rangle_{{\mathbb V}, {\mathbb V}^*} = \lim_{n \rightarrow \infty}  \langle \mathcal B_0^*z_n, u\rangle_{{\mathbb U}} =0.
\]
Hence $y_1=0$. Thus the conditions of Theorem \ref{T:6} are satisfied.
}
\end{example}

A concrete application of the set-up in the previous example is given next.
\begin{example}\label{ex:oseen}{\rm

Consider, as in \cite{EmmM13}
a linearized Navier-Stokes equation and given by
\begin{align*}
  \frac{\partial v}{\partial t} - \alpha \Delta v + \nabla p =&\ 0\\
  \nabla^T v = &\ 0,
\end{align*}
on a  spatial domain $\Omega$.

For the abstract set-up of Example \ref{ex:stokes} we choose ${\mathbb V}= H_0^1(\Omega)$, ${\mathbb X}_1 = L^2(\Omega)$, and ${{\mathbb U}} = {\mathbb X}_2=L^2(\Omega)/{\mathbb R}$, i.e., two functions in ${\mathbb U}$ are considered to be the same if they differ by a constant. Furthermore, ${\mathcal A}$ is taken as
\[
   {\mathcal A}  = \begin{bmatrix} \mathcal A_0 & \mathcal B_0 \\-\mathcal B_0^* &0 \end{bmatrix}= \begin{bmatrix} \alpha \Delta & -\nabla\\ \nabla^T & 0 \end{bmatrix}.
\]
Since for $v,w \in {{\mathbb V}}$
\[
  \int_{\Omega} \left(\Delta v\right) w\ d\omega = - \int_{\Omega} \nabla v \cdot \nabla w\ d\omega,
\]
we see that $\mathcal A_0$ is dissipative. Furthermore, it satisfies the G\aa rding inequality, see \cite{EmmM13} for the proof in the more general case of the linearized Navier-Stokes and Oseen equation. Furthermore, ${\mathcal B}_0$ is injective, has closed range and satisfies the condition \eqref{eq:23}, see e.g.\
\cite{BreF12}.
Hence if we choose
\[
 \mathcal  E = \begin{bmatrix} I &  0 \\ 0&  0 \end{bmatrix} \mbox{ and } \mathcal Q= \begin{bmatrix} I &  0 \\ 0&  I  \end{bmatrix},
\]
then it fits the framework of Example \ref{ex:stokes}. Note that ${\mathbb X}_0$ is now the space of divergence free functions.
}%end rm
\end{example}

\section{Conclusion and possible extensions}

Abstract linear dissipative Hamiltonian differential-algebraic equations (DAEs) on Hilbert spaces are studied. A characterization  is given when these are associated with singular and regular operator pairs. It is shown that due to closure relations and structural properties this class of operator equations arises typically when studying classical evolution equations. This is illustrated by several applications.

However, this class does not only arises when the state spaces are Hilbert spaces, and these abstract
DAEs are not restricted to linear systems.  To extend the presented theory for dissipative systems on a Banach space, the article \cite{ScZw14} can serve as a starting point. Among others it is shown there that Example \ref{E3:HZ} can be treated in the context of Banach spaces as well.
\section*{Acknowledgement} We thank two anonymous referees for there many comments, which helped to improve the paper substantially.
%Future work includes the extension of the theory to nonlinear systems.
%,
%will be a greater challenge,
%since for these the existence and uniqueness theory is not as fully developed as for the linear case. %For nonlinear (dissipative) mappings there is a well-developed existence and uniqueness theory and that can be used. However, since for a non-linear mapping ${\mathcal A}$, dissipativity means that Re$(\langle {\mathcal A}(x_1)-  {\mathcal A}(x_2), x_1 -x_2 \rangle) \leq 0$, not every system that will not grow in its norm/energy will be dissipative. This  is in contrast with the linear case. Hence there will be fewer realistic examples fitting this nonlinear theory.
%{I suggest to rephrase the last part on the Conclusion on nonlinear
%extensions because it is not completely clear what is meant there. }

%\bibliographystyle{plain}
%\bibliography{BeaMXZ,ph}

\begin{thebibliography}{10}

\bibitem{ABHN11}
W.~Arendt, C.~J.~K. Batty, M.~Hieber, and F.~Neubrander.
\newblock {\em Vector-valued {L}aplace Transforms and {C}auchy Problems},
  volume~96 of {\em Monographs in Mathematics}.
\newblock Birkh\"{a}user/Springer Basel AG, Basel, second edition, 2011.

\bibitem{BarCGJR23}
A.~Bartel, M.~Clemens, M.~G{\"u}nther, B.~Jacob, and T.~Reis.
\newblock Port-{H}amiltonian systems modelling in electrical engineering.
\newblock {\em arXiv preprint arXiv:2301.02024}, 2023.

\bibitem{BeaMXZ18}
C.~Beattie, V.~Mehrmann, H.~Xu, and H.~Zwart.
\newblock Port-{Hamiltonian} descriptor systems.
\newblock {\em Math. Control Signals Systems}, 30(17):1--27, 2018.

\bibitem{BenMHL24}
A.~Bendimerad-Hohl, D.~Matignon, G.~Haine, and L.~Lef{\`e}vre.
\newblock On implicit and explicit representations for 1d distributed
  port-{H}amiltonian systems.
\newblock {\em arXiv preprint arXiv:2402.07628}, 2024.

\bibitem{BreCP96}
K.~E. Brenan, S.~L. Campbell, and L.~R. Petzold.
\newblock {\em Numerical Solution of Initial-value Problems in
  Differential-algebraic Equations}.
\newblock Society for Industrial and Applied Mathematics, Philadelphia, PA,
  1996.

\bibitem{BreF12}
F.~Brezzi and M.~Fortin.
\newblock {\em Mixed and Hybrid Finite Element Methods}, volume~15.
\newblock Springer Science \& Business Media, 2012.

\bibitem{CuZw20}
R.~Curtain and H.~Zwart.
\newblock {\em Introduction to Infinite-Dimensional Systems Theory, A
  state-space approach}, volume~71 of {\em Texts in Applied Mathematics}.
\newblock Springer, New York, 2020.

\bibitem{Curt86}
R.~F. Curtain.
\newblock {Invariance Concepts in Infinite Dimensions}.
\newblock {\em SIAM J. Control Optim.}, 24(5):1009--1030, 1986.

\bibitem{CuZw95}
R.~F. Curtain and H.~J. Zwart.
\newblock {\em An Introduction to Infinite-Dimensional Linear Systems Theory},
  volume~21 of {\em Texts in Applied Mathematics}.
\newblock Springer-Verlag, New York, 1995.

\bibitem{EggK18}
H.~Egger and T.~Kugler.
\newblock Damped wave systems on networks: Exponential stability and uniform
  approximations.
\newblock {\em Numer. Math.}, 138(4):839--867, 2018.

\bibitem{EggKLMM18}
H.~Egger, T.~Kugler, B.~Liljegren-Sailer, N.~Marheineke, and V.~Mehrmann.
\newblock On structure preserving model reduction for damped wave propagation
  in transport networks.
\newblock {\em {SIAM} J. Sci. Comput.}, 40:A331--A365, 2018.

\bibitem{EmmM13}
E.~{Emmrich} and V.~{Mehrmann}.
\newblock Operator differential-algebraic equations arising in fluid dynamics.
\newblock {\em Comput. Methods Appl. Math}, 13(4):443--470, 2013.

\bibitem{ErbJMRT24}
M.~Erbay, B.~Jacob, K.~Morris, T.~Reis, and C.~Tischendorf.
\newblock Index concepts for linear differential-algebraic equations in finite
  and infinite dimensions.
\newblock {\em arXiv preprint arXiv:2401.01771}, 2024.

\bibitem{GerHR21}
H.~Gernandt, F.~E. Haller, and E.~Reis.
\newblock A linear relation approach to port-{Hamiltonian}
  differential-algebraic equations.
\newblock {\em {SIAM} J. Matrix Anal. Appl.}, 42(2):1011--1044, 2021.

\bibitem{GerHRS21}
H.~Gernandt, F.~E. Haller, T.~Reis, and A.~J. van~der {S}chaft.
\newblock Port-{Hamiltonian} formulation of nonlinear electrical circuits.
\newblock {\em J. Geom. Phys}, 159:103959, 2021.

\bibitem{GerR23}
H.~Gernandt and T.~Reis.
\newblock A pseudo-resolvent approach to abstract differential-algebraic
  equations.
\newblock {\em arXiv}, 2023.

\bibitem{GoGo91}
V.~I. Gorbachuk and M.~L. Gorbachuk.
\newblock {\em On Boundary Value Problems for Operator Differential Equations},
  volume~48 of {\em Mathematics and its Applications (Soviet Series)}.
\newblock Kluwer Academic Publishers Group, Dordrecht, 1991.
\newblock Translated and revised from the 1984 Russian original.

\bibitem{GueBJR21}
M.~G{\"u}nther, A.~Bartel, B.~Jacob, and T.~Reis.
\newblock Dynamic iteration schemes and port-{H}amiltonian formulation in
  coupled differential-algebraic equation circuit simulation.
\newblock {\em Int. J. Circuit Theory Appl.}, 49(2):430--452, 2021.

\bibitem{Hac17}
W.~Hackbusch.
\newblock {\em Elliptic Differential Equations}, volume~18 of {\em Springer
  Series in Computational Mathematics}.
\newblock Springer-Verlag, Berlin, second edition, 2017.
\newblock Theory and numerical treatment.

\bibitem{HaiW96}
E.~Hairer and G.~Wanner.
\newblock {\em Solving Ordinary Differential Equations {II}: Stiff and
  Differential-Algebraic Problems}.
\newblock Springer-Verlag, Berlin, Germany, 2nd edition, 1996.

\bibitem{JacM22}
B.~Jacob and K.~Morris.
\newblock On solvability of dissipative partial differential-algebraic
  equations.
\newblock {\em IEEE Control Systems Letters}, 6:3188--3193, 2022.

\bibitem{JaMZ15}
B.~Jacob, K.~Morris, and H.~Zwart.
\newblock {$C_0$-semigroups for hyperbolic partial differential equations on a
  one-dimensional spatial domain}.
\newblock {\em Journal of Evolution Equations}, 15(2):493 -- 502, 2015.

\bibitem{JacZ12}
B.~Jacob and H.~Zwart.
\newblock {\em Linear port-{Hamiltonian} Systems on Infinite-Dimensional
  Spaces}, volume 223 of {\em Operator Theory: Advances and Applications}.
\newblock Birkh{\"a}user, Basel, 2012.

\bibitem{JaeEGJ22}
J.~J{\"a}schke, M.~Ehrhardt, M.~G{\"u}nther, and B.~Jacob.
\newblock A port-{H}amiltonian formulation of coupled heat transfer.
\newblock {\em Math. Comput. Model. Dyn. Sys.}, 28(1):78--94, 2022.

\bibitem{KunM06}
P.~Kunkel and V.~Mehrmann.
\newblock {\em Differential-Algebraic Equations. Analysis and Numerical
  Solution}.
\newblock European Mathematical Society, Z{\"u}rich, 2006.

\bibitem{LamMT13}
R.~Lamour, R.~M{\"a}rz, and C.~Tischendorf.
\newblock {\em Differential-algebraic Equations: aPprojector Based Analysis}.
\newblock Springer Science \& Business Media, 2013.

\bibitem{GoZM05}
Y.~Le~Gorrec, H.~Zwart, and B.~Maschke.
\newblock Dirac structures and boundary control systems associated with
  skew-symmetric differential operators.
\newblock {\em SIAM J. Control Optim.}, 44(5):1864--1892, 2005.

\bibitem{MehMW18}
C.~Mehl, V.~Mehrmann, and M.~Wojtylak.
\newblock Linear algebra properties of dissipative {Hamiltonian} descriptor
  systems.
\newblock {\em {SIAM} J. Matrix Anal. Appl.}, 39(3):1489--1519, 2018.

\bibitem{MehMW21}
C.~{Mehl}, V.~{Mehrmann}, and M.~{Wojtylak}.
\newblock Distance problems for dissipative {Hamiltonian} systems and related
  matrix polynomials.
\newblock {\em Linear Algebra Appl.}, pages 335--366, 2021.

\bibitem{MehU23}
V.~Mehrmann and B.~Unger.
\newblock Control of port-{H}amiltonian differential-algebraic systems and
  applications.
\newblock {\em Acta Numerica}, pages 395--515, 2023.

\bibitem{MehS23}
V.~Mehrmann and A.J. van~der Schaft.
\newblock Differential-algebraic systems with dissipative {H}amiltonian
  structure.
\newblock {\em Math. Control Signals Systems}, pages 1--44, 2023.

\bibitem{Miya92}
I.~Miyadera.
\newblock {\em Nonlinear semigroups / Isao Miyadera ; translated by Choong Yun
  Cho.}
\newblock Translations of Mathematical Monographs, v. 109. American
  Mathematical Society, Providence, R.I, 1992.

\bibitem{Mor23}
R.~Morandin.
\newblock {\em Modeling and Numerical Treatment of Port-{H}amiltonian
  Descriptor Systems}.
\newblock Dissertation, Technische Universit\"at Berlin, 2023.

\bibitem{PaiW94}
C.~C Paige and M.~Wei.
\newblock History and generality of the {CS} decomposition.
\newblock {\em Linear Algebra Appl.}, 208:303--326, 1994.

\bibitem{PhiRS23}
F.~Philipp, T.~Reis, and M.~Schaller.
\newblock Infinite-dimensional port-{H}amiltonian systems--a system node
  approach.
\newblock {\em arXiv preprint arXiv:2302.05168}, 2023.

\bibitem{Rei21}
T.~Reis.
\newblock Some notes on port-{H}amiltonian systems on {B}anach spaces.
\newblock {\em IFAC-PapersOnLine}, 54(19):223--229, 2021.

\bibitem{ReiS23}
T.~Reis and M.~Schaller.
\newblock Port-{H}amiltonian formulation of {O}seen flows.
\newblock {\em arXiv preprint arXiv:2305.09618}, 2023.

\bibitem{ScZw14}
F.~L. Schwenninger and H.~Zwart.
\newblock Generators with a closure relation.
\newblock {\em Oper. Matrices}, 8(1):157--165, 2014.

\bibitem{Staf05}
O.~Staffans.
\newblock {\em Well-posed Linear Systems}, volume 103 of {\em Encyclopedia of
  Mathematics and its Applications}.
\newblock Cambridge University Press, Cambridge, 2005.

\bibitem{StWe12}
O.~Staffans and G.~Weiss.
\newblock A physically motivated class of scattering passive linear systems.
\newblock {\em SIAM J. Control Optim.}, 50(5):3083 -- 3112, 2012.

\bibitem{Tem77}
R.~Temam.
\newblock {\em Navier-{S}tokes equations. Theory and Numerical Analysis}.
\newblock North Holland, Amsterdam, The Netherlands, 1977.

\bibitem{Sch13}
A.~van~der Schaft.
\newblock Port-{Hamiltonian} differential-algebraic systems.
\newblock In A.~Ilchmann and T.~Reis, editors, {\em Surveys in
  Differential-Algebraic Equations I}, Differential-Algebraic Equations Forum,
  pages 173--226. Springer-Verlag, Berlin, Heidelberg, 2013.

\bibitem{SchJ14}
A.~{van der S}chaft and D.~Jeltsema.
\newblock Port-{Hamiltonian} systems theory: {A}n introductory overview.
\newblock {\em Foundations and Trends in Systems and Control}, 1(2-3):173--378,
  2014.

\bibitem{SchM20}
A.~{van der S}chaft and B.~Maschke.
\newblock Dirac and {L}agrange algebraic constraints in nonlinear
  port-{Hamiltonian} systems.
\newblock {\em Vietnam J. Mathematics}, 48(4):929--939, 2020.

\bibitem{SchM02}
A.~J. {van der Schaft} and B.~M. Maschke.
\newblock {Hamiltonian} formulation of distributed parameter systems with
  boundary energy flow.
\newblock {\em J. Geom. Phys}, 42(1--2):166--174, 2002.

\bibitem{SchM23}
A.J. van~der Schaft and V.~Mehrmann.
\newblock Linear port-{H}amiltonian {DAE} systems revisited.
\newblock {\em Systems Control Lett.}, 177:105564, 2023.

\bibitem{Wlo87}
J.~Wloka.
\newblock {\em Partial Differential Equations}.
\newblock Cambridge University Press, Cambridge UK, 1987.

\bibitem{ZwGM16}
H.~Zwart, Y.~Le~Gorrec, and B.~Maschke.
\newblock Building systems from simple hyperbolic ones.
\newblock {\em Systems Control Lett.}, 91:1--6, 2016.

\bibitem{Zwar89}
H.~J. Zwart.
\newblock {\em Geometric Theory for Infinite-Dimensional Systems}, volume 115
  of {\em Lecture Notes in Control and Information Sciences}.
\newblock Springer-Verlag, Berlin, 1989.

\end{thebibliography}

\section{Appendix on dissipative operators}
\label{sec:app}

Dissipative operators are important in this paper, and so we list some of their properties. We begin with its definition.
\begin{definition}
\label{D:A1}
Let {$\mathbb X$} be a (complex) Hilbert space. Then $\mathcal A: D(\mathcal A) \subset {\mathbb X} { \rightarrow} {\mathbb X}$. $\mathcal A$ is {\em dissipative}\/ if
\begin{equation}
\label{eq:A1}
  \mathrm{Re}\langle { \mathcal A} x,x\rangle \leq 0 \quad \mbox{ for all } x \in D(\mathcal A).
\end{equation}
\end{definition}

The following equivalent characterization is very useful. For a proof we refer to e.g.\ Proposition 6.1.5 of \cite{JacZ12}.
\begin{lemma}
\label{L:A.2}
The operator $\mathcal A: D(\mathcal A) \subset {\mathbb X} \mapsto {\mathbb X}$ is dissipative if and only if
\begin{equation}
\label{eq:A2}
  \| (\lambda I - \mathcal A)x\| \geq \lambda\|x\|, \quad \mbox{ for all } x \in D(\mathcal A), \lambda >0.
\end{equation}
\end{lemma}

For complex $s$ with positive real part, it is easy to see that we have to replace (\ref{eq:A2}) by
\[
  \| (s I - \mathcal A)x\| \geq \mathrm{Re}(s) \|x\|.
\]
From this we see immediately that a dissipative $\mathcal A$ will not have eigenvalues in ${\mathbb C}^+$. Furthermore, when $(sI-\mathcal A)$ is surjective, this inequality implies that $(sI-\mathcal A)$ is boundedly invertible. Secondly, (\ref{eq:A2}) implies that $sI-\mathcal A$ is {\em closable}, and thus $\mathcal A$ is. This means that there exists an extension of $\mathcal A$ which we denote by $\overline{\mathcal A}$ such if $x_n \in D(\mathcal A)$ converged to $x$ and $\mathcal Ax_n$ converge to $y$, then $x \in D(\overline{\mathcal A})$ and $\mathcal Ax=y$. Furthermore, this closure is dissipative, see e.g.\ \cite{ABHN11}.

Based on this consider the following two concepts.
\begin{definition}
\label{D:A3}
Let {$\mathbb X$} be a Hilbert space, and $\mathcal A: D(\mathcal A) \subset {\mathbb X} \mapsto {\mathbb X}$ a dissipative operator.
\begin{enumerate}
\item $\mathcal A$ is {\em m-dissipative} if the range of $\lambda I - \mathcal A={\mathbb X}$ for a $\lambda >0$;
\item ${ \mathcal A}$ is {\em maximally dissipative} if there does not exists an extension of $\mathcal A$ which is also dissipative.
\end{enumerate}
\end{definition}
\begin{lemma}
\label{L:A4}
Let {$\mathbb X$} be a Hilbert space, and $\mathcal A: D(\mathcal A) \subset {\mathbb X} \mapsto {\mathbb X}$ a dissipative operator. Then it { is} m-dissipative if and only if it is maximally dissipative.
\end{lemma}
For the proof we refer to Corollary 2.27 of \cite{Miya92}. Using this lemma we do not distinguish the two concepts, and we have chosen to use the term maximally dissipative when 1.\ or 2.\ holds, see Definition \ref{D:A3}.

The importance of dissipative operators is clear from the {\em Lumer-Phillips Theorem}.
\begin{theorem}
\label{T:A5}
Let {$\mathbb X$} be a Hilbert space, and $\mathcal A: D(\mathcal A) \subset {\mathbb X} \mapsto {\mathbb X}$ a linear operator. Then the following are equivalent:
\begin{enumerate}
\item $\mathcal A$ is maximally dissipative;
\item $\mathcal A$ is the infinitesimal generator of a contraction semigroup on {$\mathbb X$};
\item $\mathcal A$ is closed and densely defined, and $\mathcal A$ and $\mathcal A^*$ are dissipative.
\end{enumerate}
\end{theorem}
For the proof of (1) $\Leftrightarrow$ (2) we refer to \cite[Theorem 6.1.7]{JacZ12}, and for (2) $\Leftrightarrow$ (3) to \cite[Corollary 2.3.3]{CuZw20}.

We end this appendix with a lemma.
\begin{lemma}
\label{L:A6}
  If $\mathcal A: D(\mathcal A) \subset {\mathbb X} \mapsto {\mathbb X}$ a dissipative operator which is boundedly invertible, then it is maximally dissipative.
\end{lemma}
\proof
The proof follows from the fact that the resolvent set of an operator is always open. Thus there exists a $\lambda >0$ such that $\lambda I - \mathcal A$ is boundedly invertible, and in particular its range equals {$\mathbb X$}.
\hfill\eproof

\end{document}